\documentclass[a4paper,11pt]{article}
\usepackage{pdfsync}

 \usepackage[plainpages=false]{hyperref}
 \usepackage{amsfonts,amsmath,amssymb,amsthm,enumerate}
 \usepackage{latexsym,lscape,rawfonts,mathrsfs,mathtools}
\usepackage{enumitem}
\usepackage{commath}
\usepackage{caption}

 \usepackage[all]{xy}
 \usepackage{eufrak}
 \usepackage{makeidx}
 \usepackage{graphicx,psfrag}
 \usepackage{pstool}
 \usepackage{tikz-cd}
 \usepackage{array,tabularx,tabu}

 \usepackage{setspace}

 \usepackage[titletoc,title]{appendix}

\usepackage{txfonts}

 \newcommand{\C}{\ensuremath{\mathbb{C}}}

 \newcommand{\R}{\ensuremath{\mathbb{R}}}
  \newcommand{\T}{\ensuremath{\mathbb{T}}}
 \newcommand{\Z}{\ensuremath{\mathbb{Z}}}
 
 \newcommand{\RP}{\ensuremath{\mathbb{RP}}}

 \newcommand{\ba}{\begin{align*}}
 \newcommand{\ea}{\end{align*}}
 \newcommand{\na}{\nabla}
\newcommand{\la}{\langle}
\newcommand{\ra}{\rangle}
\newcommand{\lc}{\left(}
\newcommand{\rc}{\right)}

\newcommand{\ep}{\epsilon}

\newcommand{\hy}{hyperk\"ahler\,}

\newcommand{\lam}{\Lambda}
\newcommand{\inj}{\text{inj}}
\newcommand{\mini}{\widetilde{\C^2/\Gamma}}
\newcommand{\mcm}{\widetilde{\C^2/\Z_m}}
\newcommand{\mdm}{\widetilde{\C^2/D_{4m}}}
\newcommand{\ex}{\text{exp}}
\newcommand{\floor}[1]{\left\lfloor #1 \right\rfloor}

 \makeatletter
 \def\ExtendSymbol#1#2#3#4#5{\ext@arrow 0099{\arrowfill@#1#2#3}{#4}{#5}}
 
 \makeatother

 \makeatletter
 \def\ExtendSymbol#1#2#3#4#5{\ext@arrow 0099{\arrowfill@#1#2#3}{#4}{#5}}
 \newcommand\longright[2][]{\ExtendSymbol{-}{-}{\rightarrow}{#1}{#2}}
 \makeatother

\def\XXint#1#2#3{{\setbox0=\hbox{$#1{#2#3}{\int}$ }
\vcenter{\hbox{$#2#3$ }}\kern-.55\wd0}}

\numberwithin{equation}{section}

\newtheorem{thm}{Theorem}[section]
\newtheorem{cor}[thm]{Corollary}
\newtheorem{prop}[thm]{Proposition}
\newtheorem{lem}[thm]{Lemma}

\newtheorem{quest}[thm]{Question}
\newtheorem{rem}[thm]{Remark}

\newtheorem{defn}[thm]{Definition}

\title{On the geometry of asymptotically flat manifolds}
\author{Xiuxiong Chen\footnote{Partially supported by NSF grant DMS-1515795. }, \,Yu Li  \footnote{Partially supported by research fund from SUNY Stony Brook.} }
\date{\today}

\begin{document}
\maketitle

\begin{abstract}
In this paper, we investigate the geometry of asymptotically flat manifolds with controlled holonomy. We show that any end of such manifold admits a torus fibration over an ALE end. In addition, we prove a Hitchin-Thorpe inequality for oriented Ricci-flat $4$-manifolds with curvature decay and controlled holonomy. As an application, we show that any complete asymptotically flat Ricci-flat metric on a $4$-manifold which is homeomorphic to $\R^4$ must be isometric to the Euclidean or the Taub-NUT metric, provided that the tangent cone at infinity is not $\R \times \R_+$.
\end{abstract}

\tableofcontents

\section{Introduction}
An important question in Riemannian geometry is how the geometry and topology at infinity of a noncompact complete Riemannian manifold are controlled by its curvature. Let $(M^n,g)$ be a noncompact complete Riemannian manifold, then a natural scale-invariant condition is quadratic curvature decay:
  \begin{align}\label{eq:001}
|Rm|(x) =O(r^{-2}),
\end{align}
where $r=d(p,x)$ for a fixed point $p$. However, this condition imposes no topological restriction on the underlying manifold. In fact, Gromov observed that any noncompact manifold carries a complete metric with quadratic curvature decay, see \cite[Lemma $2.1$]{LS00}.

For this reason, we consider an asymptotic flatness condition which is stronger than \eqref{eq:001}. More precisely, a complete Riemannian manifold $(M,g)$ is \emph{asymptotically flat} (AF) if
\begin{align}
|Rm|(x) \le \frac{K(r)}{r^2}, \tag{AF} \label{cond:AF}
\end{align}
where $\{K(s),s\ge 0\}$ is a nonincreasing positive function such that
\begin{align} \label{eq:003}
\int_1^\infty \frac{K(s)}{s}\,ds<\infty.
\end{align}
The integral condition \eqref{eq:003} is added for two reasons. First, it is necessary to obtain some topological restriction. Indeed, it was proved by Abresch \cite[Theorem C]{Ab85} that given a positive function $K(s)$ such that the integral in \eqref{eq:003} diverges, there exists a complete metric on any noncompact surface so that condition \eqref{cond:AF} is satisfied. Secondly, as shown by Kasue \cite{Ka88}, condition \eqref{eq:003} guarantees the existence of a unique tangent cone at infinity $C(S(\infty))$, whose definition can be found in Section \ref{sub:tan}. 

For any asymptotically flat manifold $(M^n,g)$, Abresch \cite[Theorem B]{Ab85} proved that there are finitely many ends. Moreover, it was proved by Petrunin and Tuschmann \cite[Theorem A]{PT01} that each end is homeomorphic to $X \times \R_+$ for a closed manifold $X$. In particular, it implies that $M$ is homotopy equivalent to a compact manifold with boundary $X$. Therefore, to understand the asymptotic geometry of $M$, it is essential to study the geometry and topology of the boundary $X$. For simplicity, we assume that all manifolds considered in this paper have only one end unless otherwise stated.

With condition \eqref{cond:AF}, the volume growth of $(M,g)$ is at most Euclidean by the Bishop-Gromov volume comparison theorem (see Lemma \ref{L:vola}). If $(M,g)$ has Euclidean volume growth, Bando, Kasue and Nakajima \cite{BKN89} proved that $M$ is an ALE (asymptotically locally Euclidean) manifold, i.e., there exist a compact set $K$ in $M$, a ball $B$ in $\R^n$, a finite subgroup $\Gamma \subset O(n)$ acting freely on $S^{n-1}$ and a diffeomorphism $\Phi: M\backslash K \to (\R^n\backslash B)/\Gamma$ so that under the identification, the metric is almost Euclidean in the weighted $C^{1,\alpha}$ sense. Their paper \cite{BKN89} focuses on the case $K(r)=r^{-\ep}$, but it is not hard to generalize the result to any $K(r)$ with \eqref{eq:003} by using the same method, see Appendix \ref{app:A} for details. Therefore, the asymptotic geometry of an ALE manifold is well understood. In particular, $C(S(\infty))$ is isometric to a flat cone $\R^n/\Gamma$ and the boundary $X$ is diffeomorphic to the spherical space form $S^{n-1}/\Gamma$ (see also Remark \ref{rem:A01}). 

A natural question is can one obtain a similar result if the volume growth is not maximal? In this case, if we choose any sequence $r_i \to \infty$, the blow-down sequence $(M,r_i^{-2}g,p)$ collapses. Since $|Rm_{r_i^{-2}g}|$ are uniformly bounded away from the vertex, the collapsing theory of Cheeger-Fukaya-Gromov \cite{CFG92} applies. Indeed, it implies that the end of $M$ carries a nilpotent Killing structure (see details in \cite{CFG92}).

To describe the structure of infinity more precisely, we consider in this paper additional holonomy control conditions.

\textbf{Holonomy control}: There exist constants $\kappa \in (0,1/2)$ and $\Theta_H \in (0,\pi/2)$ such that 
\begin{align} 
\|\mathbf r (\gamma_x)\| \le \Theta_H \tag{HC}    \label{cond:HC}
\end{align}
for any point $x$ outside a compact set and any geodesic loop $\gamma_x$ based at $x$ with length smaller than $\kappa r$. Here $\mathbf r(\gamma_x)$ is the parallel transport around $\gamma_x$ and the norm $\|\cdot\|$ denotes its maximal rotational angle. For precise definitions, see Definition \ref{def:rigid}.

Notice that \eqref{cond:HC} is equivalent to the following condition (Lemma \ref{lem:holaa1}):

There exist constants $\kappa \in (0,1/2)$ and $C_H >0$ such that
\begin{align} 
\|\mathbf r (\gamma_x)\| \le C_H\frac{L(\gamma_x)}{r} \tag{$\mathbb{HC}$}    \label{cond:HC'}
\end{align}
for any point $x$ outside a compact set and any geodesic loop $\gamma_x$ based at $x$ with length smaller than $\kappa r$. Here $L$ denotes the length of the geodesic loop.

With the extra condition \eqref{cond:HC}, a direct implication is that $C(S(\infty))$ is smooth away from the vertex, see Theorem \ref{T:cone}. This follows directly from Fukaya's structure theory of the limit space \cite{Fu88}, by using the local group. Therefore, we obtain the following theorem (Theorem \ref{T:infra}, Theorem \ref{T:nil}), which is a direct application of \cite{Fu87}.

\begin{thm}\label{T:001}
Let $(M^n,g)$ be a complete Riemannian manifold with \eqref{cond:AF} and \eqref{cond:HC}. Then there exist a compact set $K \subset M$, a constant $A>0$ and a fibration $f_0: M^{n} \backslash K \longrightarrow C(S(\infty))\backslash \bar B(p_{\infty},A)$ with fiber $F$ satisfying the following properties.
\begin{enumerate} [label=(\roman*)]
\item $F$ is a nilmanifold with diameter bounded by $o(r)$.

\item $f_0$ is an $o(1)$-almost-Riemannian submersion.

\item The second fundamental form of $F$ is bounded by $O(r^{-1})$.
\end{enumerate}
\end{thm}

We briefly discuss the proof of Theorem \ref{T:001}. It follows from Fukaya's fibration theorem \cite{Fu87} that there exists a fibration from the annulus of $M$ to the corresponding annulus in $C(S(\infty))$ with above mentioned properties. Then it follows from a standard argument of \cite{CFG92} that we can make all these local fibrations compatible to form a global fibration. In particular, Theorem \ref{T:001} implies that the boundary $X$ is the total space of the fibration
\begin{align*}
f_0: X \longrightarrow S(\infty)
\end{align*}
with fiber $F$ a nilmanifold.

Although Theorem \ref{T:001} imposes some restrictions on the end, it does not provide much information in the sense of geometric analysis. For instance, it is not clear how the properties of the fibration $f$ depend on the curvature decay. In particular, there is no estimate for the diameter change of fiber $F$. It is possible that the diameter of $F$ may grow like $\sqrt r$, which makes the computation at infinity difficult. It turns out that even for flat manifolds this complicated phenomenon can happen, see Section \ref{Ex}. To further control the geometry at infinity, we consider the following stronger condition.

\textbf{Strong holonomy control}: There exist a constant $\kappa \in (0,1/2)$ and a positive function $\ep(r)$ with $\ep(r) \to 0$ if $r \to \infty$ such that 
\begin{align} 
\|\mathbf r (\gamma_x)\| \le \ep(r) \tag{SHC}    \label{cond:SHC}
\end{align}
for any point $x$ outside a compact set and any geodesic loop $\gamma_x$ based at $x$ with length smaller than $\kappa r$.

Now, we state the main theorem of this paper. Here an ALE end refers to an end of an ALE manifold.

\begin{thm}\label{T:002}
Let $(M^n,g)$ be a complete Riemannian manifold with \eqref{cond:AF} and \eqref{cond:SHC}. Then there exist an integer $0\le m \le n-1$, a flat torus $\T_{\infty}^m$, a compact set $K \subset M^n$ such that $M^n \backslash K$ is endowed with a $m$-dimensional torus fibration $f$ over an ALE end $Y$. Moreover, there exists an open cover $\Omega_i$ of $M^n\backslash K$ satisfying the following properties.

\begin{enumerate}[label=(\roman*)]
\item There exists a bundle diffeomorphim $T_i: \Omega_i \to U_i \times \T_{\infty}^m$ where $U_i \subset \R^{n-m}$ and $T_i$ satisfies (a) $T_i$ is an $O(K(r/2))$-almost-isometry and (b) $|\na^2 T_i| = O(r^{-1}K(r/2))$.

\item There exists a $\T^m$-action $\mu_i$ on $\Omega_i$ which is almost isometric in the sense that for any $a \in \T^m$, $(\mu_i(a))^*g=g+O(r^{-1}K(r/2))$. Moreover, on $\Omega_i \cap \Omega_j \ne \emptyset$, $\mu_i$ and $\mu_j$ differ by an automorphism of the torus fiber.

\item There exists a metric $\bar g$ on $M^n \backslash K$ such that (a) $\bar g$ is invariant under the action of $\mu_i$; (b) $g=\bar g+O(r^{-1}K(r/2))$; (c) $|\na \bar g|=O(r^{-1}K(r/2))$ and (d) the curvature of $\bar g$ is controlled by $O(r^{-2}K(r/2))$.

\item The structure group of $f$ is contained in $\mathbb T^m \rtimes G_{\infty}$ for some finite group $G_{\infty} \subset \emph{GL}(m,\Z)$.
\end{enumerate}
\end{thm}

We explain the statements of Theorem \ref{T:002}. First, it implies that the fiber of $f$ is a $m$-dimensional torus besides being a nilmanifold. Secondly, it gives a quantitative description of $f$ by a family of local charts $T_i: \Omega_i \to U_i \times \T_{\infty}^m$. In fact, it means that the local geometry is close to the flat piece $U_i \times \T_{\infty}^m$ and the error can be explicitly controlled by the curvature. In particular, we derive that all fibers of $f$ are converging (in the $C^1$ sense) to a flat torus $\T_{\infty}^m$ and hence the diameter is approaching to a constant. Moreover, it is clear that $f$ is an $O(K(r/2))$-almost-Riemannian submersion and the second fundamental form of the fiber is bounded by $O(r^{-1}K(r/2))$. Notice that all those properties are much sharper than their counterparts in Theorem \ref{T:001}. Thirdly, there exists a natural torus action $\mu_i$ on $\Omega_i$ which is almost isometric. All those torus actions can be made compatible in the sense that they differ by an automorphism of $\T^m$. Therefore, by averaging the metric $g$ under the torus action, we obtain the nearby invariant metric $\bar g$ such that all torus actions act isometrically. The base $Y$ can be regarded as the orbit space which is equipped with a metric so that the fibration $f$ is a Riemannian submersion (with respect to $\bar g$). Moreover, it can be shown that $Y$ is an ALE end which has the same curvature decay as $M$ (Proposition \ref{P:ALE}). Finally, the structure group of $f$ is reduced to $\T^m \rtimes G_{\infty}$ such that the finite group $G_{\infty}$ depends only on the torus $\T_{\infty}^m$, see Definition \ref{def:holin}. 

If we further assume the following condition for all higher covariant derivatives of the curvature:
\begin{align} 
|\na^k Rm|(x) =O(r^{-2-k} K(r)),\quad \forall k\ge 1 \tag{HOAF}, \label{cond:HOAF}
\end{align}
then we obtain the estimates for the higher derivatives of the fibration $f$ (Theorem \ref{T:met}). We remark that condition \eqref{cond:SHC} in Theorem \ref{T:002} can be replaced by \eqref{cond:SHC} on a fixed geodesic ray plus \eqref{cond:HC'} (Remark \ref{rem:weak}). Finally, if the fiber of $f_0$ obtained in Theorem \ref{T:001} is a circle ( i.e. $C(S(\infty))$ is $n-1$ dimensional), then the conclusions of Theorem \ref{T:002} also hold (Theorem \ref{T:circle}), which has improved \cite[Theorem $3.26$]{Mi10} of Minerbe.

Theorem \ref{T:002} yields much topological information. A direct corollary (Corollary \ref{cor:cone}) is that $C(S(\infty))$ is a flat cone $\R^{n-m}/\Gamma$. Therefore, it implies that the boundary $X$ is the total space of the fibration
\begin{align*}
f: X \longrightarrow S^{n-m-1}/\Gamma
\end{align*}
with fiber $\T^m$ such that its structure group is contained in $\mathbb T^m \rtimes G_{\infty}$. Here $\Gamma \subset O(n-m)$ is a finite subgroup acting freely on $S^{n-m-1}$ if $n-m \ge 3$ and $S^{n-m-1}/\Gamma$ is a circle if $n-m=2$ and a point if $n-m=1$. In the particular case $n=4$, the fibration $f$ and hence the total space $X$ can be completely classified, see Section \ref{sub:topo}.

We proceed to discuss the proof of Theorem \ref{T:002}. For any point $q$ far away from the base point, we consider the fundamental pseudo-group $\Gamma_q=\Gamma(q,\kappa r)$ (Definition \ref{def:group01}), which acts isometrically on the ball $\hat B(0,\kappa r) \subset T_qM$. The elements of $\Gamma_q$ correspond one-to-one to the short geodesic loops at $q$ (Lemma \ref{L:group1}) and the group action is close to the corresponding rigid motion (Lemma \ref{L:es1}). Under condition \eqref{cond:SHC}, all elements in $\Gamma_q$ are almost translational. Therefore, by a standard process from \cite[Chapter $4$]{BK81}, we can choose a short basis $\{c_1^q,\cdots,c_m^q\} \subset \Gamma_q$ (Definition \ref{def:sba}).

The key point is that the construction of the short basis in $\Gamma_q$ can be made continuous on a fixed geodesic ray. To achieve this, we consider the sliding (see Definition \ref{def:sliding}) of a geodesic loop along a curve, which is a natural way to transport a geodesic loop continuously to other points. We fix a geodesic ray $\{\alpha(t):t\ge 0\}$, then once a geodesic loop $c$ is chosen at $\alpha(t_0)$, its sliding $c(t)$ at $\alpha(t)$ is defined. With conditions \eqref{cond:AF} and \eqref{cond:SHC}, one can obtain the uniform estimates of the length and rotational part of $c(t)$ (i.e. Theorem \ref{T:esray}). In addition, given two geodesic loops $c_1$ and $c_2$, the angle between $c_1(t)$ and $c_2(t)$ is also well controlled (Proposition \ref{p:angle}). By using these estimates, there exists a large $t_0$ and short basis $\{c_1,\cdots,c_m\}$ of $\Gamma_{\alpha(t_0)}$ such that its sliding $\{c_1(t),\cdots,c_m(t)\}$ is also a short basis of $\Gamma_{\alpha(t)}$. Moreover, the length of $c_i(t)$ and their mutual angle will converge to some constants at a given rate (Theorem \ref{T:basis2}). Furthermore, we show the fundamental pseudo-group is abelian (Proposition \ref{P:abelian}). After we construct a short basis along the geodesic ray, we can extend the basis by sliding to all points on the end of $M$. Based on the short basis constructed, we then improve condition \eqref{cond:SHC} and the almost translational property of $\Gamma_q$, see Lemma \ref{L:tr}. 

Next, we construct the local torus fibration, which is the smoothing of the projection map, see Theorem \ref{T:localfiber}. Then we continue to construct a bundle diffeomorphism $T_q: \Omega_q \to U_q \times \T_{\infty}^m$, which has the above mentioned properties. From the bundle diffeomorphism, it is easy to develop the torus action which is simply the torus translation in $\T^m$. From our construction, any two local fibrations are close after a transition map (Proposition \ref{P:trans2}). Therefore, we can modify one fibration so that they are compatible and the corresponding torus actions differ by an automorphism (Proposition \ref{P:trans3}). By using a standard strategy from \cite{CFG92}, we can make all local fibrations compatible and thus obtain a global fibration on the end (Theorem \ref{T:met}). At the same time, the base $Y$ is formed by attaching all local bases and we show that $Y$ is an ALE end (Proposition \ref{P:ALE}). Finally, all other statements in Theorem \ref{T:002} can be proved by using the properties of the local fibrations.

From Theorem \ref{T:002}, we have the following natural definition, which is a direct generalization of ALE manifolds for which the torus bundle is trivial.

\begin{defn}
A noncompact complete Riemannian manifold $(M^n,g)$ is called a TALE manifold if it satisfies \eqref{cond:AF} and \eqref{cond:SHC}.
\end{defn}

If a TALE manifold $(M^n,g)$ is Ricci-flat, then we can improve the decay order of the curvature.
\begin{thm}\label{T:003}
Let $(M^{n},g)$ be a Ricci-flat TALE manifold with $l$-dimensional tangent cone at infinity. 
\begin{enumerate}[label=(\roman*)]
\item If $l\ge 4$ or $l=3$ and $n=4$, then
\begin{align*} 
|Rm|=O(r^{-\frac{(l-2)(n-1)}{n-3}}).
\end{align*} 

\item If $l=1$, then there exists a constant $\delta>0$ which depends only on $\T_{\infty}^{n-1}$ such that
\begin{align*} 
|Rm| =O(e^{-\delta r}).
\end{align*} 
\end{enumerate}
\end{thm}

Theorem \ref{T:003} (i) follows essentially from Minerbe's work \cite[Theorem $4.12$]{Mi092} since in our setting the volume increases like $r^l$. To prove Theorem \ref{T:003} (ii), we need to use the asymptotic geometry of a TALE manifold, see Section \ref{S:RF1}.

Our next result is a generalization of the Hitchin-Thorpe inequality on oriented Ricci-flat TALE $4$-manifolds. By convention, we will call such manifolds of type ALE, ALF, ALG or ALH, if the dimension of the tangent cone at infinity is $4,3,2$ or $1$, respectively.

\begin{thm}[Hitchin-Thorpe inequality]\label{T:004}
Let $(M^4,g)$ be an oriented Ricci-flat TALE $4$-manifold. Then
\begin{align*}
2( \chi(M)-\lambda) \ge 3|\tau(M)+\eta| 
\end{align*} 
with equality if and only if $(M,g)$ or its opposite orientation space is a quotient of a \hy $4$-manifold. Here the constant $\lambda=1/|\Gamma|$ for ALE manifolds and $\lambda=0$ otherwise and $\eta$ depends only on the topology of the asymptotic torus fibration. More precisely,
\begin{enumerate}[label=(\roman*)]
\item  (ALE): $\eta=\eta(S^3/\Gamma)$, where $\eta(S^3/\Gamma)$ is the eta invariant of the space form $S^3/\Gamma$.

\item  (ALF): $\eta=-\frac{e}{3}+\text{sgn}\,e$ for the cyclic type and $\eta=-\frac{e}{3}$ for the dihedral type, where $e$ is the Euler number of the asymptotic circle fibration.

\item  (ALG): $\eta=0,0,-\frac{2}{3},-1$ or $-\frac{4}{3}$ if the monodromy of the asymptotic $\T^2$-fibration is $1,\Z_2,\Z_3,\Z_4$ or $\Z_6$, respectively.

\item  (ALH): $\eta=0$.
\end{enumerate}
\end{thm}

The Hitchin-Thorpe inequality for ALE $4$-manifolds was proved by Nakajima (\cite[Theorem $4.2$]{Na90}). We prove that similar inequalities also hold for ALF, ALG and ALH cases (see Theorem \ref{T:HTALE1}, Theorem \ref{T:HTALE2} and Theorem \ref{T:HTALE3}).

It follows from Kronheimer's list \cite{Kro89a, Kro89b} that any ALE \hy $4$-manifold is diffeomorphic to the minimal resolution $\mini$ of $\C^2/\Gamma$ for $\Gamma \subset \text{SU}(2)$ a finite group. A direct application of Theorem \ref{T:004}, see Corollary \ref{C:HTALE}, is that any Ricci-flat ALE $4$-manifold which is homeomorphic to $\mini$ must be \hy. In particular, the underlying manifold is diffeomorphic to $\mini$. Similarly, one can prove (Corollary \ref{C:HTALE1}) that any Ricci-flat ALF $4$-manifold which is homeomorphic to $\mcm$ must be isometric to a Multi-Taub-NUT metric if $m \ne 2$. On the other hand, if the underlying manifold is homeomorphic to $\mdm$, then the metric is isometric to the CHIKLR metric (Corollary \ref{C:HTALE2a}). Notice that the similar results also hold for ALG and ALH cases (Corollary \ref{C:HTALE2}, Corollary \ref{C:HTALE3}).

If $(M^4,g)$ is an asymptotically flat $4$-manifold such that it is simply-connected at infinity, it was proved by Petrunin and Tuschmann \cite[Theorem A (ii)]{PT01} that $C(S(\infty))$ is isometric to $\R^4$, $\R^3$ or $\R \times \R_+$. We remark that it was conjectured in \cite{PT01} that the last case cannot happen. If $C(S(\infty))=\R^4$, then $(M,g)$ is an ALE manifold. If $C(S(\infty))=\R^3$, then it follows from circle case of Theorem \ref{T:002} (Theorem \ref{T:circle}) that $(M,g)$ is an ALF manifold. As an application of the Hitchin-Thorpe inequality, we prove the following theorem. Notice that there are infinitely many non-diffeomorphic exotic differential structures on $\R^4$, see \cite{Kirby89}.

\begin{thm}\label{T:005}
Let $(M^4,g)$ be a complete Ricci-flat Riemannian manifold with \eqref{cond:AF} such that $M$ is homeomorphic to $\R^4$. Suppose the tangent cone at infinity is not $\R \times \R_+$, then $g$ is isometric to either the flat or the Taub-NUT metric. In particular, $M$ is diffeomorphic to $\R^4$. 
\end{thm}

{\bf Organization of the paper}:
In Section $2$, we discuss some examples of asymptotically flat manifolds to give motivations for conditions \eqref{cond:HC'} and \eqref{cond:SHC}. In Section $3$, we review some basic concepts including fundamental pseudo-group, sliding, tangent cone at infinity, etc. and derive some estimates of geodesic loops which will be used throughout this paper. In addition, we prove Theorem \ref{T:001}. In Section $4$, we describe a process to choose a short basis for the fundamental pseudo-group and construct at each point a short basis. In Section $5$, we prove Theorem \ref{T:002} by first constructing the local fibrations and then modifying them to obtain a global fibration. We also discuss some topological implications. In Section $6$, we prove Theorem \ref{T:003}, Theorem \ref{T:004} and Theorem \ref{T:005}. In the last section, we propose some further questions.\\

{\bf Acknowledgements}:
Both authors are grateful to Xiaochun Rong for helpful discussions about collapsing theory. Yu Li would like to thank Olivier Biquard for answering his question on his paper \cite{BiMi11} and Xianzhe Dai for discussing his paper \cite{DaiWei07}. Yu Li would also like to thank Gao Chen, Shaosai Huang and Ruobing Zhang for many useful conversations. Last but not least, the authors would like to thank the anonymous referees for several valuable comments that help improve the exposition of the paper.

\section{Examples}
\label{Ex}
In this section, we discuss some examples of asymptotically flat Riemannian manifolds.

\subsection{Flat manifolds}

\noindent (i) $M=\R^n \times X$ where $X$ is a flat closed manifold. 

It follows from the Bieberbach theorem that $X$ is finitely covered by a flat torus. If $X$ is not a flat torus, then there exists a geodesic loop such that the rotational part of the holonomy around the loop is not identity. Therefore, $M$ satisfies \eqref{cond:HC'} (and hence \eqref{cond:SHC}) if and only if $X$ is a torus. Moreover, it is easy to see the tangent cone at infinity is $\R^n$.

\noindent (ii) (Gromov \cite[$8.9$]{Gr06}) $M$ is a quotient of $\R^3$ by a cyclic group defined by
 \begin{align*}
\tau((z,t))=(e^{2\pi \theta i}z,t+1) \quad \text{for} \quad (z,t) \in \C \times \R=\R^3,
\end{align*}
where $\tau$ is the generator of the group.

It is clear that the group action is free and $M$ is diffeomorphic to $\R^2 \times S^1$. For any $x=(z,t) \in M$, a geodesic loop $\gamma$ based at $x$ can be represented by a segment connecting $x$ and $\tau^k x$. A direct calculation shows that the length
 \begin{align}\label{eq:E201}
L(\gamma)=\sqrt{k^2+4r^2\sin^2(\pi k\theta)}
\end{align}
where $r=|z|$. On the other hand, if we set $w\equiv2\pi k \theta\, (\text{mod}\,2\pi)$ for $ 0\le w <2\pi$, then the rotational part $\|\mathbf{r}(\gamma)\|= \min\{w,2\pi-w\}$. From \eqref{eq:E201} it is easy to see
 \begin{align}\label{eq:E202}
\|\mathbf{r}(\gamma)\| \le \frac{\pi}{4r}L(\gamma).
\end{align}
In other words, $M$ satisfies \eqref{cond:HC'}. If $\theta=\frac{p}{q}$ for coprime integers $p$ and $q$, then it is easy to see that the tangent cone at infinity is the flat cone $\R^2/\Z_q$. If $\theta$ is irrational, we have

\begin{prop}\label{P201}
If $\theta$ is irrational, then the tangent cone at infinity of $M$ is $\R_+$.
\end{prop}

\begin{proof}
From the pigeonhole principle, there exists a constant $C$ such that for any $r \ge 1$ and $w \in [0,1)$ we can find an integer $k \in [0,C\sqrt r]$ satisfying 
 \begin{align}\label{eq:E202}
\{w-k\theta \} \le \frac{1}{\sqrt{r}}.
\end{align}
If we set $x=(r,0)$ and $y=(re^{2\pi wi},0)$, then
 \begin{align}\label{eq:E203}
d(x,y) \le |\tau^k x-y|=\sqrt{k^2+r^2|e^{2\pi(w-k\theta)i}-1|^2} \le C\sqrt{r},
\end{align}
where we have used \eqref{eq:E202}. From \eqref{eq:E203}, it is easy to see for any sequence $r_i \to \infty$, $(M,r_i^{-2}g)$ converges to $\R_+$.
\end{proof}

Since $M$ satisfies \eqref{cond:HC'}, it follows from Theorem \ref{T:infra} and Theorem \ref{T:nil} that there exists a torus bundle at infinity. However, it follows from \cite[Proposition $1.1$]{Mi10} that if $\theta$ is irrational, the injectivity radius is unbounded. Thus, the diameter of fiber $\T^2$ is unbounded.

On the other hand, if $\theta$ is rational, then we obtain a circle bundle at infinity. In the case, the length of the circle fiber converges to a nonzero constant, see Theorem \ref{T:circle}.

\begin{rem}\label{rem:ex1}
The same example is discussed in Minerbe's paper \emph{\cite{Mi10}}, where he shows that if $\theta$ is an irrational algebraic number, then $r^{\alpha} \lesssim \emph{inj}(x)\lesssim r^{1/2}$ for any $\alpha \in (0,1/2)$.
\end{rem}

\subsection{Gravitational instantons}
Recall that a gravitational instanton is a noncompact complete \hy $4$-manifold $(M^4,g)$ with some curvature decay condition. It follows from \cite{CC15} that any gravitational instanton with faster-than-quadratic curvature decay (i.e. \eqref{cond:AF} with $K(r)=r^{-\ep}$) falls into one of the categories: ALE, ALF, ALG and ALH, where the volume growths are $r^4,r^3,r^2$ and $r$, respectively. We briefly discuss the classification of gravitational instantons with faster-than-quadratic curvature decay.

In the ALE case, it was proved by Kronheimer \cite{Kro89a, Kro89b} that any ALE gravitational instantons are diffeomorphic to the minimal resolution $\mini$ of the flat cone $\C^2/\Gamma$, where $\Gamma \subset \text{SU}(2)$ is a finite group. 

All ALF gravitational instantons can be divided into cyclic type (ALF-$A_k$) and dihedral type (ALF-$D_k$), where the tangent cones infinity are $\R^3$ and $\R^3/\Z_2$, respectively. It was proved by Minerbe \cite{Mi11} that any ALF-$A_k$ gravitational instanton must be isometric to $\R^3 \times S^1$ (ALF-$A_{-1}$) or a Multi-Taub-NUT metric for $k \ge 0$, whose underlying complex manifold is biholomorphic to $\widetilde{\C^2/\Z_{k+1}}$. For the explicit definition of Multi-Taub-NUT metrics, see, e.g., \cite[Section $1$]{Mi11}.

In the ALF-$D_k$ case, it was proved by Biquad and Minerbe \cite{BiMi11} that $k\ge 0$. The ALF-$D_0$ gravitational instanton was constructed by Atiyah and Hitchin \cite{AH88} and ALF-$D_1$ is its double cover. The ALF-$D_2$ gravitational instanton is called the Page-Hitchin metric \cite{Hi84,Pa81}. For $k>2$, the ALF-$D_k$ gravitational instanton was constructed in \cite{CH05,CK99,Dan93,BiMi11}. It was proved recently by Chen and Chen that those are only possible examples (see \cite{CC17} where more information and references can be found).

In the ALG case, a lot of examples were constructed by Hein in \cite{Hei12} and it can be shown (\cite[Theorem $1.4$]{CC16}) that any ALG gravitational instanton must be obtained by the modified Hein's construction.

In the ALH case, it can be proved that any ALH gravitational instanton must be diffeomorphic to the minimal resolution of $\R \times \T^3/\pm$ and can be completely classified, see \cite[Theorem $1.5$]{CC16}.

In Appendix \ref{app2}, we prove that any gravitational instanton with \eqref{cond:AF} is a TALE manifold. Combined with Theorem \ref{T:003}, the same classification results hold for ALE, ALF and ALH cases.

\subsection{Euclidean Schwarzschild metric}

The $n$-dimensional ($n \ge 4$) Euclidean Schwarzschild metric (see \cite{GH78}) is defined as
 \begin{align*}
g_n=\lc1-\frac{2m}{r^{n-3}}\rc d\theta^2+\lc1-\frac{2m}{r^{n-3}}\rc^{-1}dr^2+r^2ds_{n-2}^2
\end{align*}
on $[0,L_{\infty})\times (0,\infty) \times S^{n-2} =\R^2\backslash \{0\} \times S^{n-2}$, where $ds^2_{n-2}$ is the standard metric of $S^{n-2}$ and $L_{\infty}$ is the period of the parameter $\theta$. 

By a direct calculation, $g_n$ is a Ricci-flat metric with $r^{-(n-1)}$ curvature decay and asymptotic to the flat $\R^{n-1} \times S^1$. It is easy to see that $g_n$ can be extended to a complete metric on $\R^2 \times S^{n-2}$ if and only if $
m=\frac{1}{2}\lc\frac{L_{\infty}(n-3)}{4\pi} \rc^{n-3}.$

In particular, $g_3$ has $r^{-3}$ curvature decay and is asymptotic to the flat space $\R^3 \times S^1$. Notice that $g_3$ is a non-K\"ahler, Ricci-flat, ALF metric.

\section{Preliminary results}

We first prove that the volume growth of an asymptotically flat manifold is at most Euclidean.

\begin{lem} \label{L:vola}
Let $(M^{n},g)$ be a noncompact complete Riemannian manifold with \eqref{cond:AF}. Then
\begin{align*}
\limsup_{r \to \infty} \frac{|B(p,r)|}{r^{n}} <\infty.
\end{align*}
\end{lem}

\begin{proof}
We consider the following Jacobi equation, 
\begin{align} \label{eq:ja1}
J''(t)=C_0\frac{K(t)}{(1+t)^{2}}J(t), \quad J(0)=0\quad\text{and}\quad J'(0)=1,
\end{align}
where $C_0$ is chosen so that 
\begin{align*}
|Rm|(x) \le C_0\frac{K(r)}{(1+r)^{2}}.
\end{align*}

It follows from \eqref{eq:ja1} that $J'$ is increasing, so
\begin{align} \label{eq:v001}
t\le J(t) \le J'(t)t.
\end{align}
Therefore,
\begin{align*}
J''(t)= C_0\frac{K(t)}{(1+t)^{2}}J(t) \le  C_0\frac{tK(t)}{(1+t)^{2}}J'(t)
\end{align*}
and hence
\begin{align}\label{eq:v002}
\log J'(t) \le  C_0\int_{0}^{\infty} \frac{tK(t)}{(1+t)^{2}}\,dt  <\infty,
\end{align}
where the last inequality follows from \eqref{eq:003}.

From \eqref{eq:v001} and \eqref{eq:v002}, there exists a constant $C_1>0$ such that
\begin{align*}
t \le J(t) \le C_1t.
\end{align*}
From Bishop-Gromov comparison theorem (\cite[Theorem $2.1$]{Li12}), we have
\begin{align*}
\frac{|B(p,r)|}{\int_0^r J^{n-1}(t) \,dt}
\end{align*}
is nonincreasing and hence the proof is complete.
\end{proof}

\subsection{Fundamental pseudo-group and sliding}

Let $M$ be a noncompact complete Riemannian manifold and $q$ is a fixed point on $M$. Throughout this section, we assume that on $B(q,100\rho)$,
\begin{align} \label{eq:curv1}
|Rm| \le \lam^2
\end{align}
for $\lam>0$ and $\rho \ge 100$ such that
\begin{align} \label{eq:curv2}
\lam \rho \le \ep_0,
\end{align}
where $\ep_0<\frac{\pi}{100}$ is a small positive constant to be determined later. 

If we denote the exponential map at $q$ by $\ex_q$, then $\ex_q$ is a local diffeomorphism from the ball $\hat B(0,2\rho)\subset T_q M$ to $B(q,2\rho)$. Equipped with the pullback metric $\hat g=\ex_q^*g$ on $\hat B(0,2\rho)$, $\ex_q$ is a local isometry.

Now we recall the notion of \emph{fundamental pseudo-group} of Gromov, see also \cite{Fu88}. 
\begin{defn}\label{def:group01}
The fundamental pseudo-group at $q$ and scale $\rho$ is defined as
\begin{align*} 
\Gamma(q,\rho) \coloneqq \{\tau \in C\lc\hat B(0,\rho),\hat B(0,2\rho)\rc \mid \, \emph{\ex}_q\circ \tau=\emph{\ex}_q\}
\end{align*}
where $C\lc\hat B(0,\rho),\hat B(0,2\rho)\rc$ consists of all continuous maps from $\hat B(0,\rho)$ to $\hat B(0,2\rho)$.
\end{defn}

It is easy to see $\tau \in \Gamma(q,\rho)$ if and only if $\tau \in C\lc\hat B(0,\rho), T_qM\rc$ such that $\ex_q \circ \tau=\ex_q$ and $\tau(0) \in \hat B(0,\rho)$. Next we prove the following two lemmas. Notice that for a geodesic loop $\gamma(t)$, we always assume that $t \in [0,1]$ and $\gamma$ is smooth except possibly at the base point. Moreover, we denote its length by $L(\gamma)$.

\begin{lem}\label{L:group1}
There exists an one-to-one correspondence between $\Gamma(q,\rho)$ and the set of all geodesic loops at $q$ with length smaller than $\rho$.
\end{lem}
\begin{proof}
For any geodesic loop $\gamma$ based at $q$ with $L(\gamma) <\rho$, we can lift $\gamma$ through the map $\ex_q$ to a segment starting from $0$. Denote the end point by $v$, then $\gamma$ corresponds to a unique map $\tau^v \in \Gamma(q,\rho)$ such that $\tau^v(0)=v$. Specifically, for any $w \in \hat B(0,\rho)$, there exists a geodesic $\gamma_1=\ex_q(tw)$ based at $q$. Since $\ex_q$ is a local covering map and $\ex_q(v)=q$, we can lift $\gamma_1$ to a geodesic on $\hat B(0,2\rho)$ starting from $v$. If we denote the end point by $w'$, then we define $\tau^v(w)=w'$. Conversely, any $\tau \in \Gamma(q,\rho)$ corresponds to a unique geodesic loop based at $q$ defined by $\{\gamma(t)=\ex_q(t\tau(0)):t \in  [0,1]\}$. 
\end{proof}

\begin{rem}
If we denote the exponential map corresponding to $\hat g$ by $\emph{Exp}$, then
$$
\tau^v=\emph{Exp}_v \circ (d_v \exp_q)^{-1}.
$$
\end{rem}

There is a natural product, denoted by $*$, in $\Gamma(q,\rho)$. More precisely, for any $\tau_1,\tau_2 \in \Gamma(q,\rho) $ such that $v \coloneqq \tau_1(\tau_2(0)) \in \hat B(0,\rho)$, we define for any $x \in \hat B(0,\rho)$,
\begin{align} \label{eq:product}
\tau_1 *\tau_2(x) =\tau_1(\tau_2(x)).
\end{align}
By this definition, $\tau_1 *\tau_2$ is identical with $\tau^v$, which by Lemma \ref{L:group1} is the unique map in $\Gamma(q,\rho)$ such that $\tau^v(0)=v$. If we regard $\tau_1,\tau_2$ and $\tau^v$ as short geodesic loops, then $\tau^v$ is the unique geodesic loop in the short homotopy class of the loop ``$\tau_2$ followed by $\tau_1$". Therefore the product \eqref{eq:product} agrees with Gromov's product of short geodesic loops, see \cite[Definition $2.2.3$]{BK81}.

Therefore, by Lemma \ref{L:group1} and \eqref{eq:product}, when we refer to the small geodesic loops and their products, the corresponding local isometries and compositions in the fundamental pseudo-group are implicitly understood and vice versa.

\begin{lem}\label{L:group2}
$B(q,\rho)$ is isometric to $\hat B(0,\rho)/\Gamma(q,2\rho)$.
\end{lem}

\begin{proof}
We consider the map $\ex_q: \hat B(0,\rho)/\Gamma(q,2\rho) \to B(q,\rho)$. The surjectivity is obvious. For the injectivity, we assume that there exist $w_1$ and $w_2$ such that $\ex_q(w_1)=\ex_q(w_2)=q' \in B(q,\rho)$. Denote the geodesic from $q'$ to $q$ by $\{\gamma_1(t) \coloneqq \ex_q((1-t)w_2):\, t\in [0,1]\}$, then $\gamma_1$ can be lift to a geodesic starting from $w_1$. If we denote the end point by $v$, then $|v| \le |w_1|+|w_2|<2\rho$. Therefore the map $\tau^v$ is well defined and by its definition $\tau^v(w_2)=w_1$.
\end{proof}

\begin{defn}\label{def:rigid}
For any geodesic loop $\gamma \in \Gamma(q,\rho)$, its holonomy motion is defined as
\begin{align*} 
\mathbf m(\gamma):\,& T_qM \to T_qM \\
\mathbf m(\gamma)(x)=& \mathbf r(\gamma)(x)+\mathbf t(\gamma),
\end{align*}
where the rotational part $\mathbf r(\gamma)$ is the parallel transport around $\gamma$ and the translational part $\mathbf t(\gamma):=\mathbf r(\gamma)(\dot \gamma(0))$. The norm of $\mathbf r(\gamma)$ is defined by
\begin{align*} 
\|\mathbf r(\gamma)\| \coloneqq \max\{\angle (\mathbf r(\gamma)(v),v) \mid\,v\in T_qM\}.
\end{align*}
\end{defn}

Notice that $\|\mathbf r(\gamma)\|$ is uniformly comparable to the usual matrix norm $|\mathbf r(\gamma)-I| \coloneqq \max\{|\mathbf r(\gamma)v-v|\mid\,v\in T_qM,\,|v|=1\}$, where $I$ is the identity map. That is,
\begin{align*}
 \frac{2}{\pi} \|\mathbf r(\gamma)\| \le |\mathbf r(\gamma)-I| \le \|\mathbf r(\gamma)\|.
\end{align*}

For $\gamma \in \Gamma(q,\rho)$, we set $c=\mathbf t(\gamma)$, $\mathbf r_c=\mathbf r(\gamma)$ and $\tau_c$ to be the corresponding local isometry of $\gamma$. We recall the following lemma in \cite{Mi10} which indicates that $\tau_c$ is almost a translation if both $\|\mathbf r_c\|$ and the curvature are sufficiently small.

\begin{lem}[Lemma $2.4$ of \cite{Mi10}]\label{L:es1}
For any point $w \in \hat B(0,\rho)$,
\begin{align*} 
d_{\hat g}(\tau_c(w),\mathbf r_c^{-1}(w+c)) \le \lam^2|c||w|(|c|+|w|).
\end{align*}
\end{lem}

For a geodesic loop $\gamma$ at $q$ with $L(\gamma) <\rho$, we define the \emph{sliding} of $\gamma$ to nearby points as follows. If $x$ is a point near $q$ such that $d(q,x)<\inj(q)$, let $w$ be the unique lift of $x$ on $\hat B(0,\inj(q))$. For the map $\tau \in \Gamma(q,\rho)$ corresponding to $\gamma$, we set $w'=\tau(w)$. Then the sliding of $\gamma$ at $x$ is defined to be the geodesic loop $\gamma_x=\ex_q(\hat \gamma)$ based at $x$, where $\hat \gamma$ is a geodesic with respect to $\hat g$ connecting $w$ and $w'$. Here we have used an important fact (\cite[Corollary $8.13$]{Gr06}) that any two points in $\hat B(q,2\rho)$ are connected by a unique minimizing geodesic. 

If $d(q,x)\ge \inj(q)$, there is no natural way to define the sliding of $\gamma$ to $x$ since the lift of $x$ to $\hat B(0,\rho)$ may not be unique. To overcome this, we consider a curve $\{\alpha(t): t \in [0,1]\}$ starting from $q$ such that for any $t \in [0,1]$, 
\begin{align}\label{eq:res}
d(q,\alpha(t)) <\rho.
\end{align}
Notice that $\alpha(t)$ has a unique lift $\{\tilde \alpha(t): t \in [0,1]\}$ on $\hat B(0,\rho)$ starting from $0$. Now we have the following definition.

\begin{defn}[Sliding along a curve]\label{def:sliding}
For any $\gamma \in \Gamma(q,\rho)$, the sliding of $\gamma$ at $\alpha(t)$ is defined to be the geodesic loop $\gamma_t \coloneqq \ex_q(\tilde \gamma_t)$, where $\tilde \gamma_t$ is the unique geodesic between $\tilde \alpha(t)$ and $\tau(\tilde \alpha(t))$ and $\tau$ is the corresponding map of $\gamma$.
\end{defn}

We next show that the sliding of $\gamma$ along $\{\alpha(t): 0\le t \le 1 \}$ is transitive. More precisely,

\begin{lem}[Transitivity] \label{L:trans}
Let $\gamma_t$ be the sliding of a geodesic loop $\gamma$ at $q=\alpha(0)$ along a curve $\alpha(t)$. Then for any $0 \le t_1,t_2 \le 1$, the sliding of $\gamma_{t_1}$ to $\alpha(t_2)$ along $\alpha$ is $\gamma_{t_2}$. 
\end{lem}

\begin{proof}
Given $t_1$, we set $q_1=\alpha(t_1)$ and denote the sliding of $\gamma$ at $q_1$ to be $\gamma_1$. Moreover, we assume that $\{\alpha(t): 0\le t \le 1 \}$ is lifted to a curve $\tilde \alpha_1(t)$ on $\hat B(0,\rho) \subset T_{q_1}M$ such that $\tilde \alpha_1(t_1)=0$. If we set $\gamma_t$ to be the sliding of $\gamma$ at $\alpha(t)$, then we lift $\gamma_t$ to a geodesic starting from $\tilde \alpha_1(t)$ and denote the other end point by $\tilde \alpha_2(t)$. If we set $v_1=\tilde \alpha_2(t_1)$, then it is clear that $\tilde \alpha_2(t)$ is the unique lift of $\alpha(t)$ through $v_1$. In other words, $\tau^{v_1}(\tilde \alpha_1)=\tilde \alpha_2$. By the uniqueness of the geodesic from $\tilde \alpha_2(t)$ to $\tilde \alpha_1(t)$, it follows immediately that the sliding of $\gamma_1$ agrees with the sliding of $\gamma$ along $\alpha(t)$ for $t \in [0,1]$.
\end{proof}

As long as the length of the geodesic loop is smaller than $\rho$ and the local lift is possible, we can unambiguously define the sliding of $\gamma$ along any $\alpha(t)$ for all $t \in [0,1]$ without the restriction \eqref{eq:res}, since locally the lift of $\alpha$ at any base point is unique and the sliding is transitive. Notice that Lemma \ref{L:trans} also holds for general curve $\alpha$.

Next, we show that the sliding depends only on the initial geodesic loop and the homotopy class of a given curve.

\begin{prop} \label{P:homo}
Let $\{\alpha_i(t): 0\le t \le 1 \}\,(i=1,2)$ be curves from $q$ to $q'$ and $\{\alpha_s(t): 0\le t \le 1, 1\le s\le 2 \}$ a homotopy between $\alpha_1$ and $\alpha_2$ such that for a fixed $\rho>0$, any $t \in [0,1]$ and $s\in [1,2]$, the curvature assumption \eqref{eq:curv2} is satisfied at $\alpha_s(t)$. Given a geodesic loop $\gamma$ at $q$ such that for any $1 \le s\le 2$, the length of the sliding of $\gamma$ along $\alpha_s$ is smaller than $\rho$, the slidings of $\gamma$ at $q'$ along $\alpha_1$ and $\alpha_2$ are identical.  
\end{prop}

\begin{proof}
From the transitivity of the sliding, we can assume $d(q,\alpha_s(t))<\rho$ for any $t \in [0,1]$ and $s\in [1,2]$, by decomposing the original homotopy. From our curvature assumption, we can lift $\alpha_s(t)$ to a homotopy $\tilde \alpha_s(t)$ on $\hat B(0,2\rho) \subset T_qM$. It is clear from our definition that the sliding of $\gamma$ at $q'$ along $\alpha_1$ agrees with that along $\alpha_2$.
\end{proof}

The sliding also preserves the local group structure. 

\begin{prop} \label{P:group}
Let $\{\alpha(t): 0\le t \le 1 \}$ be a curve from $q$ to $q'$ such that for any $t \in [0,1]$, the curvature assumption \eqref{eq:curv2} is satisfied at $\alpha(t)$. For any $\gamma_0^i \in \Gamma(q,\rho)\,(i=1,2,3)$ such that $\gamma_0^1*\gamma_0^2=\gamma_0^3$, suppose that the slidings along $\alpha(t)$ of $\gamma_0^i$, denoted by $\gamma^i_t$, are well defined and $L(\gamma^i_t)<\rho/2$. Then for any $t \in [0,1]$,
\begin{align*}
\gamma_t^1*\gamma_t^2=\gamma_t^3.
\end{align*}
\end{prop}

\begin{proof}
We assume that $L(\gamma^i_t) \le \rho/2-3\ep$ for $t \in [0,1]$. For a fixed $t_0 \in [0,\ep]$ and any closed loop $c$ at $q$, we define
\begin{align*}
\tilde c(t)=
\begin{cases}
\, \alpha(t_0(1-3t)) \quad &\text{if} \quad 0\le t \le \frac{1}{3}, \\
\, c(3t-1) \quad &\text{if} \quad \frac{1}{3}\le t \le \frac{2}{3}, \\
\, \alpha(t_0(3t-2)) \quad &\text{if} \quad \frac{2}{3}\le t \le 1.
\end{cases}
\end{align*}

Clearly, $\tilde \gamma_0^1 *\tilde \gamma_0^2$ and $\tilde \gamma_0^3$ are shortly homotopic and hence $\tilde \gamma_0^1 *\tilde \gamma_0^2=\tilde \gamma_0^3$. We claim that $\tilde \gamma_0^i$ and $\gamma_{t_0}^i$ are in the same short homotopy class at the base point $\alpha(t_0)$. Indeed, a short homotopy $F:[0,1] \times [0,1] \to M$ is constructed as
\begin{align*}
F(s,t)=
\begin{cases}
\, \alpha(t_0(1-3(1-s)t)) \quad &\text{if} \quad 0\le t \le \frac{1}{3}, \\
\, \gamma_s^i(3t-1) \quad &\text{if} \quad \frac{1}{3}\le t \le \frac{2}{3}, \\
\, \alpha(t_0(3s-2+3(1-s)t)) \quad &\text{if} \quad \frac{2}{3}\le t \le 1.
\end{cases}
\end{align*}
Therefore, $\gamma_{t_0}^1*\gamma_{t_0}^2=\gamma_{t_0}^3$. From the transitivity of the sliding, it is clear that for any $t \in [0,1]$,
\begin{align*}
\gamma_t^1*\gamma_t^2=\gamma_t^3.
\end{align*}
\end{proof}

\subsection{Estimates on the tangent space}
With the same assumptions \eqref{eq:curv1} and \eqref{eq:curv2}, we set $(x_1,x_2,\cdots,x_n)$ be the geodesic coordinates on $\hat B(0,\rho)$ inherited from $\R^n$ and we set $\hat g_{ij}$ to be the coefficients of $\hat g=\ex_q^*g$.

It follows from \eqref{eq:curv1} and the sectional curvature comparison thereom (see \cite[Theorem $27$]{Pe06}) that
 \begin{align}\label{eq:e0}
dt^2+\lc \frac{\sin \lam t}{\lam}\rc^2 ds^2_{n-1} \le \hat g \le dt^2+\lc \frac{\sinh \lam t}{\lam}\rc^2 ds^2_{n-1}
\end{align}
where $t=d_{\hat g}(0,\cdot)$ and $ds^2_{n-1}$ is the standard metric on $S^{n-1}$.
Therefore, by  \eqref{eq:curv2} for any $x \in \hat B(0,\rho)$,
 \begin{align}\label{eq:e1}
|\hat g_{ij}(x)-\delta_{ij}| \le C\lam^2 \rho^2
\end{align}
and
 \begin{align}\label{eq:e1aa}
(1-C\lam^2\rho^2)dx \le dV_{\hat g}(x) \le (1+C\lam^2\rho^2)dx.
\end{align}
Moreover, for any $a,b \in B_{\hat g}(0, \rho/2)$,
 \begin{align}\label{eq:e1a}
(1-C\lam^2 \rho^2)|a-b| \le d_{\hat g}(a,b) \le (1+C\lam^2 \rho^2) |a-b|.
\end{align}

In particular, since $\ep_0$ in our assumption \eqref{eq:curv2} is small, the inner product induced by $\hat g$ and the distance $d_{\hat g}$ at its tangent space are uniformly comparable to the Euclidean inner product and the Euclidean distance, respectively. 

Next, we prove some comparison estimates for the distance function.

\begin{lem}\label{lem:dis}
Under conditions \eqref{eq:curv1} and \eqref{eq:curv2}, if we set $\eta=d^2_{\hat g}(0,\cdot)/2$, then on $\hat B(0,\rho)$,
\begin{align} \label{eq:dis1}
|\na^2 \eta-\hat g | \le C \lam^2\rho^2.
\end{align}
Moreover, if $|\na^k Rm| \le c_k \lam^2 \rho^{-k}$ for any $k \ge 1$, then
\begin{align} \label{eq:dis2}
|\na^{k+2} \eta | \le C_k \lam^2 \rho^{2-k}.
\end{align}
\end{lem}

\begin{proof}
The estimate \eqref{eq:dis1} follows from the standard Hessian comparison theorem. Indeed, it follows from \cite[Theorem $27$]{Pe06} that
 \begin{align}\label{eq:dis3}
\frac{\sin 2\lam t}{2\lam} ds^2_{n-1}\le \na^2  t \le \frac{\sinh 2\lam t}{2\lam} ds^2_{n-1}
\end{align}
where $t=d_{\hat g}(0,\cdot)$.
Since $\na^2 \eta=t\na^2 t+dt^2$ and $\lam \rho$ is sufficiently small, it is clear from \eqref{eq:dis3} and \eqref{eq:e0} that
 \begin{align*}
  |\na^2 \eta-\hat g | \le C \lam^2 \rho^2.
\end{align*}
The higher-order estimates follow from \cite[Appendix B]{Mi10} and we sketch the proof for completeness. We set $E=\na^2 \eta-\hat g$ and $V=\na \eta=t\partial_t$. then it follows from the Riccatti equation that
 \begin{align}\label{eq:dis4}
 \na_V E=-E-E^2-Rm(\cdot,V)V.
\end{align}

From \eqref{eq:dis4} we have
 \begin{align*}
\na_V \na E=-2\na E+E*\na E+\na Rm*V*V+Rm*\na V*V+Rm*V.
\end{align*}

Moreover, for any $k\ge 2$,
 \begin{align*}
\na_V \na^k E=-(k+1)\na^k E+\sum_{i+j=k} \na^i E*\na^j E+\sum_{i+j+l=k}\na^i Rm*\na^j V* \na^l V+\sum_{i+j=k-1} \na^iRm*\na^j V.
\end{align*}

If we set $E_k=t^{k+1}\na^k E$, then 
 \begin{align}\label{eq:dis7}
\partial_t E_k=t^{-1} E*E_k+F_k,
\end{align}
where
 \begin{align*}
F_k=&t^{-2}\sum_{i=1}^{k-1} E_i*E_{k-i}+t^k \lc \sum_{i+j+l=k}\na^i Rm*\na^{j} V* \na^l V+\sum_{i+j=k-1} \na^iRm*\na^j V \rc \\
=&t^{-2}\sum_{i=1}^{k-1} E_i*E_{k-i}+t^k \lc \sum_{i+j+l=k}\na^i Rm*\na^{j+1} \eta* \na^{l+1} \eta+\sum_{i+j=k-1} \na^iRm*\na^{j+1} \eta \rc.
\end{align*}

Now we assume \eqref{eq:dis2} holds for any $i \le k-1$, then from \eqref{eq:dis1},
 \begin{align*}
|F_k| \le C(t^k  \lam^4 \rho^{4-k}+t^k  \lam^2 \rho^{2-k}+t^k  \lam^6 \rho^{6-k}+t^k  \lam^4 \rho^{4-k}) \le C \lam^2 \rho^{2}.
\end{align*}

Since $|E| \le C \lam^2 \rho^2$ is small, it is easy to derive from \eqref{eq:dis7} that
 \begin{align*}
|\na^{k+2} \eta |=|\na^k E| =t^{-k-1}|E_k| \le Ct^{-k-1} \lam^2 \rho^3 \le C \lam^2 \rho^{2-k}.
\end{align*}
\end{proof}

For geodesic coordinate system $(x_1,x_2,\cdots,x_n)$, we set $e_i=\partial_{x_i}$ for $1 \le i\le n$. Then it is clear that $\{e_1,e_2,\cdots, e_n\}$ is an orthonormal basis at $0$. Now we set $\tilde e_i(x)$ to be the vector at $x \in \hat B(0,\rho)$ obtained by the parallel transport of $\left.e_i \right|_0$ along the radial geodesic between $0$ and $x$. It follows from \cite[Proposition $6.6.2$]{BK81} that for any $x \in \hat B(0,\rho-1)$,
 \begin{align*}
|\tilde e_i(x)-e_i(x)|_x \le \frac{\sinh \lam \rho}{\lam\rho}-1 \le C\lam^2\rho^2,
\end{align*}
where we set $|\cdot|_x$ to be the distance induced by $\hat g$ at the tangent space of $x$. Therefore, from \eqref{eq:e1} we have
 \begin{align}\label{eq:e3}
|\tilde e_i(x)-e_i(x)| \le C\lam^2\rho^2.
\end{align}

Now we need the following existence result of a harmonic coordinate system from \cite[Section $2.8$]{Jo83}, see also \cite[Fact $(2.9)$]{BKN89}.

\begin{lem}\label{L:es2}
With the assumptions above, there exist constants $\lambda=\lambda(n)>0$, $C_0=C_0(n)>0$ and a harmonic map
\begin{align*} 
\mathbb H=(h^1,h^2,\cdots,h^n): \, \hat B(0,\lambda\rho)   \to \R^n
\end{align*}
which satisfies the following properties
 \begin{align}
|g_{ij}-\delta_{ij}| \le& C_0 \Lambda^2 \rho^2, \label{eq:es0a}\\
\sum_{|\beta|=1}|\partial^{\beta}g_{ij}| \le& C_0\Lambda^2 \rho, \label{eq:esb}\\
|\na_{\hat g}h^i-\tilde e_i| \le& C_0 \Lambda^2 \rho^2.\label{eq:esc}
\end{align}
Here $g_{ij}$ are the coefficients of $\hat g$ under the map $\mathbb H$.
\end{lem}

Let $(y_1,y_2,\cdots,y_n)$ be the harmonic coordinates constructed in Lemma \ref{L:es2} and $(\partial_1,\partial_2,\cdots,\partial_n)$ the corresponding vector fields. Now we prove

\begin{lem}\label{L:es2a}
For any $x \in \hat B(0,\lambda \rho)$,
\begin{align*} 
|\emph{d}\mathbb H(x)-I| \le C\lam^2\rho^2.
\end{align*}
\end{lem}

\begin{proof}
If we set $h_i^j(x) \coloneqq \partial_{x_i}h^j(x)$ to be the components of $\text{d}\mathbb H$, then it follows from \eqref{eq:e3} and \eqref{eq:esc} we have
 \begin{align}\label{eq:e6}
|\hat g^{ik}h^j_i e_k-e_j| \le C\lam^2\rho^2.
\end{align}
It is clear from \eqref{eq:e6} that
 \begin{align*}
|\hat g^{ik}h^j_i-\delta_{jk}| \le C\lam^2\rho^2
\end{align*}
and hence by \eqref{eq:e1}
 \begin{align*}
|h_i^j-\delta_{ij}| \le C\lam^2\rho^2.
\end{align*}
\end{proof}

It follows from \eqref{eq:e1}, \eqref{eq:es0a} and Lemma \ref{L:es2a} that up to the first order, the geodesic coordinates and harmonic coordinates are almost the same, with an error controlled by $\lam^2\rho^2$.

Next, we estimate how close a geodesic in $\hat B(0,\rho)$ is to a straight segment. This result is important when we estimate the angle between two geodesic loops.

\begin{prop}\label{P:es3}
There exist constant $C_1(n)>0$ and $\ep_0=\ep_0(n)\in (0,\frac{\pi}{100})$ such that if $\Lambda^2 \rho^2 \le \ep_0$, then for any $a,b \in \hat B(0,\frac{\lambda}{10}\rho)$, if $\gamma$ is the geodesic with respect to $\hat g$ from $a$ to $b$, then
\begin{align*} 
\angle (\dot \gamma(0),b-a) \le C_1 \Lambda^2\rho^2.
\end{align*}
\end{prop}

\begin{proof}
From Lemma \ref{L:es2}, there exists a harmonic coordinate system on $\hat B(0,\lambda\rho)$. Under this coordinates, if we write $\gamma(t)=(y_1(t),y_2(t),\cdots,y_n(t))$ for $t \in [0,1]$, then the geodesic equation can be written as
\begin{align} \label{eq:es2}
y''_k+\sum_{i,j}\Gamma^k_{ij} y'_i y'_j=0,
\end{align}
for any $1 \le k \le n$. Moreover, since the length of the tangent vector $\dot \gamma$ is constant, we have
\begin{align} \label{eq:es3}
\sum_{ij} g_{ij}(\gamma(t))  y'_i  y'_j=\hat L^2.
\end{align}

Since the entire geodesic $\gamma$ is contained in $ \hat B(0,\lambda\rho)$, it follows from \eqref{eq:es3} that if we set $\bar L^2=\sum_i  (y'_i)^2$, then
\begin{align} \label{eq:es4}
|\hat L^2-\bar L^2| \le C_0 \Lambda^2\rho^2 \bar L^2 \le C_0 \ep_0 \bar L^2
\end{align}
If we choose $\ep_0<\frac{1}{2C_0}$, then from \eqref{eq:es4} we have
\begin{align*}
\frac{1}{2}\bar L \le  \hat L \le 2 \bar L.
\end{align*}
From the definition of the Christoffel symbol $\Gamma^k_{ij}=\frac{1}{2}g^{kl}(\partial_i g_{jl}+\partial_j g_{il}-\partial_l g_{ij})$, it is clear from \eqref{eq:es0a}, \eqref{eq:esb} and \eqref{eq:es2} that for any $t \in [0,1]$,
\begin{align*}
| y''_k(t)| \le C \Lambda^2 \rho \hat L^2
\end{align*}
By integration, for any $t \in [0,1]$,
\begin{align*}
|y'_k(t)- y'_k(0)| \le C\Lambda^2 \rho \hat L^2
\end{align*}
and hence
\begin{align*}
|y_k(1)-y_k(0)-y'_k(0)| \le \int_0^{1}  |y'_k(t)-y'_k(0)|\,dt \le C\Lambda^2 \rho \hat L^2.
\end{align*}

In other words,
\begin{align}
|\mathbb H(b)-\mathbb H(a)-\text{d}\mathbb H_a(\dot \gamma(0))| \le C\Lambda^2 \rho \hat L^2. \label{eq:es7x1}
\end{align}

On the other hand, it follows from Lemma \ref{L:es2a} that
\begin{align}
|\mathbb H(b)-\mathbb H(a)-\text{d}\mathbb H_a(b-a)|\le \int_0^1 |\text{d}\mathbb H_{a+(b-a)t}(b-a)-\text{d}\mathbb H_a(b-a)| \,dt \le C\Lambda^2 \rho^2 \hat L.  \label{eq:es7x2}
\end{align}

Combining \eqref{eq:es7x1} and \eqref{eq:es7x2},
\begin{align*}
|\text{d}\mathbb H_a(\dot \gamma(0))-\text{d}\mathbb H_a(b-a)| \le C\Lambda^2 \rho \hat L^2+C \lam^2 \rho^2 \hat L \le C \lam^2 \rho^2 \hat L
\end{align*}

From Lemma \ref{L:es2a}, $|\dot \gamma(0)-(b-a)| \le C \lam^2 \rho^2 \hat L$ and the conclusion follows immediately.
\end{proof}

Next, we estimate the change of the tangent vector after the sliding. As before, for $\gamma \in \Gamma(q,\frac{\lambda}{10}\rho)$ we set $c=\mathbf t(\gamma)$, $\mathbf r_c=\mathbf r(\gamma)$ and $\tau_c$ to be the corresponding map of $\gamma$. 

\begin{prop}\label{P:es4}
Suppose 
\begin{align*} 
\|\mathbf r_c\| \le C \lam^2\rho|c| \quad \text{and} \quad C^{-1}|c|\le |\tau_c(w)-w| \le C|c|
\end{align*}
for some $w \in B_{\hat g}(0,\frac{\lambda}{10}\rho)$ and a constant $C=C(n)>1$. Then there exists a constant $C_2=C_2(n)>0$ such that
\begin{align*} 
\angle (\dot \gamma_{w}(0),c) \le C_2 \Lambda^2\rho^2,
\end{align*}
where $\gamma_w$ is the geodesic with respect to $\hat g$ from $w$ to $\tau_c(w)$.
\end{prop}

\begin{proof}
It follows from Lemma \ref{L:es1} that
\begin{align*}
&|\tau_c(w)-(w+c)|  \\
 \le& |\tau_c(w)-\mathbf r_c^{-1}(w+c)|+\|\mathbf r_c\|(|w|+|c|) \\
 \le& C\lam^2|c||w|(|c|+|w|)+C \lam^2\rho^2|c| \le C \lam^2 \rho^2 |c|. 
\end{align*}

Then it is easy to see from the law of cosine that
\begin{align} \label{eq:es14}
\angle (\tau_c(w)-w,c)\le C\lam^2 \rho^2.
\end{align}

Now it follows from Proposition \ref{P:es3} that
\begin{align} \label{eq:es15}
\angle (\dot \gamma_w(0),\tau_c(w)-w) \le C_1 \Lambda^2\rho^2.
\end{align}

Combining \eqref{eq:es14} and \eqref{eq:es15}, we conclude that
\begin{align*} 
\angle (\dot \gamma_w(0), c) \le&  \angle (\dot \gamma_w(0),\tau_c(w)-w)+\angle (\tau_c(w)-w,c) \le C\Lambda^2\rho^2.
\end{align*}
\end{proof}

\subsection{Tangent cone at infinity and a rough fibration}
\label{sub:tan}

Let $(M^{n},g,p)$ be a complete Riemannian manifold satisfying \eqref{cond:AF}. It follows from \cite{Ka88} (see also \cite[Corollary $0.4$]{MNO05}) that the tangent cone at infinity exists and is unique. More precisely, there exists $(C(S(\infty)), d_{\infty},p_{\infty})$ such that for any sequence $r_i \to \infty$, 
\begin{align*}
(M^{n},r_i^{-2}g,p) \longright{PGH}  (C(S(\infty)), d_{\infty},p_{\infty}),
\end{align*}
where the convergence is in the pointed Gromov-Hausdorff sense. 

Following \cite{Ka88}, we describe briefly how to construct the limit space $C(S(\infty))$, which is a metric cone over $S(\infty)$. Let $S(t)$ be the geodesic sphere of the radius $t$ with center $p$, we denote the intrinsic distance on $S(t)$ by $d_t$. In particular, if $x$ and $y$ are on different components of $S(t)$, then $d_t(x,y)=\infty$. Two geodesic rays $\sigma$ and $\gamma$ starting from $p$ are equivalent if 
\begin{align*}
\lim_{t \to \infty} \frac{d_t(\sigma(t),\gamma(t))}{t}=0.
\end{align*}
$S(\infty)$ is the set of equivalence classes of all geodesic rays starting from $p$. A distance $\theta_{\infty}$ on $S(\infty)$ is defined by
\begin{align*}
\theta_{\infty}([\sigma],[\gamma])=\lim_{t \to \infty} \frac{d_t(\sigma(t),\gamma(t))}{t}.
\end{align*}
Equipped with $\theta_{\infty}$, $S(\infty)$ becomes a compact length space. Moreover, there is a map $\Phi_{t,\infty}:S(t) \to S(\infty)$ satisfying $\Phi_{t,\infty}(\sigma(t))=[\sigma]$ for any ray $\sigma$. Under this map, there is a one-to-one correspondence between the components of $S(t)$ and those of $S(\infty)$, if $t$ is large. It can be proved, see \cite[Proposition $2.2$, (iv)]{Ka88}, that if $t \to \infty$,
\begin{align*}
(S(t),t^{-1}d_t) \longright{GH}  (S(\infty), \theta_{\infty}).
\end{align*}

From \cite[Proposition $2.3$(iii)]{Ka88}, each component $S_a(\infty)$ of $S(\infty)$ satisfies
\begin{align}\label{eq:diam}
\text{diam}\,S_a(\infty)=\lim_{t\to \infty} \text{diam}( S_{a}(t),t^{-1}d_t ) <\infty,
\end{align}
where $S_a(t)$ is the corresponding component of $S(t)$. In addition, the number of the components of $S(\infty)$ is always bounded by a constant which depends only on $n$ and $K(t)$, see \cite[Theorem B]{Ab85}. Notice that in this paper we mainly focus on the case that $S(\infty)$ is connected. 

\begin{rem}\label{rem1}
In \emph{\cite{Ka88}}, the curvature condition can be relaxed to the so-called asymptotically nonnegative curvature. Namely, the sectional curvature is bounded below by $-r^{-2}K(r)$, where $K(r)$ satisfies \eqref{eq:003}.
\end{rem}

Next, we show that condition \eqref{cond:HC} is equivalent to \eqref{cond:HC'}.

\begin{lem}\label{lem:holaa1}
Let $(M^{n},g,p)$ be a complete Riemannian manifold with \eqref{cond:AF} and \eqref{cond:HC}. Then there exists a positive number $C_H>0$ such that 
\begin{align*} 
\|\mathbf r (\gamma_x)\| \le C_H\frac{L(\gamma_x)}{r} 
\end{align*}
for any $x$ outside a compact set and any geodesic loop $\gamma_x$ based at $x$ with length smaller than $\kappa r$, where $\kappa$ is the constant in \eqref{cond:HC}.
\end{lem}

\begin{proof}
For any $\gamma \in \Gamma(x,\kappa r)$, we set $L=|\mathbf t(\gamma)|$ and $k=\floor{\frac{\kappa r}{L}}$. Then $|\mathbf t(\gamma^k)| \le \kappa r$ and by \eqref{cond:HC},
\begin{align*}
\|\mathbf r(\gamma^k)\| \le \Theta_H<\pi/2.
\end{align*}

Since $\kappa<1/2$ and $|Rm| \le 4r^{-2}K(r/2)$ on $B(x,r/2)$, it follows from \cite[Proposition $2.3.1$ (i)]{BK81} that if $r$ is sufficiently large,
\begin{align*} 
\|\mathbf r(\gamma^k)\| \ge k\|\mathbf r(\gamma)\|-2r^{-2}k^2K(r/2) L^2.
\end{align*}

Therefore, we have
\begin{align*}
\|\mathbf r(\gamma)\| \le C\frac{\Theta_H+K(r/2)}{r} |\mathbf t(\gamma)|.
\end{align*}
\end{proof}

Next, we prove

\begin{thm}\label{T:cone}
For any complete Riemannian manifold $(M^{n},g)$ with \eqref{cond:AF} and \eqref{cond:HC'}, $(C(S(\infty)),p_{\infty})$ is a smooth Riemannian manifold away from the vertex $p_{\infty}$.
\end{thm}

\begin{proof}
If $(M,g)$ has Euclidean volume growth, it follows from \cite{Ka89} that $C(S(\infty))$ is a flat cone $C(S^{n-1}/\Gamma)$, where $\Gamma \subset O(n)$ is a finite subgroup acting freely on $S^{n-1}$. So we only need to consider the collapsing case.

For any sequence $q_i \in M$ with $r_i=r(q_i) \to \infty$, we consider the rescaled metric $g_i=r_i^{-2}g$ and the local group $\Gamma_i=\Gamma(q_i, 2 \hat \kappa)$ for $\hat \kappa=\min\{C_H^{-1}/10,\kappa/10\}$, where $C_H$ and $\kappa$ are constants in \eqref{cond:HC'}. Then it follows from Lemma \ref{L:es1} that $\Gamma_i$ converges to a local group $\Gamma_{\infty} \subset \text{Iso}(\R^n)$. In addition, from  Lemma \ref{L:group2} and \cite[Lemma $1.11$]{Fu88} $B(q_i,\hat \kappa)$ converges, in the Gromov-Hausdorff sense, to $B(0,\hat \kappa)/\Gamma_{\infty}$, where $B(0,\hat \kappa) \subset \R^n$. Since $\Gamma_{\infty}$ is locally isometric to a Lie group, see \cite[Lemma $3.1$]{Fu88}, we only need to prove that the action of any $a \in \Gamma_{\infty} \backslash \{1\}$ on $B(0,\hat \kappa)$ is free.

Otherwise, we assume that $y \in B(0,\hat \kappa)$ is a fixed point of $a$, then we can find a sequence of $\tau_{c_i} \to a$ and $w_i \in \hat B(0,\hat \kappa) \in T_{q_i}M$ such that $w_i \to y$. On the one hand it follows from Lemma \ref{L:es1} and \eqref{eq:e1}
\begin{align}\label{E111} 
|\tau_{c_i}(w_i)-\mathbf r_{c_i}^{-1}(w_i+c_i)| \le 4K(r_i/2)|c_i||w_i|(|c_i|+|w_i|).
\end{align}

On the other hand, \eqref{cond:HC'} implies 
\begin{align}\label{E112}  
|\mathbf r_{c_i}(w_i)-w_i| \le C_H|c_i||w_i|.
\end{align}

By taking a subsequence if necessary, we assume that $c_i \to c_{\infty}$ which is nonzero since $a \ne 1$. Since $\tau_{c_i}(w_i)$ and $w_i$ converge to $y$ if $i \to \infty$, we have by taking the limits of \eqref{E111} and \eqref{E112},
\begin{align*} 
|c_{\infty}| \le C_H|c_{\infty}||w_{\infty}|
\end{align*}
which is a contradiction since $\hat \kappa \le C_H^{-1}/10$. 
\end{proof}

\begin{rem}\label{rem:sec}
Notice that $C(S(\infty))$ is not necessarily flat, but we have the following estimates of the sectional curvature from \emph{\cite[Lemma $7.2$]{Fu88}}
\begin{align*} 
0 \le \emph{sec}(x) \le 6C_H^2r(x)^{-2},
\end{align*}
where $r(x)$ is the distance to the vertex $p_{\infty}$. If we denote the sectional curvature of $S(\infty)$ by $\text{sec}_S$, then we have
\begin{align*} 
1 \le \emph{sec}_S(x) \le 1+6C_H^2.
\end{align*}
\end{rem}

If the tangent cone at infinity is smooth, we can construct a fibration on the end following \cite{CFG92}.

\begin{thm}\label{T:infra}
Let $(M^n,g,p)$ be a complete Riemannian manifold with \eqref{cond:AF} such that $C(S(\infty))$ is smooth away from the vertex. Then there exist a compact set $K \subset M$, a constant $R_0>0$ and a fibration $f: M^{n} \backslash K \longrightarrow C(S(\infty))\backslash \bar B(p_{\infty},R_0)$ with fiber $F$. Moreover, the $f$ satisfies the following properties.
\begin{enumerate} [label=(\roman*)]
\item $F$ is an infranilmanifold with diameter bounded by $o(r)$.

\item $f$ is an $o(1)$-almost-Riemannian submersion.

\item The second fundamental form of $F$ is bounded by $O(r^{-1})$.
\end{enumerate}
\end{thm}

\begin{proof}
We define $A(t,s) \coloneqq B(p,s)\backslash \bar B(p,t)$ and $A_{\infty}(t,s) \coloneqq B(p_{\infty},t)\backslash \bar B(p_{\infty},s)$. From the definition, for large $r$ the rescaled space $(A(r/4,4r),r^{-2}g)$ is close, in the Gromov-Hausdorff sense, to $A_{\infty}(1/4,4)$. Then it follows from \cite[Theorem $2.6$]{CFG92} that there exists a fibration $f_r: A(r/4,4r) \longrightarrow A_{\infty}(r/4,4r)$  satisfying the properties (i),(ii) and (iii) above.

Now we set $R_i =2^iR_0$ for a large constant $R_0$ and consider the annuli $A_i=A(2R_i/3,3R_i/2)$. Notice that only two consecutive annuli can have nonempty intersection. Let $f_i$ be the fibration from $A_i$ to $A_{\infty}(2R_i/3,3R_i/2)$ obtained above. It follows from \cite[Proposition $5.6$, Proposition $2.30$]{CFG92} that there exists a self-diffeomorphism $\phi_i$ of $A_{\infty}(4R_i/3,3R_i/2)$ such that $\phi_i \circ f_{i+1}$ and $f_i$ are close, after rescaling, in the $C^1$ sense on a neighborhood of $A_i \cap A_{i+1}$. By using \cite[Appendix $2$]{CFG92}, there exists a self-diffeomorphism $\psi_i$ on a small neighborhood of $A_i \cap A_{i+1}$ such that 
\begin{align*} 
\phi_i \circ f_{i+1}\circ \psi_i=f_i 
\end{align*}
By choosing a cutoff function, we can define a new fibration $\widetilde{f_{i+1}}$ on $A(R_i/4,4R_i)$ such that $\widetilde{f_{i+1}}=\phi_i \circ f_{i+1}\circ \psi_i$ on a small neighborhood of $A_i \cap A_{i+1}$ and $\widetilde{f_{i+1}}=f_{i+1}$ outside a larger neighborhood. In this way, we can modify our fibrations successively and construct a global fibration $f:M\backslash K \longrightarrow A_{\infty}(3R_0/2,+\infty)$ for some compact set $K$. In addition, the properties (i),(ii) and (iii) still hold.
\end{proof}

\begin{rem}\label{rem:struc}
It follows from \emph{\cite{Fu89}} and \emph{\cite{CFG92}} that for any $x \in C(S(\infty))\backslash \bar B(p_{\infty},R_0)$, there exists a flat connection on $f^{-1}(x)$ which depends smoothly on $x$. Moreover, there exists a simply-connected nilpotent group $N$ and a group of affine transformations $\Gamma$ of $N$ such that $f^{-1}(x)$ is affine equivalent to $N/\Gamma$ and $[\Gamma:\Gamma \cap N]<\infty$. Therefore, the structure group of $f$ is contained in
\begin{align*}
C(N)/(C(N)\cap \Gamma) \rtimes \emph{Aut}\,\Gamma,
\end{align*}
where $C(N)$ is the center of $N$, see \emph{\cite[Theorem 0-1,(0-3-3)]{Fu89}}.
\end{rem}

Given a complete Riemannian manifold $(M^{n},g)$ with \eqref{cond:AF} and \eqref{cond:HC'}, it follows from Theorem \ref{T:cone} and Theorem \ref{T:infra} that we have a fibration on the end of $M$. Moreover, we prove

\begin{thm}\label{T:nil}
Let $(M^n,g,p)$ be a complete Riemannian manifold with \eqref{cond:AF} and \eqref{cond:HC'}, then the fiber of $f$ obtained in Theorem \ref{T:infra} is a nilmanifold.
\end{thm}

\begin{proof}
Given a point $x \in M\backslash K$, we denote the fiber through $x$ by $F$ and the induced metric by $g_F$. With respect to $g_F$, we fix a geodesic loop $\gamma$ based at $x$ such that $l \coloneqq L(\gamma) \le \kappa r$. Moreover, we assume that $\gamma$ is homotopic to a geodesic loop $\sigma$, with respect to $g$. In particular, $L(\sigma) \le l$.

Since $\gamma$ and $\sigma$ are homotopic, it follows from \cite[$6.2.1$]{BK81} that
\begin{align} \label{eq:hol001}
|\mathbf r(\sigma)-\mathbf r(\gamma)| \le CL(\gamma)L(\sigma)r^{-2}K(r/2) \le Cr^{-1}lK(r/2).
\end{align}

For any unit vector $V$ tangent to $F$, we denote the parallel transports of $V$ along $\gamma(t)$ with respect to $g_F$ and $g$ by $V(t)$ and $\bar V(t)$, respectively. It follows from condition (iii) of Theorem \ref{T:infra} that 
\begin{align*}
|\na_{\dot \gamma(t)} V(t)| \le Cr^{-1}.
\end{align*}
Since $V(0)=\bar V(0)=V$, we have
\begin{align}  \label{eq:hol002}
|V(l)-\bar V(l)| \le \int_{0}^{l}\frac{d}{dt}|V(t)-\bar V(t)|\,dt  \le \int_{0}^{l}|\na_{\dot \gamma(t)} V(t)|\,dt \le Cr^{-1}l.
\end{align}

Therefore, it follows from \eqref{eq:hol002} that
\begin{align*} 
|\mathbf r_F(\sigma)(V)-\mathbf r(\sigma)(V)| \le Cr^{-1}l.
\end{align*}
where $\mathbf r_F$ is the rotational part with respect to $g_F$.

In addition, by \eqref{cond:HC'} and \eqref{eq:hol001} we have
\begin{align} \label{eq:hol004}
\|\mathbf r_F(\sigma)\| \le Cr^{-1}L(\sigma).
\end{align}

\eqref{eq:hol004} indicates that all small geodesic loops of $F$ have small rotational parts.   Therefore, we conclude that $F$ must be a nilmanifold, see \cite[Chapter $3,4$]{BK81}.
\end{proof}

For a complete Riemannian manifold $(M^{n},g)$ with \eqref{cond:AF} and \eqref{cond:HC'}, it follows from Theorem \ref{T:nil} that the end of $M$ is diffeomorphic to $X \times (0,\infty)$ for some closed manifold $X$. Moreover, there exists a fibration
\begin{align*} 
f: X \longrightarrow S(\infty)
\end{align*}
with fiber $F$ a nilmanifold. $X$ is called the \textbf{boundary} of $M$.

\begin{rem}\label{rem:torus}
From \emph{\cite[Lemma $1.4$]{Rong96}}, $F$ is a torus if the fundamental group of $X$ is finite.
\end{rem}

\section{Analysis of the short bases}

\subsection{The basis of the fundamental pseudo-group}

Let $(M^{n},g)$ be a complete Riemannian manifold with \eqref{cond:AF} and \eqref{cond:HC'}. In this section, we demonstrate a process to construct a basis of the fundamental pseudo-group provided that all small geodesic loops are almost translational.

For any point $q \in M^n$ such that $r=r(q)$ is large, we consider the pseudo-group $ \Gamma(q,\hat \kappa r)$ and set 
\begin{align}  \label{eq:k0} 
{\rho_0}= \hat \kappa r \quad \text{and} \quad \theta = \hat \kappa C_H
\end{align}
where $\hat \kappa \in (0,\kappa)$ is a small parameter such that
\begin{align*}
\theta<\frac{1}{100} \quad \text{and} \quad  \Lambda^2 {\rho_0}^2 \le \frac{\theta}{100},
\end{align*}
where $C_H$ and $\kappa$ are the constants in condition \eqref{cond:HC'} and $\Lambda^2 = \max_{B(q,{\rho_0})} |Rm|$. 

We first recall the definition of the normal basis from \cite[Definition $4.1.1$]{BK81}.

\begin{defn}
By induction over $n$, the $\lambda$-normal bases for $\R^n$ are defined as follows:
\begin{enumerate}[label=(\roman*)]
\item Any basis for $\R^1$ is $\lambda$-normal for each $\lambda \ge 1$. 

\item A basis $\{\gamma_1,\gamma_2,\cdots,\gamma_n\}$ for $\R^n$ is $\lambda$-normal if it satisfies:
\begin{align*} 
|\gamma_i'| \le |\gamma_i| \le \lambda |\gamma_i'|
\end{align*}
where $\displaystyle \gamma_i'=\gamma_i-\frac{\la \gamma_i,\gamma_1 \ra}{\la \gamma_1, \gamma_1 \ra} \gamma_1 \, (i=2,3,\cdots,n)$ are the projections of $\gamma_i$ into $\{\gamma_1\}^{\perp}$ and
$\{\gamma'_2,\gamma'_3,\cdots,\gamma_n'\}$ is a $\lambda$-normal basis for $\{\gamma_1\}^{\perp}=\R^{n-1}$.
\end{enumerate}
\end{defn}

To analyze $\Gamma(q,\rho_0)$, we introduce the following definition.

\begin{defn} \label{def:basis}
A finite set $\{c_1,c_2,\cdots,c_m\} \subset \Gamma(q,\rho_0)$ is called a \textbf{short basis of radius $r_0$} if it satisfies:
\begin{enumerate}[label=(\roman*)]
\item $\{c_1,c_2,\cdots,c_m\}$ is a $\lambda$-normal basis for $\R^m$ with $\lambda \le 2$.

\item Each $c \in \Gamma(q,\rho_0)$ with $|\mathbf t(c)| \le r_0$ has a unique representation $c=c_1^{l_1}*c_2^{l_2}*\cdots * c_m^{l_m}$ for $l_i \in \mathbb Z$.

\item For $1 \le i <j \le m$, there exist structure constants $k^{ij}_u \in \mathbb Z$ for $1 \le u \le i-1$ such that
\begin{align*}
[c_i,c_j]=c_1^{k^{ij}_1}*c_2^{k^{ij}_2}*\cdots * c_{i-1}^{k^{ij}_{i-1}},
\end{align*}
where $[c_i,c_j]:=c_i*c_j*c_i^{-1} *c_j^{-1}$ is the commutator.
\end{enumerate}
\end{defn}

To construct a short basis, we have the following proposition similar to \cite[Proposition $3.5$]{BK81}.

\begin{prop}\label{P:norm}
For any geodesic loops $\alpha,\beta \in \Gamma(q,{\rho_0})$ with $| \mathbf t (\alpha) |, | \mathbf t (\beta)| \le \frac{1}{3}{\rho_0}$ we have
\begin{enumerate}[label=(\roman*)]
\item $\displaystyle\| \mathbf r (\alpha) \| \le \frac{\theta}{{\rho_0}}\, |\mathbf t (\alpha)|$,

\item $\displaystyle | \mathbf t (\alpha * \beta) -\mathbf t (\alpha) - \mathbf t (\beta)| \le 2\,\frac{\theta}{ {\rho_0}} \,|\mathbf t (\alpha)|\,|\mathbf t (\beta)| $,

\item $ \displaystyle | \mathbf t ([\beta ,\alpha])| \le 3\,\frac{\theta}{ {\rho_0}}\, |\mathbf t (\alpha)|\,|\mathbf t (\beta)| $.
\end{enumerate}
\end{prop}

\begin{proof}
(i) is exactly \eqref{cond:HC'}. Based on (i), (ii) and (iii) follow from \cite[Proposition $2.3.1$ (ii),(iii)]{BK81}.
\end{proof}

Now we have the following definition similar to \cite[Definition $4.2.1$]{BK81}. Notice that our definition does not require the denseness (i.e. \cite[Definition $4.2.1$ (ii)]{BK81}).

\begin{defn} \label{def:sba}
A finite set $T \subset \R^n$ is called a $\theta$-translational subset of radius $\rho$ if it satisfies:
\begin{enumerate}[label=(\roman*)]
\item $0 \in T$; if $c \in T$, then $|c| \le \rho$.

\item For all $a,b\in T$, $|a+b| \le \rho(1-\theta)$ a product $a*b \in T$ is defined and for each $a \in T$, $|a| \le \rho(1-\theta)$ there exists a unique $a^{-1} \in T$ such that $a * a^{-1}=a^{-1}*a=0$.

\item Associativity  $(a*b)*c=a*(b*c)$ holds, if the existence of all products involved follows from (ii).

\item The product satisfies
\begin{align} 
|a*b-a-b|\le \frac{2\theta}{\rho} |a|\,|b| \quad \text{and} \quad [[a,b]| \le  \frac{3\theta}{\rho} |a|\,|b|. \label{eq:tr1}
\end{align}
\end{enumerate}
\end{defn}

Now we set for $1\le i\le n$,
\begin{align}
\rho_i \coloneqq \frac{1}{3}12^{1-i} \rho_0 \quad \text{and} \quad \bar \rho \coloneqq \frac{1}{3}(1+\theta)^{-n-1}12^{-n}2^{-\frac{n^2}{2}} \rho_0.\label{eq:chorhoab}
\end{align}

It follows from Proposition \ref{P:norm} that the set
\begin{align} \label{eq:Tr}
T_1 \coloneqq \{ a= \mathbf t(\alpha) \mid \, \alpha \in \Gamma(q,\rho_0),\,|\mathbf t(\alpha)| \le \rho_1\}
\end{align}
with product $\mathbf t(\alpha) * \mathbf t(\beta)=\mathbf t( \alpha * \beta)$ is a $\theta$-translational subset of radius $\rho_1$. It is clear that the map $\alpha \to \mathbf t(\alpha)$ is a group isomorphism. For this reason, for any geodesic loop $c$, we also use $c$ to denote $\mathbf t(c)$.

We need the following lemmas, whose proofs can be found in \cite[Proposition $4.2.3$, Proposition $4.1.4$(ii)]{BK81}.

\begin{lem}\label{L:esta}
Let $a,b \in T_1$ with $|a|,|b|,|a-b|\le (1-3\theta)\rho_1$, then
\begin{align*} 
|a^{-1}*b-(b-a)| \le \frac{\theta(1+\theta)}{\rho_1}|a||b-a| \le \theta(|a|+|b|). 
\end{align*}
\end{lem}

\begin{lem}\label{L:esta1}
Let $\{c_1,c_2,\cdots,c_m\}$ be a short basis of $\Gamma(q,\rho_0)$. If $\sum_{i=1}^m |l_ic_i| \le  2^{-\frac{1}{2}m^2}\rho_1$, then
\begin{align*}
\left|\sum_{i=1}^m l_ic_i\right |\ge& 2^{-\frac{1}{2}m^2} \sum_{i=1}^m |l_ic_i|, \\
\left|c_1^{l_1}*c_2^{l_2}*\cdots *c_m^{l_m}-\sum_{i=1}^m l_ic_i \right | \le& \theta \left|\sum_{i=1}^m l_ic_i \right | \le \frac{\theta}{1-\theta}|c_1^{l_1}*c_2^{l_2}*\cdots *c_m^{l_m}|. 
\end{align*}
\end{lem}

Now we can a choose a short basis of $T_1$ following the construction in \cite[Section $4.3$]{BK81}. We first choose $c_1 \in T_1$ to be a shortest nonzero element and set 
\begin{align*}
|c_1|=\sigma_1.
\end{align*}
Notice that $\sigma_1=2\text{inj}(q)$ and the choice of $c_1$ may not be unqiue. With $c_1$ fixed, for any $c \in T_1$ satisfying $|c| \le (1-4\theta) \rho_1-2\sigma_1$, it follows from \cite[Section $4.3.2$]{BK81} that there exists a unique representation $\tilde c$ such that 
\begin{align}\label{eq:cho2}
\la c_1 ,\tilde c \ra >0,\quad \la c_1,c_1^{-1}*\tilde c\ra \le 0\quad \text{and} \quad 
c=c_1^k*\tilde c
\end{align}
for some integer $k$ satisfying
\begin{align}\label{eq:chok}
|k| \le \frac{1}{1-\theta}\frac{\la c_1,c\ra}{|c_1|^2}.
\end{align}

Now we denote the map from $c$ to $\tilde c$ by $Q_1$ and the image by $\tilde T_1$. In addition, we define $T'_1$ to be the image of the map
\begin{align}\label{eq:cho3}
 \tilde c \in \tilde T_1 \overset{P_1}{\longrightarrow} c' \coloneqq \tilde c-\frac{\la \tilde c,c_1 \ra}{\la c_1,c_1 \ra} c_1.
\end{align}
Notice that here $P_1$ is defined to be the orthogonal projection to $\{c_1\}^{\perp}$. It follows from \cite[Section $4.3.3$]{BK81} that the map $P_1$ is injective and
 \begin{align}\label{eq:cho4}
|c'| \le |\tilde c| \le \lambda |c'|\quad \text{for} \quad \lambda \le (\frac{1}{2}-2\theta)^{-\frac{1}{2}}.
\end{align}

Next, we define 
 \begin{align}\label{eq:cho4a}
T_2 \coloneqq \{c' \in T'_1 \mid\, |c'| \le \rho_2 \}.
\end{align}
If $T_2 $ is empty, we stop the process. Otherwise, we define a product in $T_2$ as in \cite[Proposition $4.4.1$]{BK81}. For any $a',b' \in T_2$ such that $|a'+b'| \le (1-\theta) \rho_2$, there exist unique $\tilde a$ and $\tilde b$, which are preimages with respect to $P_1$, satisfying 
\begin{align*}
|\tilde a| \le \lambda |a'| \quad \text{and} \quad |\tilde b| \le \lambda |b'|.
\end{align*}
If we set $c=\tilde a*\tilde b$, then the product in $T_2$ is defined as 
 \begin{align}\label{eq:cho5}
a' * b' \coloneqq c'.
\end{align}

With definition \eqref{eq:cho5}, it follows from \cite[Proposition $4.4.2$]{BK81} that $T_2 \subset \{c_1\}^{\perp}=\R^{n-1}$ is a $\theta$-translational subset of radius $\rho_2$. Now we choose $c_2 \in \tilde T_1$ such that $P_1(c_2)$ is the shortest element among $T_2$. We also define 
 \begin{align*}
\sigma_2=|P_1(c_2)|.
\end{align*}
Notice that the choice of $c_2$ may not be unique.

By continuing this process, there exists an integer $1 \le m <n$ so that we can define the sets $T_2 \supset T_3 \supset \cdots \supset T_m$ and the generators $\{c_2,c_3,\cdots,c_m\}$ by $m-1$ induction steps. We call $\{c_1,c_2,\cdots,c_m\}$ obtained in this way a \textbf{standard short basis} of the fundamental pseudo-group $\Gamma(q,\rho_0)$.

To generalize the standard short basis, we fix a parameter
\begin{align} \label{eq:cho8}
 \theta_1=\frac{\theta}{100}.
\end{align}

With the same definition of $T_1$ as in \eqref{eq:Tr}, we first choose any $c_1 \in T_1$ such that
\begin{align*}
|c_1| \le (1+\theta_1) \sigma_1,
\end{align*}
where $\sigma_1$ is the smallest length in $T_1$. After $c_1$ is chosen, we define as in \eqref{eq:cho2} and \eqref{eq:cho3} the maps $Q_1$ and $P_1$ such that \eqref{eq:cho4} holds. In fact, since the error parameter $\theta_1$ is much smaller than $\theta$ by our assumption \eqref{eq:cho8}, the same proof of \cite[Section $4.3.3$]{BK81} shows that $P_1$ is injective. Next we define the set $T_2$ as in \eqref{eq:cho4a} with its product \eqref{eq:cho5}, then the same proof of \cite[Proposition $4.4.2$]{BK81} shows that $T_2$ is a $\theta$-translational subset with radius $\rho_2$.

Next, we choose $c_2 \in T_1$ satisfying for some $ k \in \{-1,0,1\}$
\begin{align}
c_1^{k}*c_2 &\in \tilde T_1,\,\label{eq:cho9a} \\
-\theta \le \cos \angle(c_1&,c_2) \le \lc \frac{1}{2}+9\theta \rc^{\frac{1}{2}}, \label{eq:cho9b}\\
|P_1(c_2)| &\le (1+\theta_1) \sigma_2 \label{eq:cho10},
\end{align}
where $\sigma_2$ is the smallest length in $T_2$. Notice that by \eqref{eq:cho4} (see \cite[Section $4.3.3$]{BK81} for details), \eqref{eq:cho9b} is automatically satisfied if \eqref{eq:cho9a} holds for $k=0$.

If we set $c^{(2)}_2=P_1(c_2)$, then we define $\bar T_2 \coloneqq T_2 \cup \{c^{(2)}_2\}$. In addition, for any $a'\in  T_2$ such that $P_1(a)=a'$ for $a \in \tilde T_1$, we define
\begin{align*} 
c^{(2)}_2 * a'=P_1(Q_1(c_2*a)).
\end{align*}
By the same proof of  \cite[Proposition $4.4.2$]{BK81}, $\bar T_2 $ is a $\theta$-translational subset of radius $\rho_2$. 

Now we can consecutively define $c_1,c_2,\cdots,c_m$ and the sets $\bar T_1=T_1 \supset \bar T_2 \supset \cdots \supset \bar T_m$ such that for $1 \le i \le m$,
\begin{align*}
|c^{(i)}_i| \le (1+\theta_1)\bar \sigma_i,
\end{align*}
where $c^{(i)}_i$ is the projection of $c_i$ to $\bar T_i$ and $\bar \sigma_i$ is the smallest length in $\bar T_i$. For any $1 \le i \le m-1$,
\begin{align*}
(c^{(i)}_i)^k*c_{i+1}^{(i)} \in \widetilde{ \bar T_i}
\end{align*}
for some $k \in \{-1,0,1\}$ and
\begin{align} \label{eq:cho10d}
-\theta \le \cos \angle(c^{(i)}_i,c^{(i)}_{j}) &\le \lc \frac{1}{2}+9\theta \rc^{\frac{1}{2}}.
\end{align}
for any $i+1 \le j \le m$. In addition, $\bar T_i \subset \R^{n-i+1}$ is a $\theta$-translational subset of radius $\rho_i$.

Now we call a basis $\{c_1,c_2,\cdots,c_m\}$ chosen in this way a \textbf{generalized standard short} $\boldsymbol{\theta_1}$-\textbf{basis} of the fundamental pseudo-group $\Gamma(q,\rho_0)$. 

The following proposition is proved verbatim as \cite[Theorem $4.5$]{BK81} by adjusting $\theta$ slightly.

\begin{prop} \label{P:basis}
If $\{c_1,c_2,\cdots,c_m\}$ is a generalized standard short $\theta_1$-basis of $\Gamma(q,\rho_0)$, then
\begin{enumerate}[label=(\roman*)]
\item $\{c_1,c_2,\cdots,c_m\}$ is a $\lambda$-normal basis for $\lambda \le (\frac{1}{2}-2\theta)^{-\frac{1}{2}}$.

\item If $\sum |l_ic_i| \le (1-2\theta)^2\rho_1$, then $c_1^{l_1}*c_2^{l_2}*\cdots * c_m^{l_m}$ is defined and associativity holds.

\item Each $c \in T_1$ with $|c| \le \bar \rho$ has a unique representation $c=c_1^{l_1}*c_2^{l_2}*\cdots * c_m^{l_m}$. For the images $c^{(k)}$ of $c$ under the iterated projections $T_1 \to T_2 \to T_3\to \cdots$, we have
\begin{align*}
c^{(k)}=((c_k)^{(k)})^{l_k}*((c_{k+1})^{(k)})^{l_{k+1}}*\cdots * ((c_m)^{(k)})^{l_m}.
\end{align*}

\item For $c \in T_1$ with $|c| \le \bar \rho$ and $c=c_1^{l_1}*c_2^{l_2}*\cdots * c_m^{l_m}$,
\begin{align*}
\left |c-\sum_{i=1}^m l_ic_i\right | \le \theta \left |\sum_{i=1}^m l_ic_i\right | \le \frac{\theta}{1-\theta} |c|.
\end{align*}

\item For $1 \le i <j \le m$, there exist structure constants $k^{ij}_u \in \mathbb Z$ for $1 \le u \le i-1$ such that
\begin{align*}
[c_i,c_j]=c_1^{k^{ij}_1}*c_2^{k^{ij}_2}*\cdots * c_{i-1}^{k^{ij}_{i-1}}.
\end{align*}
\end{enumerate}
\end{prop}

In particular, Proposition \ref{P:basis} implies that a generalized standard short $\theta_1$-basis is a short basis of radius $\bar \rho$ by our definition \eqref{def:basis}. Since the generalized standard short $\theta_1$-basis is less rigid than the standard short basis, we can construct it by sliding, as will be shown in Section $4.3$.

Notice that the generalized standard short $\theta_1$-bases are not unique. However, we have

\begin{prop} \label{P:basis1}
There exists a constant $C_0=C_0(n)<\infty$ such that there are at most $C_0$ generalized standard short $\theta_1$-bases.
\end{prop}

\begin{proof}
We only need to prove that for the standard short $\theta_1$-bases, there are finitely many choices for $c_1$. Indeed, if $c_i$ is fixed, the corresponding $\bar T_{i+1}$ is a $\theta$-translational subset of radius $\rho_{i+1}$.

We set $S=\{c \in T_1\, \mid \sigma_1 \le |c| \le (1+\theta_1)\sigma_1\}$. Then there exists a constant $C=C(n)$ such that if $|S|>C$, we can find two different elements $a,b\in S$ satisfying
\begin{align*}
|b-a| \le \frac{1}{2}\sigma_1.
\end{align*}
It follows from Lemma \ref{L:esta} that
\begin{align*}
|a^{-1}*b|  \le |a^{-1}*b-(b-a)|+|b-a| \le \theta(|a|+|b|)+ \frac{1}{2}\sigma_1 \le 2\theta(1+\theta_1)\sigma_1+ \frac{1}{2}\sigma_1 \le \frac{2\sigma_1}{3}.
\end{align*}
However, it contradicts the definition of $\sigma_1$.
\end{proof}

\subsection{Estimates of geodesic loops}
From now on, we fix a complete Riemannian manifold $(M^{n},g,p)$ with \eqref{cond:AF} and \eqref{cond:SHC}. Moreover, we define for any $t>0$,
\begin{align} 
K_0(t) \coloneqq \int_t^\infty \frac{K(s)}{s}\,ds \quad \text{and} \quad K_1(t) \coloneqq \int_t^\infty \frac{K(s)}{s^2}\,ds \label{def:k0}.
\end{align}

From \eqref{cond:AF}, there exists a constant $A_0>0$ such that for any $q$ with $r=r(q) \ge A_0$ and any $\rho \in (100,r/400)$, 
\begin{align*}
\rho^2 |Rm| \le \ep_0 \quad \text{on} \quad B(q,100\rho),
\end{align*}
where $\ep_0$ is the same constant in Proposition \ref{P:es3}.

We consider an arc-length parametrized curve $\{\beta(t): t \in [-\ep,+\ep]\}$ such that there exists a small geodesic loop $\gamma$ at $q \coloneqq \beta(0)$. We assume that $r(q)$ is large and the length $L(\gamma)$ is small enough such that the sliding of $\gamma$ along $\beta$ is well defined and we denote the geodesic loop at $\beta(t)$ by $\gamma_t$.

Now we prove

\begin{thm}\label{T:hol}
Let $l(t)$ and $\mathbf r(t)$ be the length and rotational part of $\gamma_t$ respectively, then for $t \in (-\ep,+\ep)$
\begin{enumerate}[label=(\roman*)]
\item $|l'(t)| \le  |\mathbf r(t)-I|$.

\item For any unit vector field $X(t)$ parallel along $\beta(t)$,
\begin{align*}
\left| |\mathbf r(t)X(t)-X(t)|' \right|\le l(t)\max_{\gamma_t} |Rm|.
\end{align*}
\end{enumerate}
\end{thm}
 
\begin{proof}
(i):
We assume that $\tau$ is corresponding local isometry of $\gamma$ and $\tilde \beta(t)$ is the lift of $\beta(t)$ on $T_qM$. For any $t \in (-\ep,+\ep)$, we set $W(t)$ to be the initial tangent vector of the geodesic from $\tilde \beta(t)$ to $\tau(\tilde \beta(t))$. From our definition, $l(t)=d_{\hat g}(\tilde \beta(t),\tau(\tilde \beta(t)))$ and hence
\begin{align*}
l'(t)=-\hat g_{\tilde \beta(t)}(W(t),\tilde \beta'(t))+\hat g_{\tau(\tilde \beta(t))}((d\tau)_*(\tilde \beta'(t)),(d\tau)_*(W_1(t))),
\end{align*}
where $W_1(t)=\mathbf r(t)(W(t))$. Since $(d\tau)_*$ is an isometry, we have
\begin{align*}
|l'(t)|=|\la \beta'(t), \mathbf r(t)W(t)-W(t)\ra| \le  |\mathbf r(t)-I| .
\end{align*}

(ii): 
We construct a vector field $X(s,t)$ such that $X(0,t)=X(t)$ and for fixed $t$, $X(s,t)$ is parallel along $\gamma_t$. By direct computations,
\begin{align*}
\abs{ \partial_s|\partial_t X(s,t)|} \le |\partial_s\partial_t X(s,t) |=|\partial_t\partial_s X(s,t)+Rm(\partial_s,\partial_t)X(s,t)| \le \max_{\gamma_t} |Rm|.
\end{align*}
By integration,
\begin{align*}
|\partial_t X(1,t)| \le |\partial_t X(0,t)|+l(t)\max_{\gamma_t} |Rm| \le l(t)\max_{\gamma_t} |Rm|.
\end{align*}
Consequently,
\begin{align*}
\left||\mathbf r(t)X(t)-X(t)|' \right|=\left||X(1,t)-X(0,t)|' \right| \le |\partial_t X(1,t) |\le l(t)\max_{\gamma_t} |Rm|.
\end{align*}
\end{proof}

With the same proof as Lemma \ref{lem:holaa1}, we can improve \eqref{cond:SHC}.

\begin{lem}\label{lem:holaa}
Let $(M^{n},g,p)$ be a complete Riemannian manifold with \eqref{cond:AF} and \eqref{cond:SHC}. Then there exists a positive function $\ep_1(r)$ with $\ep_1(r) \to 0$ if $r \to \infty$ such that 
\begin{align*} 
\|\mathbf r (\gamma_x)\| \le \ep_1(r)\frac{L(\gamma_x)}{r}
\end{align*}
for any $x$ outside a compact set and any geodesic loop $\gamma_x$ based at $x$ with length smaller than $\kappa r$, where $\kappa$ is the constant in \eqref{cond:SHC}.
\end{lem}

We fix a geodesic ray $\{\alpha(t),t \ge 0\}$ starting from $p$ and estimate the lengths and rotational parts of the slidings of a geodesic loop.

\begin{thm}\label{T:esray}
There exist constants $C_1>1$ and $A_1>0$ such that for any geodesic loop $\gamma$ based at $\alpha(s)$ with $s \ge A_1$ and $|\mathbf t(\gamma)| \le C_1^{-1}s$, if we denote the length and rotational part of the sliding of $\gamma$ at $\alpha(t)$ by $l(t)$ and $\mathbf r(t)$ respectively, then for any $t \ge s$,
\begin{enumerate}[label=(\roman*)]
\item $\|\mathbf r(t)\| \le C_1 l(t)K_1(t/2)$.
\item $\displaystyle l(s) \exp\lc -C_1\int_{s}^t K_1(z/2)\,dz \rc \le l(t) \le l(s) \exp\lc C_1\int_{s}^t K_1(z/2)\,dz \rc$.
\end{enumerate}
Here, the function $K_1$ is defined in \eqref{def:k0}.
\end{thm}

\begin{proof}
First, we claim that if $C_1$ and $A_1$ are sufficiently large, $l(t) \le \kappa t/2$ for any $t \ge s$. Indeed, we assume $C_1 \ge 4/\kappa$ and set $s_0 \ge s$ to be the largest number of $t$ such that $l(t) \le \kappa t/2$. If $s_0$ is finite, then $l(s_0)=\kappa s_0/2$. It follows from Theorem \ref{T:hol} (i) and Lemma \ref{lem:holaa} that 
\begin{align}
|l'(t)| \le \|\mathbf r(t)\| \le \frac{\ep_1(t)}{t}l(t) \label{T410:x1}
\end{align}
for any $t \in [s,s_0]$, where $\ep_1(t)$ is a positive function such that $\lim_{t \to \infty} \ep_1(t)=0$. If $A_1$ is sufficiently large, we have $\ep_1(t) \le 1$ for any $t \ge s$. Hence, it follows from \eqref{T410:x1} that
\begin{align*}
l(s_0) \le \frac{s_0}{s}l(s) \le \kappa s_0/4,
\end{align*}
which is a contradiction. In particular, we conclude for any $t \ge s$,
\begin{align} \label{eq:esray1a}
l(t) \le \frac{t}{s} l(s) \le t l(s).
\end{align}

For any $t \ge s$, we choose a unit parallel vector field $X$ along $\alpha$ such that 
\begin{align*}
|\mathbf r(t)-I|=|\mathbf r(t)X(t)-X(t)|.
\end{align*}
Then it follows from Theorem \ref{T:hol} and \eqref{cond:SHC} that
\begin{align} \label{eq:esray2}
|l'(t)| \le |\mathbf r(t)-I| \le \int_t^{\infty} l(z) \max_{\gamma_z} |Rm|\,dz \le 4\int_t^{\infty} l(z)  \frac{K(z/2)}{z^2} \,dz.
\end{align}

Combining \eqref{eq:esray1a} and \eqref{eq:esray2}, it follows from the next lemma that for any $t \ge s$,
\begin{align*} 
l(t) \le C l(s).
\end{align*}

From \eqref{eq:esray2}, we have
\begin{align*}
|l'(t)| \le |\mathbf r(t)-I|\le Cl(t) \int_t^{\infty} \frac{K(z/2)}{z^2}\,dz \le Cl(t)K_1(t/2)
\end{align*}
and by integration,
\begin{align*}
l(s) \exp\lc -C\int_{s}^t K_1(z/2)\,dz \rc \le l(t) \le l(s) \exp\lc C\int_{s}^t K_1(z/2)\,dz \rc.
\end{align*}
\end{proof}

We prove the following lemma which plays an important role in the proof of Theorem \ref{T:esray}.

\begin{lem}\label{lem:ode}
Let  $\{x(t),t\in [1,\infty)\}$ be a positive differentiable function satisfying the following inequalities
\begin{align*}
x(t) \le t x(1) \quad \text{and}\quad x'(t) \le \int_t^{\infty} x(s)\frac{k(s)}{s^2}\,ds 
\end{align*}
for any $t \ge 1$, where $\{k(t),t\in [1,\infty)\}$ is a nonincreasing positive function such that
\begin{align}
\int_1^{\infty}\frac{k(s)}{s}\,ds <\frac{1}{10}. \label{eq:ode0}
\end{align}
Then there exists a constant $C>0$ which depends only on $k(s)$ such that for any $t \ge 1$
\begin{align*}
x(t) \le C x(1).
\end{align*}
\end{lem}

\begin{proof}
We assume $x(1)=1$ and define a sequence $\{t_0=1 <t_1<t_2<\cdots\}$ by
\begin{align*}
t_i=\inf\{t \in [1,\infty): \forall s \ge t,x(s) \le 2^{-i}s\}.
\end{align*}
Notice that the sequence $\{t_i\}$ exists since $x'(t)$ is sublinear from
\begin{align*}
x'(t) \le \int_t^{\infty} x(s)\frac{k(s)}{s^2}\,ds \le \int_t^{\infty}\frac{k(s)}{s}\,ds.
\end{align*}

From continuity, we have $x(t_i)=2^{-i}t_i$. For any $1 \le a<b$, we have
\begin{align*}
&x(b)-x(a)  \\
=& \int_a^b x'(t) \,dt \le \int_a^b \int_t^{\infty} x(s)\frac{k(s)}{s^2}\,ds\,dt  \\
= & (b-a) \int_b^{\infty} x(s)\frac{k(s)}{s^2}\,ds+\int_a^b (t-a)x(t) \frac{k(t)}{t^2}\,dt  \\
\le &(b-a) \int_a^{\infty} x(s)\frac{k(s)}{s^2}\,ds+\int_a^b x(t) \frac{k(t)}{t}\,dt.
\end{align*}

Therefore, it follows from the Gronwall's theorem that
\begin{align}
x(b) \le \lc x(a)+(b-a) \int_a^{\infty} x(s)\frac{k(s)}{s^2}\,ds \rc \exp \lc\int_a^b \frac{k(t)}{t}\,dt \rc.
 \label{eq:ode2}
\end{align}

If we set $\sigma_i=\int_{t_i}^{t_{i+1}} \frac{k(t)}{t}\,dt$, then by choosing $a=t_i$ and $b=t_{i+1}$ in \eqref{eq:ode2} we obtain
\begin{align*}
2^{-i-1}t_{i+1} \le \lc 2^{-i}t_{i}+(t_{i+1}-t_i)\sum_{j=i}^{\infty}2^{-j}\sigma_j \rc e^{\sigma_i}.
\end{align*}
After simplification we have
\begin{align}
t_{i+1} \le \frac{2^{-i}-\sum_{j=i}^{\infty}2^{-j}\sigma_j}{2^{-i-1}e^{-\sigma_i}-\sum_{j=i}^{\infty}2^{-j}\sigma_j}t_i \le \frac{ 2t_i}{1-\sigma_i-\sum_{j=0}^{\infty}2^{1-j}\sigma_{i+j}}.
 \label{eq:ode4}
\end{align}
Notice that by our assumption \eqref{eq:ode0}, the denominator above is a finite positive number.

It is clear that $\prod_{i=0}^{\infty} \lc 1-\sigma_i-\sum_{j=0}^{\infty}2^{1-j}\sigma_{i+j}\rc$ is finite. Indeed, this follows from
\begin{align*}
\sum_{i=0}^{\infty} \ln \lc 1-\sigma_i-\sum_{j=0}^{\infty}2^{1-j}\sigma_{i+j} \rc >-\infty
\end{align*}
since
\begin{align*}
\sum_{i=0}^{\infty} \lc \sigma_i+\sum_{j=0}^{\infty}2^{1-j}\sigma_{i+j} \rc <5\sum_{i=0}^{\infty}\sigma_i<\infty
\end{align*}
from our assumption \eqref{eq:ode0}.
By \eqref{eq:ode4}, there exists a constant $C>0$ such that
\begin{align*}
t_i \le C 2^i
\end{align*}
and by our definition of $t_i$
\begin{align*}
x(t) \le 2C
\end{align*}
for any $t \ge 1$.
\end{proof}

Next, we estimate the change of angle between two geodesic loops.

\begin{prop}\label{p:angle}
Given two geodesic loops $\gamma^1$ and $\gamma^2$ at $\alpha(t_0)$ for $t_0 \ge A_1$ such that $|\mathbf t(\gamma^i)| \in [a_0^{-1},a_0]$, for some constant $a_0 \in (1,C^{-1}_1 t_0)$, we denote their slidings at $\alpha(t)$ by $\gamma_t^1$ and $\gamma_t^2$ respectively and set $c_i(t)=\mathbf t(\gamma_t^i)$ for $i=1,2$. Then there exist constants $\Theta \ge 0$ and $C>0$ such that for any $t \ge t_0$,
\begin{align*}
|\angle (c_1(t),c_2(t))-\Theta | \le CK_0(t/4).
\end{align*}
\end{prop}

\begin{proof}
If we set $\mathbf r_i(t)$ to be the rotational part of $\gamma^i_t$ for $i=1,2$, then it follows from Theorem \ref{T:esray} that 
\begin{align*}
|\mathbf r_i(t)-I| \le a_1K_1(t/2) \le  2a_1\frac{K(t/2)}{t} \quad \text{and} \quad a_1^{-1} \le |\mathbf t(\gamma_t^i)| \le a_1
\end{align*}
for some constant $a_1>1$.

Now we define the function $\Theta(t) \coloneqq \angle (\dot \gamma^1_t(0),\dot \gamma^2_t(0))$. For any point $\alpha(t)$, if we set $\rho=\frac{t}{100}$ and $D=1+\frac{\lambda}{1000}>1$ where $\lambda$ is the constant in Lemma \ref{L:es2}, then it follows from Proposition \ref{P:es4} that for $t \le s\le Dt$,
\begin{align*}
|\Theta(s)-\Theta(t)| \le CK(t/2).
\end{align*}

It is clear that for any $s \ge t$, if $D^{k-1}t \le s < D^kt$ for some integer $k \ge 1$, then
\begin{align*}
|\Theta(s)-\Theta(t)| \le& \sum_{i=1}^{k-1} |\Theta(D^it)-\Theta(D^{i-1}t)|+|\Theta(s)-\Theta(D^{k-1}t)| \\
\le & C\sum_{i = 1}^{\infty} K(D^{i-1}t/2) \le C\int_{t/4}^{\infty} \frac{K(z)}{z}\,dz=CK_0(t/4).
\end{align*}
Therefore, there exists a constant $\Theta \ge 0$ such that
\begin{align*}
|\Theta(t)-\Theta| \le CK_0(t/4).
\end{align*}

Moreover, since $\mathbf r_i(t)(\dot \gamma^i_t(0))=c_i(t)$, we have
\begin{align*}
|c_i(t)-\dot \gamma^i_t(0)| \le C|\mathbf r_i(t)-I| \le CK_1(t/2) \le C\frac{K_0(t)}{t}.
\end{align*}

Therefore, it is clear that
\begin{align*}
|\angle (c_1(t),c_2(t))-\Theta | \le CK_0(t/4).
\end{align*}
\end{proof}

\subsection{Construction of the short bases on the end}

In this section, we construct the short bases on the end. For simplicity, we define $\tilde K(t) \coloneqq \max\{K(t),K_0(t)\}$.

\begin{thm}\label{T:basis2}
Let $(M^n,g)$ be a complete Riemannian manifold with \eqref{cond:AF} and \eqref{cond:SHC}, there exist constants $A_2>0$, $\kappa_1>0$, $C_2>0$ and an integer $1 \le m <n$ such that for any $q$ with $r=r(q) \ge A_2$, we can find a short basis with radius $\kappa_1 r$, denoted by $\{c_1^q,c_2^q,\cdots,c_m^q\}$, of $\Gamma(q,\rho_0)$ satisfying
\begin{enumerate}[label=(\roman*)]

\item For any $1\le i\le m$, 
\begin{align} \label{eq:basis1}
\|\mathbf r (c^q_i)\| &\le C_2 r^{-1}K(r/2) \quad \text{and} \quad |L(c^q_i)-L_i| \le C_2 \tilde K(r/2)
\end{align}
for some constants $L_i>0$.
\item For any $1\le i<j\le m$,
\begin{align}\label{eq:basis2}
|\angle(c_i^q, c_j^q)-\Theta_{ij}| &\le C_2 \tilde K(r/4),
\end{align}
for some constants $\Theta_{ij}>0$.

\item The fundamental pseudo-group $\Gamma(q,\kappa_1 r)$ is abelian.
\end{enumerate}
\end{thm}

The proof of Theorem \ref{T:basis2} consists of two steps. 

\textbf{Step 1: Construction of the short bases on a geodesic ray}

Fix a geodesic ray $\alpha(t)$ starting from $p$, we set $T_1(t)$, $\sigma_1(t)$, etc. to be the corresponding elements in Section $4.1$ at $\alpha(t)$. We choose $q=\alpha(t^{(1)})$ where $t^{(1)}$ is a large constant to be determined later and set $\{c_{1,1},c_{1,2},\cdots,c_{1,k_1}\}$ to be all shortest elements in $\Gamma(q,\rho_1(t^{(1)}))$. Notice that $k_1 \le C_0$ by Proposition \ref{P:basis1}. For any $1 \le i\le k_1$ and $t \ge t^{(1)}$, we denote the slidings of $c_{1,i}$ along $\alpha(t)$ by $c_{1,i}(t)$. Now we define the subset $Z_1 \subset [t^{(1)},\infty)$ such that $t \in Z_1$ if and only if all shortest elements in $T_1(t)$ are contained in $\{c_{1,1}(t),c_{1,2}(t),\cdots,c_{1,k_1}(t)\}$.

\begin{lem}\label{LB1}
The set $Z_1$ is nonempty and open in $ [t^{(1)},\infty)$.
\end{lem}

\begin{proof}
 From the definition $t^{(1)} \in Z_1$ and hence $Z_1$ is nonempty. To prove that $Z_1$ is open, we only need to prove that $Z_1$ is open at $t^{(1)}$ and the general case is similar. Assuming the contrary, there exists a sequence $t_i \to (t^{(1)})^+$ such that at $\alpha(t_i)$ we can find a shortest element $b_i \in T_1(t_i)$ which is not contained in $\{c_{1,1}(t_i),c_{1,2}(t_i),\cdots,c_{1,k_1}(t_i)\}$. If we denote the sliding of $b_i$ to $q$ along $\alpha(t)$ by $e_i$, then by taking a subsequence if necessary, we assume that $e_i$ converges to a geodesic loop $e_{\infty} \subset T_1(t^{(1)})$. By continuity, $e_{\infty}$ has the smallest length in $T_1(t^{(1)})$ and hence $e_{\infty}$ is identical with $c_{1,s}$ for some $1\le s\le k_1$. It follows from Lemma \ref{L:esta} that
\begin{align*}
| e_{\infty}^{-1}*e_i-(e_i-e_{\infty})| \le \frac{\theta(1+\theta)}{\rho_1} |e_{\infty}|\,| e_i-e_{\infty}|.
\end{align*}
Since $| e_i-e_{\infty}| \to 0$, for sufficiently large $i$ we must $e_i=e_{\infty}$ since $e_{\infty}$ has the smallest length. From the transitivity of the sliding, $e_i$ is the sliding of $e_{\infty}$ if $i$ is large, which is a contradiction.
\end{proof}

If $Z_1 \ne [t^{(1)},\infty)$, we set $t_{1,2} \coloneqq \inf\{t \in [t^{(1)},\infty)\backslash Z_1\}$. Notice that it follows from Lemma \ref{LB1} that $t_{1,2} \notin Z_1$. It is clear from Theorem \ref{T:esray} (ii) that if $t^{(1)}$ is sufficiently large,
\begin{align}\label{LB1aa}
 \lc 1-\frac{\theta_1}{10}\rc^{\frac{1}{C_0}} \le \frac{|c_{1,i}(t)|}{|c_{1,j}(t)|} \le  \lc 1+\frac{\theta_1}{10}\rc^{\frac{1}{C_0}}
\end{align}
for any $1 \le i,j \le k_{1}$ and $t \ge t^{(1)}$, where the constant $C_0$ is the constant in Proposition \ref{P:basis1}. In particular, $\{c_{1,1}(t_{1,2}),c_{1,2}(t_{1,2}),\cdots,c_{1,k_1}(t_{1,2})\}$ are different elements in $T_1(t_{1,2})$ such that 
\begin{align*}
|c_{1,i}(t_{1,2})| < (1+\theta_1) \sigma_1(t_{1,2}).
\end{align*}

Since $t_{1,2} \notin Z_1$, we extend the collection $\{c_{1,1}(t_{1,2}),c_{1,2}(t_{1,2}),\cdots,c_{1,k_1}(t_{1,2})\}$ to a new collection $\{c_{1,1}(t_{1,2}),c_{1,2}(t_{1,2}),\cdots,c_{1,k_{1,2}}(t_{1,2})\}$ for $k_{1,2}>k_1$ such that if $k_1<i \le k_{1,2}$, $c_{1,i}(t_{1,2})$ is a shortest element in  $T_1(t_{1,2})$ and all shortest elements of $T_1(t_{1,2})$ are included. Now we define the subset $Z_{1,2} \subset [t_{1,2},\infty)$ such that $t \in Z_{1,2}$ if and only if all shortest elements in $T_1(t)$ are contained in $\{c_{1,1}(t),c_{1,2}(t),\cdots,c_{1,k_{1,2}}(t)\}$.

Similar to Lemma \ref{LB1}, $Z_{1,2}$ is open in $ [t_{1,2},\infty)$. If $Z_{1,2}=[t_{1,2},\infty)$, we stop here. Otherwise, we define the set $Z_{1,3}$ for some integer $k_{1,3}>k_{1,2}$. Notice that by Proposition \ref{P:basis1}, this process must end after finite steps. By redefining $t^{(1)}$ and $k_1$, we have proved

\begin{lem}\label{LB1a}
For any $t \ge t^{(1)}$, all shortest elements in $T_1(t)$ is contained in $\{c_{1,1}(t),c_{1,2}(t),\cdots,c_{1,k_1}(t)\}$.
\end{lem}

We set $c_1(t)=c_{1,1}(t)$. It follows from \eqref{LB1aa} and Lemma \ref{LB1a} that for any $t \ge t^{(1)}$,
\begin{align*}
|c_1(t)| \le (1+\theta_1)\sigma_1(t).
\end{align*}
In other words, $c_1(t)$ is the first member of a standard $\theta_1$-basis in $T_1(t)$.

In addition, it follows from Theorem \ref{T:esray} that
\begin{align*}
\|\mathbf r(c_1(t))\| \le CK_1(t/2)
\end{align*}
and there exists a constant $L_1>0$ such that
\begin{align}\label{E503}
|L(c_1(t))-L_1| \le C\int_t^{\infty}K_1(z/2)\,dz \le C\int_t^{\infty} \frac{K(z/2)}{z}\,dz \le CK_0(t/2).
\end{align}

For any $t \ge t^{(1)}$, after $c_1(t)$ is chosen, we can define the set $T_2(t)$ as in Section $4.1$. Moreover, we set $\sigma_2(t)$ to be the shortest length in $T_2(t)$.

Next, we prove
\begin{lem}\label{LB4b}
If $\sigma_2(t)\ne o(t)$ as $t \to \infty$, then there exists a constant $C>0$ such that for any $t \ge t^{(1)}$,
\begin{align*} 
\sigma_2(t) \ge C t.
\end{align*}
\end{lem}

\begin{proof}
Assuming the contrary, there exists a sequence $t_j \to \infty$ such that
\begin{align}\label{eq:B2a}
\lim_{j \to \infty }\frac{\sigma_2(t_j)}{t_j}=0.
\end{align}

At $t_j$, there exists a $c^j \in \tilde T_1(t_j)$ such that
\begin{align*}
|P_1(c^j)|=\sigma_2(t_j)
\end{align*}
and hence by the definition of $\tilde T_1(t_j)$,
\begin{align}\label{eq:B2c}
|c^j| \le C\sigma_2(t_j).
\end{align}

Now we define the sliding of $c^j$ along $\alpha(t)$ by $c^j(t)$, then it follows from \eqref{eq:B2c} and  Theorem \ref{T:esray} (ii) that for any $t \ge t_j$,
\begin{align*}
|c^j(t)| \le C|c^j| \le C \sigma_2(t_j).
\end{align*}

For any $t \ge t_j$, there exists an integer $k$ such that $c^k_1(t)*c^j(t) \in \tilde T_1(t)$ and from \eqref{eq:chok}
\begin{align*}
|k| \le \frac{1}{1-\theta}\frac{\la c_1(t),c^j(t)\ra}{|c_1(t)|^2}\le 2 \frac{|c^j(t)|}{|c_1(t)|}\le C|c^j(t)|
\end{align*}
since $|c_1(t)|$ converges to a nonzero constant by \eqref{E503}. Now it follows from \eqref{eq:tr1} and \cite[Proposition $4.2.3$(iv)]{BK81} that
\begin{align*}
|c^k_1(t)*c^j(t)| &\le |c^j(t)|+\lc 1+\frac{\theta}{\rho_1}|c^j(t)|\rc|c^k_1(t)| \\
& \le |c^j(t)|+ck|c_1(t)| \le C|c^j(t)|.
\end{align*}

By Proposition \ref{P:homo}, $c^j(t)$ is not generated by $c_1(t)$. Therefore,
\begin{align} \label{eq:B2f}
|\sigma_2(t)| \le |P_1(c^k_1(t)*c^j(t))| \le C|c^j(t)| \le C \sigma_2(t_j).
\end{align}

Then we conclude from \eqref{eq:B2a} and \eqref{eq:B2f} that
\begin{align*}
\lim_{t \to \infty} \frac{\sigma_2(t)}{t}=0
\end{align*}
and we obtain a contradiction.
\end{proof}

If $\sigma_2(t)\ne o(t)$, then it follows from Lemma \ref{LB4b} that there exists a constant $\kappa'>0$ such that for any $t \ge t^{(1)}$, $\Gamma(\alpha(t),\kappa't)$ is generated by $c_1(t)$ and the construction is complete. In this case, we choose $m=1$. Therefore, we only need to consider the case when $\sigma_2(t)= o(t)$ as $t \to \infty$.

For a $t^{(2)} \ge t^{(1)}$ to be determined later, we assume that $\{c_{2,1},c_{2,2},\cdots,c_{2,k_2}\} \subset \tilde T_1 \cup T_{1,0}$ are all elements such that $|P_1(c_{2,i})|=\sigma_2(t^{(2)})$ at $\alpha(t^{(2)})$. Here $T_{1,0} \coloneqq \{c_1\}^{\perp} \cap T_1$. Like before, we denote the sliding of $c_{2,i}$ by $c_{2,i}(t)$ along $\alpha(t)$. Moreover, we set $\tilde T_1(t)$, $T_{1,0}(t)$, $T_2(t)$, etc. to be the corresponding sets with respect to $c_1(t)$ at $\alpha(t)$. 

We define the subset $Z_2 \subset [t^{(2)},\infty)$ such that $t \in Z_2$ if and only if any $a \in \tilde T_1(t) \cup T_{1,0} (t)$ such that $|P_1(a)|=\sigma_2(t)$ is contained in $\{c_{2,1}(t),c_{2,2}(t),\cdots,c_{2,k_2}(t)\}$. Similar to Lemma \ref{LB1}, we prove

\begin{lem}\label{LB4}
The set $Z_2$ is nonempty and open in $ [t^{(2)},\infty)$.
\end{lem}

\begin{proof}
It is obvious that $t^{(2)} \in Z_2$ and $Z_2$ is nonempty. To prove that $Z_2$ is open, we only need to prove that $Z_2$ is open at $t^{(2)}$. Assuming the contrary, there exists a sequence $t_i \to (t^{(2)})^+$ such that at $\alpha(t_i)$ we can find a $b_i \in \tilde T_1(t_i) \cup T_{1,0} (t_i)$ which is not contained in $\{c_{2,1}(t_i),c_{2,2}(t_i),\cdots,c_{2,k_2}(t_i)\}$. If we denote the sliding of $b_i$ to $\alpha(t^{(2)})$ along $\alpha(t)$ by $e_i$, then by taking a subsequence if necessary, we assume that $e_i$ converges to a geodesic loop $e_{\infty}$. On the one hand, if $b_i \in T_{1,0} (t_i)$ for infinitely many $i$, then by taking the limit we have $e_{\infty} \in T_{1,0} (t^{(2)})$. On the other hand, if $b_i \in \tilde T_1 (t_i)$ for all large $i$, then by the definition of $\tilde T_1$ we have
\begin{align} \label{eq:B3}
\la c_1(t_i), b_i\ra >0\quad \text{and} \quad \la c_1(t_i), c_1^{-1}(t_i)b_i\ra \le 0.
\end{align}

By taking the limit of \eqref{eq:B3}, we conclude that
\begin{align*}
\la c_1(t^{(2)}), e_{\infty} \ra \ge 0\quad \text{and} \quad \la c_1(t^{(2)}), c_1^{-1}(t^{(2)}) e_{\infty}\ra \le 0.
\end{align*}
Therefore, $e_{\infty} \subset \tilde T_1(t^{(2)}) \cup T_{1,0} (t^{(2)})$. By continuity, $|P_1(e_{\infty})|=\sigma_2(t^{(2)})$ and hence $e_{\infty}$ is identical with $c_{2,u}$ for some $1\le u\le k_2$. It follows from Lemma \ref{L:esta} that
\begin{align} \label{eq:B5}
|e^{-1}_{\infty}*e_i-(e_i-e_{\infty})| \le \frac{\theta(1+\theta)}{\rho_1} |e_{\infty}|\,|e_i-e_{\infty}|.
\end{align}

If $i$ is sufficiently large, it is easy to see from \eqref{eq:B5} that the image $P_1(Q_1(e^{-1}_{\infty}*e_i))$ has the length smaller than $\sigma_2(t_1)$. Therefore, for sufficiently large $i$ we must have $e_i=e_{\infty}$ and $b_i=c_{2,u}(t_i)$, which is a contradiction.
\end{proof}

Now we assume $t^{(2)}$ is sufficiently large such that from Theorem \eqref{T:esray} (ii) and Proposition \ref{p:angle} that
\begin{align} \label{LB5}
 \lc 1-\frac{\theta_1}{10}\rc^{\frac{1}{C_0}} \le \frac{|P_1(c_{2,i}(t))|}{|P_1(c_{2,j}(t))|} \le  \lc 1+\frac{\theta_1}{10}\rc^{\frac{1}{C_0}}
\end{align}
for any $1 \le i,j \le k_{2}$ and $t \ge t^{(2)}$.

By the same argument as in Lemma \ref{LB1a}, we can redefine $t^{(2)}$ and $k_2$ so that the following lemma holds.

\begin{lem}\label{LB4a}
For any $t \ge t^{(2)}$, any $a \in \tilde T_1(t) \cup T_{1,0} (t)$ such that $|P_1(a)|=\sigma_2(t)$ is contained in $\{c_{2,1}(t),c_{2,2}(t),\cdots,c_{2,k_2}(t)\}$.
\end{lem}

Now we set $c_2(t)=c_{2,1}(t)$, then it follows from Theorem \ref{T:esray} that
\begin{align*}
\|\mathbf r(c_2(t))\| \le CK_1(t/2)
\end{align*}
and there exists a constant $L_2>0$ such that
\begin{align*}
|L(c_2(t))-L_2| \le CK_0(t/2).
\end{align*}

Moreover, it follows from Proposition \ref{p:angle} that
\begin{align*}
|\angle (c_1(t),c_2(t))-\Theta_{12} | \le CK_0(t/4)
\end{align*}
for some constant $\Theta_{12} \ge 0$.

It is clear from \eqref{LB5} and Lemma \ref{LB4a} that for any $t \ge t^{(2)}$,
\begin{align}\label{LB7}
|P_1(c_2(t))| \le (1+\theta_1)\sigma_2(t).
\end{align}

In addition, for any $t \ge t^{(2)}$,
\begin{align}\label{LB8}
c_1^{k}(t)*c_2(t) \in \tilde T_1(t).
\end{align}
for some $k \in \{-1,0,1\}$.

It follows from \eqref{eq:cho10d} that for any $t \ge t^{(2)}$,
\begin{align}\label{LB9}
-\theta \le \cos \angle(c_1(t),c_2(t)) &\le \lc \frac{1}{2}+9\theta \rc^{\frac{1}{2}}.
\end{align}

In particular, \eqref{LB9} implies that $\Theta_{12}$ must be positive. Therefore, it follows from \eqref{LB7}, \eqref{LB8} and \eqref{LB9} that $\{c_1(t),c_2(t)\}$ are the first two members of a generalized standard short $\theta_1$-basis.

By iteration, we can consecutively define $\{c_1(t),c_2(t),\cdots, c_m(t)\}$ and the sets $\bar T_1(t)=T_1(t) \supset \bar T_2(t) \supset \cdots \supset \bar T_m(t)$ along $\alpha(t)$ as before such that the following properties are satisfied.

\begin{enumerate}[label=(\roman*)]

\item For any $1 \le i \le m$,
\begin{align} \label{eq:B13}
|\mathbf r(c_i(t))-I| \le CK_1(t/2).
\end{align}

\item There exist some constants $L_i>0$ such that
\begin{align} \label{eq:B13a}
|L(c_i(t))-L_i| \le CK_0(t/2).
\end{align}

\item There exists some constants $\Theta_{ij}>0$ such that
\begin{align} \label{eq:B13b}
|\angle (c_i(t),c_j(t))-\Theta_{ij}| \le CK_0(t/4).
\end{align}

\item For any $1 \le i \le m$,
\begin{align*}
|c^{(i)}_i(t)| \le (1+\theta_1)\bar \sigma_i(t),
\end{align*}
where $c^{(i)}_i(t)$ is the projection of $c_i(t)$ to $\bar T_i(t)$ and $\bar \sigma_i(t)$ is the smallest length in $\bar T_i(t)$.

\item For any $1 \le i \le m-1$,
\begin{align*}
(c^{(i)}_i(t))^k*c_{i+1}^{(i)}(t) \in \widetilde{\bar T_i}(t)
\end{align*}
for some $k \in \{-1,0,1\}$.

\item For any $1 \le i <j\le m$,
\begin{align} \label{eq:B13e}
-\theta \le \cos \angle(c^{(i)}_i(t),c^{(i)}_{j}(t)) &\le \lc \frac{1}{2}+9\theta \rc^{\frac{1}{2}}.
\end{align}
\end{enumerate}

To summarize, for sufficiently large $t$, we have constructed a generalized standard short $\theta_1$-basis at $\alpha(t)$. In particular, Theorem \ref{T:basis2} is proved along the geodesic ray $\alpha(t)$.

\textbf{Step 2: Extension of the short bases on the end}

To construct a short basis on the entire end of $M^n$, we consider $q \in \partial B(p,t)$ for $t \ge A$, where $A$ is a large constant. From \eqref{eq:diam}, we can choose an arc-length parametrized curve $\{\beta_q(s): s\in [0,a]\}$ on $B(p,11t/10)\backslash \bar B(p,9t/10)$ such that $\beta_q(0)=\alpha(t)$ and $\beta_q(a)=q$. Moreover, there exists a constant $C >0$ independent of $t$ and $q$ such that
\begin{align} \label{eq:B15}
L(\beta_q)=a \le Ct.
\end{align}

If we assume that $\{c^q_1(s),c^q_2(s),\cdots, c^q_m(s)\}$ are the slidings of $\{c_1(t),c_2(t),\cdots, c_m(t)\}$ along $\beta_q(s)$, then we have

\begin{prop} \label{LBx8}
 There exist constants $A_{2,1}>0$ and $\kappa_1>0$ such that if $t \ge A_{2,1}$, $$\{c^q_1(a),c^q_2(a),\cdots, c^q_m(a)\}$$ is a short basis with radius $\kappa_1t$ satisfying \eqref{eq:basis1} and \eqref{eq:basis2}.
\end{prop}

\begin{proof}
For any $s \le a$, we choose a unit parallel vector field $X_i$ along $\beta_q$ such that
\begin{align*} 
|\mathbf r(c_i^q(s)) X_i(s)-X_i(s)|=|\mathbf r(c_i^q(s))-I|.
\end{align*}

Then it follows from Theorem \ref{T:hol} and \eqref{eq:B13} that
\begin{align} 
&\frac{d}{ds}L(c^q_i(s)) \le |\mathbf r(c_i^q(s))-I|  \notag \\
\le& |\mathbf r(c_i^q(0))-I|+Ct^{-2}K(t/2)\int_0^s  L(c^q_i(s)) \,ds \notag \\
\le & CK_1(t/2)+Ct^{-2}K(t/2) \int_0^s L(c^q_i(s)) \,ds  \label{eq:B19a}.
\end{align}

Now we claim that $L(c^q_i(s))$ must be uniformly bounded. Indeed, it follows from \eqref{eq:B19a} that
\begin{align*} 
L(c^q_i(s)) \le& L(c^q_i(0))+CsK_1(t/2)+Cst^{-2}K(t/2) \int_0^s L(c^q_i(s)) \,ds \\
\le& C+CtK_1(t/2)+Ct^{-1}K(t/2)\int_0^s L(c^q_i(s)) \,ds
\end{align*}
where we have used \eqref{eq:B13a} and \eqref{eq:B15}. Then it follows from the Gronwall's inequality that
\begin{align} \label{eq:B19bb}
L(c^q_i(s)) \le C(1+tK_1(t/2))\exp(CK(t/2)) \le C.
\end{align}

From \eqref{eq:B19a} we conclude that
\begin{align} \label{eq:B19b}
|\mathbf r(c_i^q(s))-I| \le C K_1(t/2)+Ct^{-1}K(t/2) \le Ct^{-1}K(t/2).
\end{align}

From \eqref{eq:B19a} and \eqref{eq:B19bb} that for any $s \in [0,a]$,
\begin{align*}
|L(c^q_i(s))-L(c^q_i(0))| \le C sK_1(t/2)+Cs^2t^{-2}K(t/2) \le CK(t/2).
\end{align*}

Now it follows from \eqref{eq:B13a} that
\begin{align} \label{eq:B19d}
|L(c^q_i(s))-L_i| \le C(K(t/2)+K_0(t/2)) \le C\tilde K(t/2).
\end{align}

By \eqref{eq:B19b} and \eqref{eq:B19d}, it follows from Proposition \ref{P:es4} and \eqref{eq:B13b} that for any $1 \le i,j \le m$ and $s \in [0,a]$,
\begin{align*}
|\angle(c^q_i(s),c^q_j(s))-\Theta_{ij}| \le C(K(t/4)+K_0(t/4))\le C\tilde K(t/4).
\end{align*}

Therefore, we have proved that for any $1 \le i<j \le m$,
\begin{align} \label{eq:B20c}
\begin{cases}
&|\mathbf r(c_i^q(a))-I| =O(t^{-1}K(t/2)), \\
&|L(c^q_i(a))-L_i|=O(\tilde K(t/2)),\\
&|\angle(c^q_i(a),c^q_j(a))-\Theta_{ij}|  =O(\tilde K(t/4)).
\end{cases}
\end{align}

Now we claim that $\{c^q_1(a),c^q_2(a),\cdots, c^q_m(a)\}$ is a $\lambda$-normal basis with $\lambda \le 2$. For any two vectors $  u,  v \in \R^m$ and there their projections $  u',  v'$ onto some hyperplane $H$ with normal vector $  \vec n$, the angle $\angle (  u',  v')$ is completely determined by $\angle (u, v),\angle (  u,  \vec n)$ and $\angle (  v,  \vec n)$. Indeed, it is easy to show
\begin{align}\label{eq:B20d}
\cos \angle (u',v')=\frac{\cos \angle (u,v)-\cos \angle (u,\vec n) \cos \angle (v,\vec n)}{\sin \angle (u,\vec n) \sin \angle (v,\vec n)}.
\end{align}

Therefore, the fact that $\{c^q_1(a),c^q_2(a),\cdots, c^q_m(a)\}$ is a $\lambda$-normal basis with $\lambda \le 2$ follows from \eqref{eq:B13e}, \eqref{eq:B20c} and the formula \eqref{eq:B20d}. Hence, condition (i) in Definition \ref{def:basis} is satisfied.

Since $\{c_1(t),c_2(t),\cdots, c_m(t)\}$ is a generalized standard short $\theta_1$-basis, it follows Proposition \ref{P:basis} that there exist structure constants $k^{ij}_u \in \mathbb Z$ for $1\le i<j\le m$ and $1 \le u \le i-1$ such that
\begin{align*}
[c_i(t),c_j(t)]=c_1^{k^{ij}_1}(t)*c_2^{k^{ij}_2}(t)*\cdots * c_{i-1}^{k^{ij}_{i-1}}(t).
\end{align*}
Notice that all those structure constants $k_u^{ij}$ are independent of $t$ since by Proposition \ref{P:group} the sliding preserves the group structure. As $c^q_i(a)$ is the sliding of $c_i(t)$ along $\beta_q$, by the same reason we have
\begin{align*}
[c^q_i(a),c^q_j(a)]=(c^q_1(a))^{k^{ij}_1}*(c^q_2(a))^{k^{ij}_2}*\cdots * (c^q_{i-1}(a))^{k^{ij}_{i-1}}.
\end{align*}

Therefore, condition (iii) in Definition \ref{def:basis} is satisfied.

Recall the definitions \eqref{eq:k0} and \eqref{eq:chorhoab}, if we set $\bar \rho=\bar \rho(t)=\kappa_0 t$ for some constant $\kappa_0>0$, then $\{c_1(t),c_2(t),\cdots, c_m(t)\}$ is a short basis with radius $\bar \rho$ by Proposition \ref{P:basis}. For any $\gamma \in \Gamma(q,\rho_0)$ with $|\gamma| \le a_0 \bar \rho$, where $a_0$ is a constant to be determined later, we denote the sliding of $\gamma$ along $\beta_q(a-s)$ by $\gamma(s)$. 

It follows from Theorem \ref{T:hol}(i) and \eqref{cond:SHC} that
\begin{align*}
\frac{d}{ds}L(\gamma(s)) \le |\mathbf r(\gamma(s))-I| \le \frac{C}{t} L(\gamma(s)).
\end{align*}

Therefore, we conclude that for some constant $C>1$,
\begin{align*} 
L(\gamma(a)) \le C L(\gamma) \le Ca_0\bar \rho.
\end{align*}

If we choose $a_0 <C^{-1}$, then at $\alpha(t)$, we have
\begin{align} \label{eq:B24}
\gamma(a)=c^{l_1}_1(t)*c^{l_2}_2(t) *\cdots * c^{l_m}_m(t).
\end{align}

Moreover, it follows from Lemma \ref{L:esta1}
\begin{align*}
\sum_{i=1}^m |l_ic_i(t)| \le 2^{\frac{1}{2}m^2}\left|\sum_{i=1}^m l_ic_i(t)\right |\le \frac{2^{\frac{1}{2}m^2}}{1-\theta} |\gamma(a)| \le \frac{2^{\frac{1}{2}m^2}}{1-\theta} \bar \rho.
\end{align*}
From \eqref{eq:B19d}, we easily conclude that for sufficiently large $t$,
\begin{align*}
\sum_{i=1}^m |l_ic_i^q(a)| \le 2\sum_{i=1}^m |l_ic_i(t)| \le \frac{2^{\frac{1}{2}m^2+1}}{1-\theta} \bar \rho.
\end{align*}
Therefore, $(c^q_1(a))^{l_1}*(c^q_2(a))^{l_2}*\cdots * (c^q_m(a))^{l_m}$ is well defined in $\Gamma(q,\rho_0)$ and by Proposition \ref{P:group},
\begin{align*} 
\gamma=(c^q_1(a))^{l_1}*(c^q_2(a))^{l_2}*\cdots * (c^q_m(a))^{l_m}.
\end{align*}

Assume that there exists another representation
\begin{align*}
\gamma=(c^q_1(a))^{l'_1}*(c^q_2(a))^{l'_2}*\cdots * (c^q_m(a))^{l'_m},
\end{align*}
then by the same reason $c^{l'_1}_1(t)*c^{l'_2}_2(t) *\cdots * c^{l'_m}_m(t)$ is well defined in $\Gamma (\alpha(t),\rho_0)$ and
\begin{align} \label{eq:B31}
\gamma(a)=c^{l'_1}_1(t)*c^{l'_2}_2(t) *\cdots * c^{l'_m}_m(t).
\end{align}

By Proposition \ref{P:basis}, two representations \eqref{eq:B24} and \eqref{eq:B31} must be identical. Therefore, we conclude that $l_i=l'_i$ and the representation of $\gamma$ at $q$ is unique. In other words, condition (ii) in Definition \ref{def:basis} is satisfied for the radius $r_0=a_0 \bar \rho$.
\end{proof}

Finally we show that $\Gamma(q,r_0)$ is abelian if $r=r(q)$ is sufficiently large. To prove this, we only need to prove $[c^q_i,c^q_j]=0$ for any $1\le i<j \le m$.

\begin{prop} \label{P:abelian}
There exists a constant $A_{2,2}>0$ such that if $r=r(q)>A_{2,2}$, then the short basis $\{c_1^q,c_1^q,\cdots,c_m^q\}$ is abelian.
\end{prop}

\begin{proof}
We omit the superscript $q$ and assume for $i<j$,
\begin{align*} 
[c_i,c_j]=c_1^{k_1}*c_2^{k_2}*\cdots*c_{i-1}^{k_{i-1}}.
\end{align*}
Then it follows from Lemma \ref{L:esta1} we have
\begin{align} \label{eq:B33}
\sum_{u=1}^{i-1} |k_uc_u| \le 2^{\frac{1}{2}m^2}\left|\sum_{u=1}^{i-1} k_uc_u\right |\le \frac{2^{\frac{1}{2}m^2}}{1-\theta} |[c_i,c_j]|.
\end{align}
In addition, it follows from Proposition \ref{P:norm}(iii) that
\begin{align} \label{eq:B34}
|[c_i,c_j]| \le 3\frac{\theta}{\rho_0}|c_i||c_j|.
\end{align}
Combining \eqref{eq:B33}, \eqref{eq:B34} and \eqref{eq:basis1}, we conclude that for large $r$
\begin{align*}
\sum_{u=1}^{i-1} |k_u|L_u \le  \frac{C}{r}L_iL_j.
\end{align*}
Since all $k_u$ are integers, we must have $k_u=0$, provided that $r$ is sufficiently large.
\end{proof}

Combining \textbf{Step 1} and \textbf{Step 2}, the proof of Theorem \ref{T:basis2} is complete. 

\begin{rem}\label{rem:weak}
In the proof of Theorem \ref{T:basis2}, we only need \eqref{cond:SHC} on a geodesic ray and \eqref{cond:HC'}.
\end{rem}

From Proposition \ref{P:abelian}, the sliding of short geodesic loops at $x$ to nearby points is independent of the path.

\begin{cor}\label{c:slideind}
Let $x,y$ be points with $r=r(x),r(y) \ge A_{2,2}$ such that there are two curves $\beta_i,i=1,2$ with $L(\beta_i) \le \frac{\kappa_1r}{4}$ starting from $x$ to $y$.
Then for any $c \in \Gamma(x,\frac{\kappa_1r}{2})$, the sliding of $c$ to $y$ along $\beta_1$ is identical with that along $\beta_2$.
\end{cor}

\begin{proof}
We only need to prove that for any two lifts $u,v \in \hat B(0,\frac{\kappa_1r}{2}) \subset T_xM$ of $y$, the geodesic $\gamma_1$ connecting $u$ and $\tau_c(u)$ and the geodesic $\gamma_2$ connecting $v$ and $\tau_c(v)$ have the same image under $\ex_x$. Indeed, since $u$ and $v$ are the lifts of $y$, there exists an $a \in \Gamma(x,\frac{\kappa_1r}{2})$ such that $\tau_a(u)=v$. Then $\gamma_2$ is the image of $\gamma_1$ under $\tau_a$, since $\tau_a \circ \tau_c(u)=\tau_c \circ \tau_a(u)$. From this, the proof is complete.
\end{proof}

\subsection{Flat torus at infinity}

Now we have the following definition.

\begin{defn}\label{def:toin}
Given a complete Riemannian manifold $(M^n,g)$ with \eqref{cond:AF} and \eqref{cond:SHC}, the flat torus at infinity is defined by 
\begin{align*}
\T^m_{\infty} \coloneqq \R^m/ \la c^{\infty}_1,c^{\infty}_2,\cdots,c^{\infty}_m \ra,
\end{align*}
where $\{c^{\infty}_1,c^{\infty}_2,\cdots,c^{\infty}_m \}$ is a normal basis of $\R^m$ such that
\begin{align*}
|c^{\infty}_i|=L_i \quad \text{and} \quad \angle (c^{\infty}_i,c^{\infty}_j)=\Theta_{ij}.
\end{align*}
Here $L_i$ and $\Theta_{ij}$ are constants obtained in Theorem \ref{T:basis2}. 
\end{defn}

If we denote the standard basis of $\R^m$ by $(e_1,e_2,\cdots,e_m)$, then there exists a matrix $A^{\infty}=(a^{\infty}_{ij})_{1 \le i,j \le m}$ such that
\begin{align} \label{torus0a}
c_i^{\infty}=\sum a^{\infty}_{ji} e_j.
\end{align}

Notice that the matrix $A^{\infty}$ is well defined up to a left multiplication by an orthogonal matrix and determines the isometry class of $\T_{\infty}^m$ 

From \eqref{eq:basis1} and \eqref{eq:basis2}, for any $x$ with $r=r(x) \ge A_2$, there exists an orthonormal basis $(e_1^x,e_2^x,\cdots, e_m^x)$ of $\la c^x_1,c^x_2,\cdots,c^x_m \ra$ with respect to $g_x$ and a matrix $A^x=(a^x_{ij})_{1 \le i,j \le m}$ such that
\begin{align}  \label{torus2}
c_i^x=\sum a^x_{ji}e^x_j \quad \text{and} \quad|A^x-A^{\infty}| =O(\tilde K(r/4)).
\end{align}

In addition, we define the flat torus
\begin{align*}
\T^{m}_{x} \coloneqq \R^m/ \la c^x_1,c^x_2,\cdots,c^x_m \ra
\end{align*}
and a natural map $i_x$ from $\T_x^m$ to $\T_{\infty}^m$ defined by 
\begin{align*}
i_x(\overline{t_1c_1^x+t_2c_2^x+\cdots+t_mc_m^x}) \coloneqq \overline{t_1c_1^{\infty}+t_2c_2^{\infty}+\cdots+t_mc_m^{\infty}}
\end{align*}
for any $t_i \in \R$, where $\overline{\,\cdot\,}$ denotes the quotient map. It is easy to see by using \eqref{eq:basis1} and \eqref{eq:basis2} that the map $i_x$ is a $C \tilde K(r/4)$-almost isometry. In particular, we conclude that
\begin{align} \label{torus5x}
d_{GH}(\T_{\infty}^m,\T_x^m) =O( \tilde K(r/4)).
\end{align}

Let $x,y$ be points such that $r=r(x)\ge A_{2}$ with $d(x,y) \le \frac{\kappa_1r}{4}$. For any $1 \le i\le m$, we denote the sliding of $c_i^x$ to $y$ along the minimizing geodesic by $c_i^{x'}$, then by the same proof of Proposition \ref{LBx8} we have
\begin{align*} 
\begin{cases}
&|\mathbf r(c_i^{x'})-I| = O(r^{-1}K(r/2)), \\
&|L(c_i^{x'})-L_i| =O(\tilde K(r/2)),\\
&|\angle(c_i^{x'},c_i^{x'})-\Theta_{ij}| =O(\tilde K(r/4)).
\end{cases}
\end{align*}

In addition, $\la c^{x'}_1,c^{x'}_2,\cdots,c^{x'}_m \ra$ is a short basis of $\Gamma(y,\rho_0)$ with radius $\kappa_1'r(y)$ for some constant $\kappa_1'>0$ independent of $x$ and $y$. Since $\la c^{y}_1,c^{y}_2,\cdots,c^{y}_m \ra$ is a short basis, then for any $1 \le i \le m$,
\begin{align}
c_i^{y}=\sum k_{ji}c_j^{x'}, \label{eq:B36}
\end{align}
for some $K_{x,y}=(k_{ij}) \in \text{GL}(m,\Z)$. In addition, the norm of $K_{x,y}$ is uniformly bounded independent of the choice of $x$ and $y$. 

\begin{prop}\label{P:finite}
There exists a constant $A_{2,3}>0$ such that for any $x,y$ with $r=r(x) \ge A_{2,3}$ and $d(x,y) \le \kappa_1r/4$, 
\begin{align*} 
A^{\infty} K_{x,y} (A^{\infty})^{-1} \in O(m).
\end{align*}
\end{prop}

\begin{proof}
We denote the parallel transport from $x$ to $y$ along the minimizing geodesic by $P$ and set $e_i^{x'}=P(e_i^x)$. It follows from \cite[Proposition $6.6.2$]{BK81} and Lemma \ref{L:es1} that
\begin{align*} 
|c_i^{x'}-P(c_i^x)|=O(\tilde K(r/4)).
\end{align*}

Therefore, by \eqref{torus2} we have
\begin{align}  \label{torus5}
c_i^{x'}=\sum a^x_{ji}e^{x'}_j+O(\tilde K(r/4))=\sum a^{\infty}_{ji} e^{x'}_j+O(\tilde K(r/4)).
\end{align}
Moreover, by Definition \eqref{eq:B36} we have
\begin{align} \label{torus6}
c_i^{x'}=\sum k^{ji}c_j^{y},
\end{align}
where $K_{x,y}^{-1}=(k^{ij})$ is the inverse matrix of $K_{x,y}$. From \eqref{torus5} and \eqref{torus6} we have
\begin{align*}
(c_1^y,c_2^y,\cdots,c_m^y)=(e_1^{x'},e_2^{x'},\cdots,e_m^{x'}) A^{\infty}K_{x,y}+O(\tilde K(r/4)).
\end{align*}

From \eqref{eq:basis1} and \eqref{eq:basis2}, there exists a matrix $O \in O(m)$ such that
\begin{align*} 
OA^{\infty}K_{x,y}=A^{\infty}+O(\tilde K(r/4)).
\end{align*}

Since $K_{x,y}=(k_{ij}) \in \text{GL}(m,\Z)$ and the norm is uniformly bounded, it is easy to see if $r$ is sufficiently large, $A^{\infty} K_{x,y} (A^{\infty})^{-1} \in O(m)$.
\end{proof}

We define the following subgroup of $\text{GL}(m,\Z)$ for any $A \in \text{GL}(m,\R)$,
\begin{align} \label{def:group}
G(A) \coloneqq \{W\in\text{GL}(m,\Z) \,\mid \,  A W A^{-1} \in O(m)\}.
\end{align}

\begin{lem}\label{L:finiteorder}
For any $A \in \text{GL}(m,\R)$, $G(A)$ is a finite group. Conversely, if $G$ is a finite subgroup of $\text{GL}(m,\Z)$, then $G\subset G(A)$ for some $A \in \text{GL}(m,\R)$.
\end{lem}

\begin{proof}
It is clear that the map $\phi: G(A) \to O(m)$ by $\phi(W)=A W A^{-1}$ is a group monomorphism. Since $G(A)$ is discrete and $O(m)$ is compact, the image of $\phi$ and hence $G(A)$ are finite. Conversely, if $G$ is a finite subgroup of $\text{GL}(m,\Z)$, there exists an inner product which is invariant under the action of $G$. Then it is easy to show for some $A \in \text{GL}(m,\R)$, $AGA^{-1} \subset O(m)$.
\end{proof}

Now we have the following definition.

\begin{defn}\label{def:holin}
The finite group $G_{\infty} \coloneqq G(A^{\infty})$ is called the holonomy at infinity.
\end{defn}

Given $x$ and $y$, we can naturally define an automorphism of $\T_{\infty}^m$ such that for any $1 \le i \le m$,
\begin{align}\label{eq:auto}
L_{x,y}(c_i^{\infty})\coloneqq\sum_{j=1}^m k_{ji}c_j^{\infty}.
\end{align}

Next, we prove that all torus automorphisms are compatible.

\begin{lem}\label{L:trans5}
Given three points $x,y,z$ such that $r=r(x) \ge A_2$ and their respective distances are bounded by $\kappa_1r/8$, then we have
\begin{align*}
L_{x,z}=L_{x,y} \circ L_{y,z}.
\end{align*}
\end{lem}

\begin{proof}
From the definition, all constants $k_{ij}$ are determined by
\begin{align*}
c_i^y = \sum_{j=1}^m k_{ji}c_j^{x'},
\end{align*}
where $c_j^{x'}$ is the sliding of $c_j^x$ to $y$. Since by Corollary \ref{c:slideind}, the sliding of $c_i^x$ to $z$ along the geodesic from $x$ to $z$ is the same as that along the geodesic from $x$ to $y$ followed by the geodesic from $y$ to $z$, we easily conclude that
\begin{align*}
L_{x,z}=L_{x,y} \circ L_{y,z}.
\end{align*}
\end{proof}

\section{Proof of the main theorem}

In this section, we construct a torus fibration on the end of $(M^n,g)$. With the help of Theorem \ref{T:basis2}, we first construct a local fibration around each point $q$ if $r=r(q)$ is sufficiently large. Then we glue all local fibrations into a global fibration. Notice that this is a standard strategy from \cite{CFG92}. Here, we follow the argument of Minerbe~\cite[Section $3.4$-$3.6$]{Mi10} (where $K(r)=r^{-1}$, $n=4$ and $m=1$ are assumed).

\subsection{Construction of local fibrations}

For any $q$ with $r=r(q) \ge A_2$, it follows from Theorem \ref{T:basis2} that there exists an abelian short basis $\{c_1,c_2,\cdots,c_m\}$ of radius $\kappa_1 r$ such that \eqref{eq:basis1} and \eqref{eq:basis2} are satisfied. Here we omit the superscript $q$ of $c_i^q$ for simplicity. We also set for $i\ge 2$,
\begin{align*}
\kappa_i=\frac{\kappa_1}{10^{i-1}}.
\end{align*}

For any geodesic loop $a \in \Gamma(q,\kappa_1 r)$, we set $a=\mathbf t(a)$ and denote the corresponding local isometry by $\tau_a$ and the rotational map by $\mathbf r_a$. Now we show the local isometry $\tau_a$ is almost translational (this is similar to \cite[Lemma $3.9$]{Mi10}).

\begin{lem} \label{L:tr}
There exists a constant $C_{2,1}>0$ such that for any $q$ with $r \ge A_2$ and $a \in \Gamma(q,\kappa_1 r)$, we have
\begin{enumerate}[label=(\roman*)]
\item $|\mathbf r_a-I| \le C_{2,1}  C|a|r^{-1}K(r/2).$

\item $|\tau_a(w)-(w+a)| \le C_{2,1}  |a| K(r/2)$ for any $w \in \hat B(0,\kappa_1 r)$.
\end{enumerate}
\end{lem}

\begin{proof}
From Theorem \ref{T:basis2}, there exists a representation 
\begin{align*}
a=k_1c_1+c_2 k_2+\cdots+k_m c_m
\end{align*}
for $k_i \in \mathbb Z$. It follows from the definition of short basis and Lemma \ref{L:esta1} that 
\begin{align}
\sum_{i=1}^m |k_i| \le C|a|. \label{eq:501}
\end{align}

Now we claim that for any $1 \le i \le m$,
\begin{align*}
|\mathbf r_{k_ic_i}-I| \le C|a|r^{-1}K(r/2).
\end{align*}
We only need to prove the claim for $i=1$.  For any $1 \le j\le k_1$, we set 
\begin{align*}
e_j=\mathbf r_{jc_1}-I.
\end{align*}
Then we have
\begin{align}\label{eq:503}
|e_j-e_{j-1}|=|\mathbf r_{jc_1}-\mathbf r_{c_1}\circ \mathbf r_{(j-1)c_1}|+|\mathbf r_{c_1}\circ \mathbf r_{(j-1)c_1}-\mathbf r_{(j-1)c_1}|.
\end{align}

It follows from \cite[Proposition $2.3.1$(i)]{BK81} and \eqref{eq:501} that
\begin{align}\label{eq:504}
|\mathbf r_{jc_1}-\mathbf r_{c_1}\circ \mathbf r_{(j-1)c_1}| \le Cr^{-2}K(r/2)|jc_1|\,|c_1| \le Cr^{-2}K(r/2) \le Cr^{-1}K(r/2).
\end{align}

In addition, it is clear from \eqref{eq:basis1} that
\begin{align}\label{eq:505}
|\mathbf r_{c_1}\circ \mathbf r_{(j-1)c_1}-\mathbf r_{(j-1)c_1}| \le |\mathbf r_{c_1}-I| \le Cr^{-1}K(r/2).
\end{align}

Combining \eqref{eq:504} and \eqref{eq:505}, we conclude from \eqref{eq:503} that
\begin{align*}
|e_j-e_{j-1}| \le Cr^{-1}K(r/2).
\end{align*}

Therefore, the claim holds since
\begin{align*}
|\mathbf r_{k_1c_1}-I|=|e_{k_1}|=\sum_{j=1}^{k_1}|e_j-e_{j-1}| \le Ck_1r^{-1} K(r/2) \le C|a|r^{-1}K(r/2).
\end{align*}

If we set $b_i=\mathbf r_{k_1c_1+\cdots+k_ic_i}-I$, then it follows from the claim and \cite[Proposition $2.3.1$(i)]{BK81} again that
\begin{align*}
|b_i-b_{i-1}|=&|\mathbf r_{k_1c_1+\cdots+k_i c_i}-\mathbf r_{k_ic_i}\circ \mathbf r_{k_1c_1+\cdots+k_{i-1}c_{i-1}}|+|\mathbf r_{k_ic_i}\circ \mathbf r_{k_1c_1+\cdots+k_{i-1}c_{i-1}}-\mathbf r_{k_1c_1+\cdots+k_{i-1}c_{i-1}}| \\
\le& C|a|r^{-1}K(r/2).
\end{align*}

Therefore, (i) is proved since
\begin{align*}
|\mathbf r_a-I|=\sum_{i=1}^m |b_i-b_{i-1}| \le C|a|r^{-1}K(r/2).
\end{align*}

To prove (ii), it follows from Lemma \ref{L:es1} that
\begin{align*}
&|\tau_a(w)-(w+a)| \\
\le& |\tau_a(w)-\mathbf r_a^{-1}(w+a)|+|\mathbf r_a-I|(|w|+|a|) \\
\le & Cr^{-2}K(r/2)|a|\,|w|(|a|+|w|)+C|a|K(r/2) \le C|a|K(r/2).
\end{align*}

Therefore, the proof of (ii) is complete.
\end{proof}

Now we define an open set $\hat U_q \subset \hat B(0,\kappa_1 r) $ such that $x \in \hat U_q$ if and only if the projection of $x$, denoted by $x_c$, onto $\la c_1,c_2,\cdots,c_m \ra$ has the representation
\begin{align*}
x_c=s_1c_1+s_2c_2+\cdots+s_mc_m
\end{align*}
such that $\max_i|s_i|<2$. Moreover, we define the norm
\begin{align*}
\|x\|_c=\sum_{i=1}^m |s_i|.
\end{align*}

Notice that $\|x\|_c$ is uniformly comparable to the usual norm $|x|$. Next, we prove

\begin{lem}\label{L:neigh}
There exists a constant $A_{2,1}>0$ such that for any $x$ with $r=r(x) \ge A_{2,1}$ and $y \in B(x,\frac{\kappa_1}{2}r)$, there exists $\hat y \in \hat U_x$ such that $\emph{\ex}_x (\hat y)=y$.
\end{lem}

\begin{proof}
We assume that $\hat y$ is the lift of $y$ in $\hat B(x,\frac{\kappa_1}{2}r)$ such that $\|y\|_c$ is minimal. We claim that $\hat y \in \hat U_x$. Otherwise, we can assume without loss of generality that $y_c=s_1c_1+s_2c_2+\cdots+s_mc_m$ and $s_1 \ge 2$. Then it follows from Lemma \ref{L:tr} (ii) that
\begin{align*}
|\tau_{-c_1}(\hat y)-(\hat y-c_1)| \le C  K(r/2).
\end{align*}
Then it is clear that 
\begin{align*}
\|\tau_{-c_1}(\hat y)\|_c \le \|\hat y\|_c-1+C  K(r/2)
\end{align*}
and hence
\begin{align*}
\|\tau_{-c_1}(\hat y)\|_c <\|\hat y\|_c
\end{align*}
if $r$ is sufficiently large. This gives us a contradiction.
\end{proof}

\begin{lem}\label{L:neigh2}
There exists constants $C_{2,2}>0$ and $A_{2,2}>0$ such that for $r \ge A_{2,2}$, if $w \in \hat U_q$ and $\tau_a(w) \in \hat U_q$ for $a \in \Gamma(q,\kappa_1r)$, then $|a| \le C_{2,2}$.
\end{lem}

\begin{proof}
It follows from Lemma \ref{L:tr} (ii) that 
\begin{align*}
|\tau_a(w)-(w+a)| \le C_{2,1}|a|K(r/2).
\end{align*}

If $r$ is sufficiently large such that $C_{2,1}K(r/2)<\frac{1}{2}$, then
\begin{align*}
|a| \le 2(|w|+|\tau_a(w)| \le C
\end{align*}
since $w,\tau_a(w) \in \hat U_q$.
\end{proof}

Now we recall the following definition.

\begin{defn}\label{def:GH}
A map $h: X \to Y$ of metric spaces is called an $\delta$-Gromov-Hausdorff approximation if the following two conditions are satisfied.
\begin{enumerate}[label=(\roman*)]
\item $|d_X(x_1,x_2)-d_Y(h(x_1),h(x_2))| \le \delta$ for all $x_1,x_2 \in X$.

\item For any $y \in X$, there exists $x \in X$ such that $|h(x)-y| \le \delta$.
\end{enumerate}
\end{defn}

In particular, the existence of an $\delta$-Gromov-Hausdorff map implies that $d_{GH}(X,Y) \le 2\delta$, see \cite[Corollary $7.3.28$]{BBI01}.

Next, we prove that for any $q$ with $r(q)$ sufficiently large, there exists a neighborhood of $q$ which is close to a ball in $\R^{n-m}$, in the Gromov-Hausdorff sense. The following proposition corresponds to \cite[Proposition $3.10$]{Mi10}.

\begin{prop}\label{P:GH}
There exists a constant $C>0$ such that for any $q$ with $r=r(q) \ge A_2$, there exists an open neighborhood $U_q$ of $q$ and the Gromov-Hausdorff distance between $U_q$ and $B(0,\kappa_2 r) \subset \R^{n-m}$ is bounded by $C\max\{1,rK(r/2)\}$.
\end{prop}

\begin{proof}
Given $q$ with $r\ge A_2$, we define $H=\la c^q_1,c^q_2,\cdots c^q_m\ra^{\perp}$ and use the subscript $v_H$ to denote the projection of $v \in \hat B(0,\kappa_1 r)$ onto $H$. Here, the orthogonal complement is defined by using the inner product $\hat g=\ex_q^*g$ on $T_q M$.

For any $x \in B(q,\frac{\kappa_1}{2}r)$, we define the center of mass
\begin{align*}
h(x) \coloneqq \frac{1}{N_x} \sum_{v\in A_x} v_H,
\end{align*}
where $A_x=\{v \in \hat U_q\,\mid \ex_q ( v)=x\}$ and $N_x=|A_x|$.

Now we set $B=B(0,\kappa_2 r) \subset H$ and define $U_q= h^{-1}(B)$. We claim that $h:U_q \to B $ is a $C\max\{1,rK(r/2)\}$-Gromov-Hausdorff approximation. 

For any $v \in B$ and $a \in \Gamma(q,\kappa_1 r)$ such that $\tau_a(v)\in \hat U_q$, it follows from Lemma \ref{L:tr}(ii) and Lemma \ref{L:neigh2} that 
\begin{align*}
|\tau_a(v)-(v+a)| \le C|a|K(r/2) \le CK(r/2).
\end{align*}
Since $a=k_1c_1+k_2c_2+\cdots+k_mc_m$, we have
\begin{align*}
|(\tau_a(v))_H-v| \le CK(r/2).
\end{align*}

Passing to the center of mass, we conclude that
\begin{align*}
|h(\ex_q(v))-v|  \le CK(r/2).
\end{align*}

In particular, if $d(v,H\backslash B) \ge CK(r/2)$, then $h(\ex_q(v)) \in B$. Therefore, it is easy to show that the range of $h$ is a $CK(r/2)$-net and condition (ii) in Definition \ref{def:GH} is satisfied.

For any $x,y \in U_q$, by Lemma \ref{L:neigh} we can find their lifts $v,w \in \hat U_q$. By the same arguments as before, we have
\begin{align}\label{eq:516}
|h(x)-v_H|+|h(y)-w_H| \le CK(r/2).
\end{align}

There exists a $w' \in \hat B(0,\frac{\kappa_1}{2}r)$ such that $d_{\hat g}(v,w')=d(x,y)$. Since $\ex_q(w)=\ex_q(w')$, we set $\tau_a(w)=w'$ for $a \in \Gamma(q,\frac{\kappa_1}{2}r)$. It follows from Lemma \ref{L:tr} (ii) that
\begin{align*}
|w'-(w+a)|=|\tau_a(w)-(w+a)| \le C|a|K(r/2)\le CrK(r/2).
\end{align*}

On the one hand, we have
\begin{align}
&|h(x)-h(y)| \notag \\
\le& |v_H-w_H|+CK(r/2) \le |v_H-(w+a)|+CK(r/2) \notag\\
\le& |v_H-w'|+CrK(r/2) \le |v-v_H|+|v-w'|+CrK(r/2) \notag\\
\le& (1+CK(r/2)) d_{\hat g}(v,w')+C(1+rK(r/2))= d(x,y)+C(1+rK(r/2)).\label{eq:518}
\end{align}

On the other hand,
\begin{align}
& d(x,y) \le d_{\hat g}(v,w) \notag \\
\le& |v-w|+CrK(r/2)  \notag\\
\le& |v_H-w_H|+C(1+rK(r/2)) \notag\\
\le& |h(x)-h(y)|+C(1+rK(r/2)). \label{eq:519}
\end{align}

Combining \eqref{eq:518} and \eqref{eq:519}, condition (i) in Definition \ref{def:GH} is satisfied and the proof is complete.
\end{proof}

Notice that the map $h$ constructed in Proposition \ref{P:GH} may not even be continuous. To construct a local fibration, we smooth $h$ by a convolution as \cite[Theorem $2.6$]{CFG92} and \cite[Proposition $3.11$]{Mi10}. In the following Theorem, the map defined in \eqref{eq:520} is slightly different from that in \cite[Proposition $3.11$]{Mi10}, so that $f_q$ has stronger higher-order estimates than \cite[Proposition $3.11$]{Mi10}.

\begin{thm}\label{T:localfiber}
There exist constants $C>1$ such that for any $q$ with $r=r(q) \ge A_2$, there exists an open neighborhood $\Omega_q$ of $q$ and a map $f_q: \Omega_q \to B(0,\kappa_3r) \subset \R^{n-m}$ satisfying the following properties.
\begin{enumerate}[label=(\roman*)]
\item $f_q$ is a $CK(r/2)$-almost-Riemannian submersion.

\item $|\na^2 f_q|=O(r^{-1}K(r/2))$.

\item If $(M^n,g)$ further satisfies \eqref{cond:HOAF}, then for any $i \ge 3$, $|\na^i f_q| =(r^{1-i}K(r/2))$.

\item The Gromov-Hausdorff distance between any fiber of $f_q$ and $\mathbb T^m_{\infty}$ is bounded by $O(\tilde K(r/4))$.

\end{enumerate}
\end{thm} 

\begin{proof}
We fix a smooth nonincreasing function $\phi(t)$ on $\R$ such that $\phi=1$ on $(-\infty,\frac{1}{3}]$ and $\phi=0$ on $[\frac{2}{3},\infty)$. In addition, we set $\phi_r=\phi(\frac{2t}{\kappa_3^2r^2})$.

On $\hat B(0,3\kappa_3 r) \subset T_q(M)$, we denote the lift of $h$ by $\hat h=h \circ \ex_q$ and set $\rho_v \coloneqq d^2_{\hat g}(v,\cdot)/2$. Then we define on $\hat B(0,2\kappa_3 r)$,
\begin{align}\label{eq:520}
\hat f(w)=\frac{1}{c_0r^{n}}\int_{T_qM}\hat h(v)\phi_r(\rho_v(w))\,dV(v),
\end{align}
where the volume form $dV$ is with respect to $\hat g$ and the constant $c_0>0$ is determined by
\begin{align*}
\int_{\R^n} \phi_r(|z|^2/2)\,dz=c_0 r^n.
\end{align*}

It is clear from the definition \eqref{eq:520} that $\hat f$ descends to a map $f$ on $B(q,\kappa_3 r)$ since $\Gamma(q,2\kappa_3 r)$ acts isometrically on $\hat B(0,\kappa_3 r)$. We claim that the map $f$ satisfies all the properties in the statement. Notice that we only need to prove the corresponding properties of $\hat f$ are satisfied.

It follows from \eqref{eq:e1aa} and \eqref{eq:e1a} that
\begin{align*}
|dV(v)-dv| \le& CK(r/2)dv
\end{align*}
and for any $v,w \in \hat B(0,2\kappa_3 r)$,
\begin{align}\label{eq:521}
|d_{\hat g}(v,w)-|v-w||\le CK(r/2)|v-w| \le CrK(r/2).
\end{align}

By \eqref{eq:521}, it is clear that
\begin{align}\label{eq:521a}
|\rho_v(w)-|v-w|^2/2|\le Cr^2K(r/2).
\end{align}

In addition, from the proof of Proposition \ref{P:GH} we have
\begin{align}\label{eq:522}
|\hat h(v)-v_H| \le CrK(r/2).
\end{align}

Now we write
\begin{align}
&\int\hat h(v)\phi_r(\rho_v(w))\,dV(v)  \notag \\
=& \int\hat h(v)\phi_r(\rho_v(w))\,(dV(v)-dv)+\int\hat h(v)\lc \phi_r(\rho_v(w))-\phi_r(|v-w|^2/2)\rc \,dv \notag\\
&+\int(\hat h(v)-v_H)\phi_r(|v-w|^2/2)\,dv+\int v_H \phi_r(|v-w|^2/2)\,dv.\label{eq:523}
\end{align}

From \eqref{eq:520}, \eqref{eq:521a} and \eqref{eq:522}, we conclude immediately from \eqref{eq:523} that
\begin{align}\label{eq:524}
\int\hat h(v)\phi_r(\rho_v(w))\,dV(v)=\int v_H \phi_r(|v-w|^2/2)\,dv+O(r^{n+1}K(r/2)).
\end{align}

A change of variables shows that
\begin{align}\label{eq:525}
\int v_H \phi_r(|v-w|^2/2)\,dv=w_H\int \phi_r(|z|^2/2)\,dz+\int z_H \phi_r(|z|^2/2)\,dz=c_0 w_H r^n,
\end{align}
where the last equality holds since by parity
\begin{align*}
\int z_H \phi_r(|z|^2/2)\,dz=0.
\end{align*}

Hence, it follows from \eqref{eq:524} and \eqref{eq:525} that
\begin{align}\label{eq:526}
\hat f(w)=w_H+O(rK(r/2)).
\end{align}

By taking the differential of $\hat f$, we have
\begin{align*}
d \hat f_w=\frac{1}{c_0r^{n}}\int \hat h(v)\phi'_r(\rho_v(w)) (d\rho_v)_w\,dV(v),
\end{align*}

It follows from Proposition \ref{P:es3} that
\begin{align}\label{eq:528}
|(d\rho_v)_w-\bar g^{w-v} |\le CrK(r/2).
\end{align}
where $\bar g^{w-v} \coloneqq \la w-v,\cdot\ra$.

Similar to \eqref{eq:523}, by using \eqref{eq:520}, \eqref{eq:521a}, \eqref{eq:522} and \eqref{eq:528} we have 
\begin{align}\label{eq:529}
d \hat f_w =- \frac{1}{c_0r^{n}} \int z_H \phi'_r(|z|^2/2) \bar g^z \,dz+O(K(r/2)).
\end{align}

We choose an orthonormal basis $(e_1\,e_2,\cdots,e_n)$ of $T_qM$ such that $e_i \perp H$ for $ n-m <i \le n$. For $1\le i \ne j \le n-m$, parity shows that
\begin{align*}
\int z_iz_j\phi'_r(|z|^2/2)\,dz=0.
\end{align*}
Moreover, an integration by parts shows that for any $1 \le i \le n-m$,
\begin{align*}
\int z_i^2\phi'_r(|z|^2/2)\,dz=-\int \phi_r(|z|^2/2)\,dz.
\end{align*}

Therefore, it follows from \eqref{eq:529} that
\begin{align}\label{eq:530}
d \hat f_w =\sum_{i=1}^{n-m} e_i \otimes \la e_i,\cdot \ra+O(K(r/2)).
\end{align}

From \eqref{eq:530} we conclude that $f$ is a $CK(r/2)$-almost-Riemannian submersion. This finishes the proof of (i).

To prove (ii) and (iii), we need to estimate all higher-order covariant derivatives of $\hat f$. Recall that from Lemma \ref{lem:dis} we have
\begin{align}
|\na \rho_v| \le Cr \quad \text{and} \quad |\na^2 \rho_v-g| &\le  C  K(r/2).\quad \label{eq:531}
\end{align}

If we further assume \eqref{cond:HOAF}, we have
\begin{align}
|\na^k \rho_v| &\le Cr^{2-k}K(r/2),\quad \forall k\ge 3.  \label{eq:533}
\end{align}

We also need the following formula for the change of orders of higher-order covariant derivatives. For any $k \ge 2$ and tensor $T$,
\begin{align}\label{eq:533a}
|\na_{X_1}\na_{X_2}\cdots \na_{X_k}T-\na_{X_{i_1}}\na_{X_{i_2}}\cdots \na_{X_{i_k}} T| \le C \sum_{a+b=k-2} |\na^a Rm| |\na^b T|,
\end{align}
where $X_i$ are unit vectors and $(i_1,i_2,\cdots,i_k)$ is any permutation of $(1,2,\cdots,k)$. We claim that for any $k \ge 2$, the following Fa\`a di Bruno's formula holds.

\begin{align}\label{eq:534}
\na^k \phi_r(\rho_v)=\sum \frac{k!}{l_1!l_2!\cdots l_k !} \phi^{(L)}_r(\rho_v) \lc \frac{\na \rho_v}{1!} \rc^{l_1}\lc \frac{\na^2 \rho_v}{2!} \rc^{l_2} \cdots \lc \frac{\na^k \rho_v}{k!} \rc^{l_k}+O(r^{-k}K(r/2)),
\end{align}
where the sum is over all nonnegative integer solutions of the Diophantine equation
\begin{align}\label{eq:535}
l_1+2l_2+\cdots+kl_k=k
\end{align}
and $L \coloneqq l_1+l_2+\cdots+l_k$. Here, we denote the $L$th derivative of $\phi_r$ by $\phi^{(L)}_r$ and $\otimes^{l_1}(\na \rho_v)\otimes^{l_2}(\na^2 \rho_v )\otimes \cdots \otimes^{l_k}(\na^k \rho_v)$ by $(\na \rho_v)^{l_1}(\na^2 \rho_v)^{l_2} \cdots (\na^k \rho_v)^{l_k}$.

Indeed, it follows from \eqref{eq:531} and \eqref{eq:533} that
\begin{align}
&|\phi^{(L)}_r(\rho_v) (\na \rho_v)^{l_1}(\na^2 \rho_v)^{l_2}\cdots (\na^k \rho_v)^{l_k}| \notag\\
 \le & Cr^{-2L}r^{l_1}r^{\sum_{i=3}^k (2-i)l_i} (K(r/2))^{\sum_{i=3}^k l_i}=Cr^{-k} (K(r/2))^{L-l_1-l_2} \le Cr^{-k}.\label{eq:533a}
\end{align}

Notice that \eqref{eq:533a} also holds if we change the orders of all $\na^i \rho_v$. Therefore, it follows from \eqref{eq:533a} and the proof of Fa\`a di Bruno's formula, see \cite{Bru57}, that \eqref{eq:534} holds.

We set $X$ be the subset of all solutions of \eqref{eq:535} such that $l_i=0$ for $i\ge 3$ and $X^c$ is the complement of $X$. Similar to \eqref{eq:533a}, if $(l_1,l_2,\cdots,l_k) \in X^c$, we have
\begin{align}
|\phi^{(L)}_r(\rho_v) (\na \rho_v)^{l_1}(\na^2 \rho_v)^{l_2}\cdots (\na^k \rho_v)^{l_k}| \le Cr^{-k} (K(r/2))^{L-l_1-l_2}\le Cr^{-k}K(r/2).\label{eq:536}
\end{align}

Therefore, it is clear from \eqref{eq:536} that
\begin{align*}
\na^k \hat f_w=\frac{1}{c_0r^n} \sum_{(l_1,l_2,\cdots,l_k) \in X} \int \frac{k!}{l_1!l_2!} \hat h(v) \phi^{(L)}_r(\rho_v) \lc \frac{\na \rho_v}{1!} \rc^{l_1}\lc \frac{\na^2 \rho_v}{2!} \rc^{l_2}  \,dV+O(r^{1-k}K(r/2)).
\end{align*}

Similar to \eqref{eq:523}, by using \eqref{eq:520}, \eqref{eq:521a}, \eqref{eq:522} \eqref{eq:528} and \eqref{eq:531} we have 
\begin{align}\label{eq:538}
\na^k \hat f_w=\frac{1}{c_0r^n} \sum_{(l_1,l_2,\cdots,l_k) \in X} \int \frac{k!}{l_1!l_2!} v_H \phi^{(L)}_r(|v-w|^2/2) \lc \frac{\bar g^{w-v}}{1!} \rc^{l_1}\lc \frac{\bar g}{2!} \rc^{l_2}  \,dv+O(r^{1-k}K(r/2)).
\end{align}
where $\bar g= \la \cdot,\cdot \ra$. If $l_1$ is even, then
\begin{align}\label{eq:539}
\int v_H \phi^{(L)}_r(|v-w|^2/2) (\bar g^{w-v})^{l_1}\bar g^{l_2}  \,dv=w_H \int  \phi^{(L)}_r(|z|^2/2) (\bar g^{z})^{l_1}\bar g^{l_2}  \,dz
\end{align}
since by parity
\begin{align*}
\int  z_H\phi^{(L)}_r(|z|^2/2)(\bar g^{z})^{l_1}\bar g^{l_2}  \,dz=0.
\end{align*}

If $l_1$ is odd, then
\begin{align}\label{eq:540}
\int v_H \phi^{(L)}_r(|v-w|^2/2) (\bar g^{w-v})^{l_1}\bar g^{l_2}  \,dv=- \int  z_{H}\phi^{(L)}_r(|z|^2/2) (\bar g^{z})^{l_1}\bar g^{l_2}  \,dz
\end{align}
since by parity
\begin{align*}
w_H\int  \phi^{(L)}_r(|z|^2/2)(\bar g^{z})^{l_1}\bar g^{l_2}  \,dz=0.
\end{align*}

If we denote the function $\phi_r(|z|^2/2)$ by $\psi(z)$, then it follows from \eqref{eq:538}, \eqref{eq:539}, \eqref{eq:540} and the Fa\`a di Bruno's formula that
\begin{align}\label{eq:541}
\na^k \hat f_w=\frac{w_H}{c_0r^n} \int \bar \na^k \psi \,dz-\frac{1}{c_0r^n} \int z_H \bar \na^k \psi \,dz+O(r^{1-k}K(r/2)),
\end{align}
where $\bar \na$ is the Levi-Civita connection of the flat metric on $\R^n$. Here we have used the fact that $\bar \na^i (|z|^2/2)=0$ for $i \ge 3$. It is clear that the first integral in \eqref{eq:541} vanishes. For the second integral, we notice that for $1 \le j \le n-m$,
\begin{align}\label{eq:542}
\int z_j \frac{\partial^k \psi}{\partial z_{i_1}\partial z_{i_2}\cdots \partial z_{i_k}} \,dz=0,
\end{align}
if one of $z_{i_s}$ is different from $z_j$.

In addition, from the integration by parts we have
\begin{align}\label{eq:543}
\int z_j \frac{\partial^k \psi}{\partial z^k_j} \,dz=-\int  \frac{\partial^{k-1} \psi}{\partial z^{k-1}_j} \,dz=0
\end{align}
since $k\ge 2$.

Combining \eqref{eq:542} and \eqref{eq:543}, we conclude that the second integral in \eqref{eq:541} also vanishes. Therefore it follows from \eqref{eq:541} that
\begin{align*}
\na^k \hat f_w=O(r^{1-k}K(r/2))
\end{align*}
and the proof of (ii) and (iii) is complete.

Now we show that (iv) holds for the fiber $F$ through the point $q$ and the general case is similar. We define a fundamental domain
\begin{align*}
P \coloneqq \{s_1c_1+s_2c_2+\cdots+s_mc_m\,\mid 0 < s_i < 1 \quad \text{for all} \quad 1\le i \le m\}.
\end{align*}

We choose a $K(r/2)$-net $\{w_i\}$ of $P$. Notice that $\{w_i\}$ can also be regarded as a $K(r/2)$-net of $\mathbb T^m_q$. For any $w_i$, we set $x_i=\ex_q w_i$ and there exists $y_i \in F$ such that $d(x_i,y_i)$ is the distance between $x_i$ and $F$.

Since $w_i \in \hat U_q$, similar to \eqref{eq:516} we have
\begin{align}\label{eq:542x}
|\hat h(w_i)|=|\hat h(w_i)-(w_i)_H|\le CK(r/2).
\end{align}

From \eqref{eq:542x} the same argument as \eqref{eq:526} shows that
\begin{align}\label{eq:543x}
|\hat f(w_i)|\le CK(r/2).
\end{align}

Since $\hat f$ is a $CK(r/2)$-almost-Riemannian submersion, we have from \eqref{eq:543x}
\begin{align}\label{eq:544}
d(x_i,y_i) \le CK(r/2).
\end{align}

For any $i,j$, we can assume that $|w_i-w_j|=d_{T}(w_i,w_j)=d(x_i,x_j)+O(K(r/2))$. Then it is clear from \eqref{eq:544} that
\begin{align}\label{eq:545}
|d_{T}(w_i,w_j)-d(y_i,y_j)| \le CK(r/2).
\end{align}

Now we compare the intrinsic and extrinsic distances of $F$. For any $a,b \in F$, we denote their lifts by $\hat a,\hat b \in \hat F$ such that $d_{\hat g}(\hat a,\hat b)=d(a,b)$, where $\hat F$ is the lift of $F$. Let $\gamma(t)$ be the geodesic with respect to $\hat g_{\hat F}$ with unit speed connecting $\hat a$ and $\hat b$. Then it follows from (ii) above that
\begin{align*}
|\na_{\dot \gamma(t)}\dot \gamma(t)| \le Cr^{-1}K(r/2).
\end{align*}
In other words, $\gamma(t)$ is almost a geodesic with respect to $\hat g$. By the same argument of Proposition \ref{P:es3} we have
\begin{align*}
d_{\hat F}(\hat a, \hat b)  \le (1+CK(r/2))|\hat a-\hat b| \le (1+CK(r/2))d_{\hat g}(\hat a,\hat b) = (1+CK(r/2))d(a,b).
\end{align*}

Therefore, we have
\begin{align*}
d(a,b) \le d_F(a,b) \le (1+CK(r/2))d(a,b).
\end{align*}
where $g_F$ is the induced metric on $F$. In particular, it follows from \eqref{eq:545} that
\begin{align}\label{eq:546}
|d_{T}(w_i,w_j)-d_{F}(y_i,y_j)| \le CK(r/2).
\end{align}
Now for any $y \in F$, by Lemma \ref{L:neigh} we can find a lift $w \in \hat U_q$. Similar to \eqref{eq:543} we have
\begin{align*}
|w_H|=|\hat f(w)-w_H| \le CK(r/2).
\end{align*}

Notice that there exist $\tau_a$ and $w_i$ such that if $w'=\tau_a(w)$,
\begin{align*}
|w'_c-w_i| \le CK(r/2),
\end{align*}
where $\cdot_c$ is the projection to $\la c_1,c_2,\cdots,c_m \ra$ and $|w'_H|\le CK(r/2)$. Therefore, we have
\begin{align}
&d_{F}(y,y_i) \le Cd_{g}(y,y_i) \notag \\
\le& C(d(y,x_i)+d(x_i,y_i))  \notag\\
\le& C|w_i-w'|+CK(r/2)  \notag\\
\le& C|w_i-w'_c|+C|w'-w'_c|+CK(r/2) \le CK(r/2).\label{eq:549}
\end{align}

Combining \eqref{eq:546} and \eqref{eq:549}, we conclude that
\begin{align*}
d_{GH}(F,\T_q^m) \le CK(r/2).
\end{align*}

Therefore, (iv) is proved by using \eqref{torus5x}.
\end{proof}

Theorem \ref{T:localfiber} endows $\Omega_q$ with the structure of a fiber bundle. Next, we construct a local parametrization.

\begin{prop}\label{P:para}
For any $q$ with $r=r(q)\ge A_2$, there exists an open neighborhood $\Omega_q$ of $q$ and a bundle isomorphism $T=T_q: \Omega_q \to B(0,\kappa_3r)\times \T^m_{\infty}$ satisfying 
\begin{enumerate}[label=(\roman*)]
\item $T$ is a $CK(r/2)$-almost-isometry.

\item $|\na^2 T| =O(r^{-1}K(r/2))$.

\item If $(M^n,g)$ further satisfies \eqref{cond:HOAF}, then for any $i \ge 3$, $|\na^i T|=O(r^{1-i}K(r/2))$.
\end{enumerate}
\end{prop}

\begin{proof}
As in the proof of Theorem \ref{T:localfiber}, we first construct a map from $\hat B(0,\kappa_3 r)$ to $\T_q^m$. Recall that we denote the subspace generated by $\la c_1^q,c_2^q,\cdots,c_m^q\ra$ by $C$.

For any $z \in \Omega_q$, we define $l(z)$ to be the unique minimum point of the following function defined on $\T_q^m$, 
\begin{align*}
F(y)\coloneqq \frac{1}{N_z} \sum_{v\in A_z} d^2_T(y,\bar{v}_c)
\end{align*}
where $A_z=\{v \in \hat U_q\,\mid \ex_q ( v)=z\}$, $N_z=|A_z|$, $d_T$ is the distance function of $\T_q^m$ and $\bar{v}_c$ is the projection of $v$ onto the subspace $C$. From the same arguments as in the proof of Proposition \ref{P:GH}, we conclude that for any $v,w \in A_z$,
\begin{align}\label{eq:551}
d_T(\bar{v}_c,\bar{w}_c) \le CK(r/2).
\end{align}

If we set $\hat l \coloneqq l\circ\ex_q$, then by using the same function $\phi$ as in the proof of Theorem \ref{T:localfiber}, we define for any $w \in \hat B(0,\kappa_3r)$,
\begin{align*}
\hat L(w)=\frac{1}{c_1}\int_{T_qM}\hat l(v)\phi(2\rho_v(w))\,dV(v),
\end{align*}
where the constant $c_1>0$ is determined by
\begin{align*}
\int_{\R^n} \phi(|z|^2)\,dz=c_1.
\end{align*}

By using \eqref{eq:551} and the same proof of Theorem \ref{T:localfiber}, $\hat L$ is almost a projection onto $C$. More precisely, for any $w \in \hat U_q$,
\begin{align}\label{eq:553a}
\hat L(w)=\bar w_c+O(K(r/2)).
\end{align}

Moreover, the following properties hold for $\hat L$.

\begin{enumerate}[label=(\alph*)]
\item $\hat L$ is a $CK(r/2)$-almost-Riemannian submersion.

\item $|\na^2 \hat L|=O(r^{-1}K(r/2))$.

\item If $(M^n,g)$ further satisfies \eqref{cond:HOAF}, then for any $i \ge 3$, $|\na^i \hat L|=O(r^{1-i}K(r/2))$.
\end{enumerate}

Now we define for any $z \in \hat B(0,\kappa_3 r)$
\begin{align}\label{eq:554}
\hat T(z)=(\hat f_q(z),i_q \circ \hat L(z)).
\end{align}
Then it is clear that $\hat T$ descends to a bundle map $T$ from $\Omega_q$ to $B(0,\kappa_3r)\times \T^m_{\infty}$.

Combining Theorem \ref{T:localfiber} and (a),(b),(c) above, we easily conclude that
\begin{enumerate}[label=(\roman*)]
\item $\hat T$ is a $CK(r/2)$-almost-isometry.

\item $|\na^2 \hat T|=O(r^{-1}K(r/2))$.

\item If $(M^n,g)$ further satisfies \eqref{cond:HOAF}, then for any $i \ge 3$, $|\na^i \hat T|=O(r^{1-i}K(r/2))$.
\end{enumerate}

In particular, $\hat T$ restricted on each fiber is a covering map. In fact, from Theorem \ref{T:localfiber} (iv), we easily conclude that $\hat T$ is a bundle isomorphism if $r$ is sufficiently large.
\end{proof}

Now we consider the map 
\begin{align*}
\bar T_q \coloneqq T_q^{-1} \circ (\text{id}\times \pi) : B(0,\kappa_3 r) \times \R^m \overset{\text{id}\times \pi}{\longrightarrow} B(0,\kappa_3 r) \times \T^m_{\infty}\overset{T_q^{-1}}{\longrightarrow} \Omega_q,
\end{align*}
where $\pi$ is the quotient map. With the help of $\bar T_q$, we can do all computations in the local coordinate system.

\begin{cor}\label{C:para}
The pullback metric $(\bar T_q)^*g$ under the local coordinates satisfies
\begin{enumerate}[label=(\roman*)]
\item $g_{ij}-\delta_{ij}=O(K(r/2))$.

\item $\partial_kg_{ij}=O(r^{-1}K(r/2))$.

\item If $(M^n,g)$ further satisfies \eqref{cond:HOAF}, then for any multi-index $\alpha$ with $|\alpha| \ge 2$, $|\partial^{\alpha} g_{ij}|=O(r^{-|\alpha|}K(r/2))$.
\end{enumerate}
\end{cor}

\begin{proof}
(i) follows immediately from Proposition \ref{P:para} (i). Assume $(x_1,x_2,\cdots,x_n)$ are local coordinates induced by $\bar T_q$, then 
\begin{align*}
\partial_k g^{ij}=\partial_k(g(\na x_i,\na x_j))=g(\na_{\partial_k}\na x_i,\na x_j)+g(\na x_i,\na_{\partial_k}\na x_j)=O(r^{-1}K(r/2)),
\end{align*}
where we have used Proposition \ref{P:para} (ii) for the last equation.

Therefore, we have
\begin{align*}
\partial_k g_{ij}=-g_{il}(\partial_kg^{ls})g_{sj}=O(r^{-1}K(r/2)).
\end{align*}

Finally, (iii) follows from Proposition \ref{P:para} (iii) and induction.
\end{proof}

From the bundle isomorphism $T_q:\Omega_q \to B(0,\kappa_3r)\times \T^m_{\infty}$, we can define a torus action $\mu_q$ of $\T^m_{\infty}$ such that for any $x \in \Omega_q$ and $a \in \T^m_{\infty}$ we have 
\begin{align}\label{eq:555}
\mu_q(a)(x)=T_q^{-1}\circ (\overline{a+T_q(x)}),
\end{align}
where $\overline{\,\cdot\,}$ denotes the quotient map onto $\T^m_{\infty}$.

\begin{lem}\label{L:ave1}
For any $a \in \T^m_{\infty}$, $(\mu_q(a))^*g=g+O(r^{-1}K(r/2))$. If $(M^n,g)$ further satisfies \eqref{cond:HOAF}, then $|\na^i((\mu_q(a))^*g)|=O(r^{-1-i}K(r/2))$ for any $i\ge 1$.
\end{lem}
\begin{proof}
We do all the estimates under the coordinates given by $\bar T_q$. Without loss of generality, we assume that $a=(0,0,\cdots,1)$. Then for any $x \in B(0,\kappa_3 r) \times \R^m$,
\begin{align}\label{eq:555a}
((\mu_q(a))^*g)_{ij}(x)=g_{ij}(x+a)=g_{ij}(x)+\int_{0}^1 \partial_n g_{ij}(x+ta)\,dt.
\end{align}
Then it is clear from Corollary \ref{C:para} (ii) that
\begin{align*}
(\mu_q(a))^*g=g+O(r^{-1}K(r/2)).
\end{align*}
If $(M^n,g)$ further satisfies \eqref{cond:HOAF}, by taking the covariant derivative of \eqref{eq:555a} we have
\begin{align*}
\na_k((\mu_q(a))^*g)_{ij}(x)= \int_{0}^1 \na_k\partial_n g_{ij}(x+ta)\,dt.
\end{align*}

On the other hand, from Corollary \ref{C:para} (iii) we have
\begin{align*}
\na_k\partial_n g_{ij}=\partial_k\partial_ng_{ij}-\Gamma_{ki}^s \partial_n g_{sj}-\Gamma_{kj}^s \partial_ng_{is}=O(r^{-2}K(r/2)).
\end{align*}
Therefore,
\begin{align*}
\na_k((\mu_q(a))^*g)_{ij}(x)=O(r^{-2}K(r/2)).
\end{align*}

Finally, the estimates of all higher covariant derivatives follow by induction.
\end{proof}

Locally we can define an invariant metric $\bar g_q$ on $\Omega_q$ by
\begin{align}\label{eq:556}
\bar g_q \coloneqq \frac{1}{\mathfrak{m}(\T^m_{\infty})} \int_{\T^m_{\infty}} (\mu_q(a))^*g\,d\mathfrak{m}(a),
\end{align}
where $d\mathfrak{m}$ is the Lebesgue measure of $\R^m$ and we identify $\T^m_{\infty}$ with its fundamental domain in $\R^m$ . Under the local coordinates, if we denote the components of $\bar g_q$ by $\bar g_{ij}$, then
\begin{align}\label{eq:556a}
\bar g_{ij}(x)= \frac{1}{\mathfrak{m}(\T^m_{\infty})} \int_{\T^m_{\infty}} g_{ij}(x+a)\,d\mathfrak{m}(a).
\end{align}

\begin{prop}\label{P:ave}
For the metric $\bar g_q$ defined in \eqref{eq:556}, the following properties are satisfied on $\Omega_q$.
\begin{enumerate}[label=(\roman*)]
\item $g=\bar g_q+O(r^{-1}K(r/2))$.

\item $|\na \bar g_q|=O(r^{-1}K(r/2))$.

\item The curvature operator $\overline{Rm}_q$ of $\bar g_q$ satisfies $\overline{Rm}_q=O(r^{-2}K(r/2))$.

\item If $(M^n,g)$ further satisfies \eqref{cond:HOAF}, then $|\na^i \bar g_q|=O(r^{-1-i}K(r/2))$ for any $i\ge 1$.
\end{enumerate}
\end{prop}

\begin{proof}
Notice that (i) and (iv) follow immediately from Lemma \ref{L:ave1}, so we only need to prove (ii) and (iii). From Corollary \ref{C:para} (ii) and \eqref{eq:556a} we have
\begin{align}\label{eq:556b}
\partial_k\bar g_{ij}(x)= \frac{1}{\mathfrak{m}(\T^m_{\infty})} \int_{\T^m_{\infty}} \partial_kg_{ij}(x+a)\,d\mathfrak{m}(a)=O(r^{-1}K(r/2)).
\end{align}

Therefore,
\begin{align*}
\na_k \bar g_{ij}=\partial_k \bar g_{ij}-\Gamma_{ki}^s \bar g_{sj}-\Gamma_{kj}^s \bar g_{is}=O(r^{-1}K(r/2)).
\end{align*}

Now we denote the Christoffel symbol and curvature of $\bar g_q$ by $\bar \Gamma_{ij}^k$ and $\bar R_{ijkl}$, respectively. It follows from (i) and \eqref{eq:556b} that
\begin{align}\label{eq:556c}
\bar \Gamma_{ij}^k=\frac{1}{2}\bar g^{kl}(\partial_i \bar g_{jl}+\partial_j \bar g_{il}-\partial_l\bar g_{ij})=O(r^{-1}K(r/2)).
\end{align}

For the curvature, we have
\begin{align}
R_{ijkl}=&\frac{1}{2}(\partial_j\partial_kg_{il}+\partial_i\partial_lg_{jk}-\partial_j\partial_lg_{ik}-\partial_i\partial_kg_{jl})+g_{st}(\Gamma_{jk}^s\Gamma_{il}^t-\Gamma_{jl}^s\Gamma_{ik}^t) \notag \\
=&\frac{1}{2}(\partial_j\partial_kg_{il}+\partial_i\partial_lg_{jk}-\partial_j\partial_lg_{ik}-\partial_i\partial_kg_{jl})+O(r^{-2}K(r/2)). \label{eq:556d}
\end{align}

Similarly, by using \eqref{eq:556c} we have
\begin{align}\label{eq:556e}
\bar R_{ijkl}=\frac{1}{2}(\partial_j\partial_k\bar g_{il}+\partial_i\partial_l\bar g_{jk}-\partial_j\partial_l\bar g_{ik}-\partial_i\partial_k\bar g_{jl})+O(r^{-2}K(r/2)). 
\end{align}

From \eqref{eq:556d}, \eqref{eq:556e} and the formula \eqref{eq:556a}, the proof is complete since
\begin{align*}
\bar R_{ijkl}(x)= \frac{1}{\mathfrak{m}(\T^m_{\infty})} \int_{\T^m_{\infty}} R_{ijkl}(x+a)\,d\mathfrak{m}(a)+O(r^{-2}K(r/2)).
\end{align*}

\end{proof}

 \subsection{Construction of a global fibration}

We first construct a transition map between two local torus fibrations constructed in Theorem \ref{T:localfiber}. The following Proposition is similar to \cite[Lemma $3.12$]{Mi10}, see also \cite[Proposition $5.6$]{CFG92}.

\begin{prop}\label{P:trans}
Given two points $x,y$ such that $r(x),r(y)\ge A_2$ and $d(x,y) \le \kappa_4r$ with $r=r(x)$, if $\Omega_{x,y}\coloneqq \Omega_x \cap \Omega_{y}$, then there exists a map $\phi_{x,y}$ between $f_{y}(\Omega_{x,y})$ and $f_x(\Omega_{x,y})$ satisfying
\begin{enumerate}[label=(\roman*)]
\item $\phi_{x,y}$ is a $CK(r/2)$-almost-isometry.

\item $|f_x-\phi_{x,y}\circ f_{y}| \le CrK(r/2)$.

\item $|d f_x-d\phi_{x,y}\circ d f_{y}| \le CK(r/2)$.

\item $|\na^2 \phi_{x,y}| \le Cr^{-1}K(r/2)$.

\item If $(M^n,g)$ further satisfies \eqref{cond:HOAF}, then $|\na^i \phi_{x,y}|=O(r^{1-i}K(r/2))$ for any $i \ge 3$.
\end{enumerate}
\end{prop}

\begin{proof}
From Lemma \ref{L:neigh}, there exists a $u \in \hat U_x$ which is a lift of $y$. We can identify $T_{y}M$ and $T_xM$ through an isometry $\iota_u$ between the metrics $\hat{g}_{y}=\ex_{y}^*g$ and $\hat{g}_{x}=\ex_{x}^*g$. More precisely, $\iota_u$ is defined by $\iota_u \coloneqq \text{Exp}_u \circ ( d_u\ex_x)^{-1}$, where $\text{Exp}_u$ is the exponential map at $u$. It is clear that
\begin{align*}
\ex_x \circ \iota_u=\ex_{y}.
\end{align*}

Now we define on $f_{y}(\Omega_{x,y})$,
\begin{align}\label{eq:5b002}
\phi_{x,y} \coloneqq f_x \circ \ex_{y}.
\end{align}

Then it is clear from the definition that
\begin{align}\label{eq:5b003}
\phi_{x,y} \circ f_{y} \circ \ex_x =\hat f_x \circ \tilde{f}_{y},
\end{align}
where $\hat f_x=f_x\circ \ex_x$ and $\tilde{f}_{y}\coloneqq \iota_u \circ \hat f_{y}\circ \iota_u^{-1}=\iota_u \circ f_{y} \circ \ex_{y}\circ \iota_u^{-1}$. As before, we set $H \subset T_xM$ to be the subspace which is perpendicular to $\la c_1^x,c_2^x,\cdots,c_m^x\ra$ and $H'$ the corresponding subspace in $T_{y}M$. Moreover, we define $\tilde H \coloneqq \iota_u(H')$ to be the $(n-m)$-dimensional submanifold through $u$.

As before, we assume that $c_i^{x'}$ is the sliding of $c_i^x$ to $y$ along a minimizing geodesic from $x$ to $y$. Moreover as \eqref{eq:B36} we set
\begin{align}
c_i^{y}=\sum_{j=1}^m k_{ji}c_j^{x'}, \label{eq:5b003a}
\end{align}
for some $K_{x,y}=(k_{ij}) \in \text{GL}(m,\Z)$. From the definition of the sliding, the geodesic loop $c_i^{x'}$ can be obtained by the geodesic with respect to $\hat g_x$ connecting $u$ and $\tau_{c_i^x}(u)$.

We denote the initial tangent vectors of $c_i^y$ and $c_i^{x'}$ by $v_i$ and $v_i'$ respectively and set $V_i=(d\iota_u)_0(v_i)$ and $V'_i=(d\iota_u)_0(v'_i)$. Notice that $V_i,V_i' \in T_u(T_xM)$, which is identified with $T_xM$ by geodesic coordinates. It follows from \eqref{eq:5b003a} and Lemma \ref{L:tr} (i) that
\begin{align*}
v_i=\sum_{j=1}^m k_{ji}v_j'+O(r^{-1}K(r/2))
\end{align*}
and hence
\begin{align}
V_i=\sum_{j=1}^m k_{ji}V_j'+O(r^{-1}K(r/2)). \label{eq:5b003b}
\end{align}

On the one hand, by Proposition \ref{P:es3} we have
\begin{align}\label{eq:5b004}
|\tau_{c_i^x}(u)-(u+V'_i)| \le CK(r/2).
\end{align}

On the other hand, from Lemma \ref{L:tr}(ii) we have
\begin{align}\label{eq:5b005}
|\tau_{c_i^x}(u)-(u+c_i^x)| \le CK(r/2).
\end{align}
Combining \eqref{eq:5b004} and \eqref{eq:5b005} we have
\begin{align*}
|c_i^x-V'_i| \le CK(r/2).
\end{align*}

Therefore, it follows from \eqref{eq:5b003b} that
\begin{align}\label{eq:5b006a}
V_i=\sum_{j=1}^m k_{ji}c_j^x+O(K(r/2)).
\end{align}

Now we set $\hat U \coloneqq \hat B(0,\kappa_4r(x)) \cap \hat B(u,\kappa_4r(y))$. For any $z \in \hat U \cap \tilde H$, there exists a vector $W$ which is perpendicular to $\la V_1,V_2,\cdots,V_m \ra$ with respect to $g_{y}$ such that
\begin{align*}
z=\text{Exp}_u(W).
\end{align*}

From Proposition \ref{P:es3} we have
\begin{align}\label{eq:5b007}
|z-(u+W)| \le CrK(r/2).
\end{align}

It follows from \eqref{eq:5b006a}, \eqref{eq:5b007} and the fact that $\hat g_{y}$ is $K(r/2)$ close to $\hat g_x$ that
\begin{align*}
|z-(u+W_H)| \le CrK(r/2).
\end{align*}

Therefore, for any $w \in \hat U$,
\begin{align}\label{eq:5b009}
|w_{H'}-w_H-(u-u_H)|\le CrK(r/2),
\end{align}
where we have used the same notations for $w,w_{H'}$ and their images under $\iota_u$.

From Theorem \ref{T:localfiber}, we have
\begin{align*}
|\hat f_x(w)-w_H|+|\hat f_{y}(w)-w_{H'}| \le CrK(r/2)
\end{align*}
and hence from \eqref{eq:5b009},
\begin{align*}
|\tilde f_{y}(w)-\hat f_x(w)-(u-u_H)| \le CrK(r/2).
\end{align*}

Composing with $\hat f_x$, we conlude that
\begin{align*}
|\hat f_x \circ \tilde f_{y}(w)-\hat f_x(w)| \le CrK(r/2).
\end{align*}

Therefore, we have proved (ii) from the formula \eqref{eq:5b002} since $\ex_x$ is surjective.

For any $z \in \hat U$, we set $z'=\iota_u^{-1}(z)$. From Theorem \ref{T:localfiber} (i), we know that $(d\hat f_{y})_{z'}$ is $CK(r/2)$-close to the projection onto $H'$. If we set $\tilde V_i=(d\iota_u)_{z'}(v_i)$, then $(d\tilde f_{y})_{z}$ is $CK(r/2)$-close to the projection onto $\la\tilde V_1,\tilde V_2,\cdots,\tilde V_m  \ra^{\perp}$. Now we consider the geodesic with respect to $\hat g_{y}$ from $z'$ to $\tau_{c_i^{y}}(z')$ and denote the initial tangent vector at $z'$ by $w_i$. From Proposition \ref{P:es3} and Lemma \ref{L:tr} we conclude that
\begin{align}\label{eq:5b012}
|v_i-w_i| \le |v_i-(\tau_{c_i^{y}}(z')-z')|+|w_i-(\tau_{c_i^{y}}(z')-z')| \le CK(r/2).
\end{align}

If we set $a_i=\sum_{j=1}^m k_{ji}c_j^x$, then it is clear that the loop corresponding to the geodesic connecting $z'$ and $\tau_{c_i^y}(z')$ is the sliding of $a_i$ to $\ex_y(z')$ along $\beta*\gamma$ where $\gamma$ is image under $\ex_x$ of the geodesic connecting $u$ and $z$. In particular, $\iota_u(\tau_{c_i^y}(z'))=\tau_{a_i}(z)$.

Then it follows from Proposition \ref{P:es3}, Lemma \ref{L:tr}(ii) and \eqref{eq:5b012} that
\begin{align*}
|a_i-\tilde V_i| \le& |a_i-(\tau_{a_i}(z)-z)|+|(\tau_{a_i}(z)-z)-((\iota_u)_*)_{z'}(w_i)|+|((\iota_u)_*)_{z'}(w_i)-((\iota_u)_*)_{z'}(v_i)| 
\\
\le& CK(r/2)+|v_i'-v_i| \le CK(r/2). 
\end{align*}

Therefore, it is easy to conclude that $(d\hat f_{y})_{z'}$ is $CK(r/2)$-close to the projection onto $H$. By taking the differential of \eqref{eq:5b003} we have
\begin{align}\label{eq:5b014}
d(\phi_{x,y} \circ f_{y} \circ \ex_x) =d\hat f_x \circ d\tilde{f}_{y}.
\end{align}

Since $d\hat f_x$ is $CK(r/2)$-close to the projection onto $H$, we conclude from \eqref{eq:5b014} that
\begin{align}\label{eq:5b015}
|d f_x-d\phi_{x,y}\circ d f_{y}| \le CK(r/2)
\end{align}
and hence (iii) is proved.

For any vector $W$ tangent to $f_{y}(\Omega_{x,y})$, we denote a horizontal lift of $W$ with respect to $f_{y}$ by $W'$. Since $d\hat f_x$ and $d\tilde f_{y}$ are $CK(r/2)$-close, $W'$ is $CK(r/2)$-close to a horizontal lift $\hat W$ of $W$ with respect to $f_x$. As both $f_x$ and $f_{y}$ are $CK(r/2)$-almost-Riemannian submersions, we conclude from \eqref{eq:5b015} that
\begin{align*}
||d\phi_{x,y}(W)|-|W|| \le |d\phi_{x,y}(df_{y}(W'))-df_x(W')|+|df_x(W')-df_x(\hat W)| \le CK(r/2)|W|.
\end{align*}
Therefore, $\phi_{x,y}$ is a $CK(r/2)$-almost-isometry and the proof of (i) is complete.

Finally, the proof of (iv) and (v) follow from From Theorem \ref{T:localfiber} (ii) and (iii) by taking the higher covariant derivatives of the formula \eqref{eq:5b003}.
\end{proof}

Next, we prove the following lemma, which is similar to \cite[Lemma $3.13$]{Mi10}.

\begin{lem}\label{L:trans1}
Given points $x,y,z$ such that $r=r(x) \ge A_2$ and their respective distances are bounded by $\kappa_4r$, then we have
\begin{enumerate}[label=(\roman*)]
\item $|\phi_{x,z}-\phi_{x,y}\circ \phi_{y,z}| \le CrK(r/2)$.

\item $|d\phi_{x,z}-d(\phi_{x,y}\circ \phi_{y,z})| \le CK(r/2)$.

\item $|\na^2 (\phi_{x,y}\circ \phi_{y,z})| \le Cr^{-1}K(r/2)$.

\item If $(M^n,g)$ further satisfies \eqref{cond:HOAF}, then $|\na^i (\phi_{x,y}\circ \phi_{y,z})|=O(r^{1-i}K(r/2))$ for any $i \ge 3$.
\end{enumerate}
\end{lem}

\begin{proof}
It is clear from Proposition \ref{P:trans} (i),(ii) that
\begin{align*}
&|(\phi_{x,z}-\phi_{x,y}\circ \phi_{y,z})\circ f_z| \\
\le & |\phi_{x,z}\circ f_z-f_x|+|\phi_{x,y}\circ \phi_{y,z}\circ f_z-\phi_{x,y}\circ f_y|+|\phi_{x,y} \circ f_y -f_x|\le CrK(r/2).
\end{align*}

Since $f_z$ is surjective, the proof of (i) is complete. (ii) is proved similarly by using Proposition \ref{P:trans} (i),(iii). Finally, (iii) and (iv) follow immediately from Proposition \ref{P:trans} (iv),(v).
\end{proof}

In the setting of Proposition \ref{P:trans}, there are two fibrations $f_x$ and $\phi_{x,y}\circ f_y$ on $\Omega_{x,y}$ over $f_x(\Omega_{x,y})$. It follows from Proposition \ref{P:para} that there exists a natural parametrization of $\phi_{x,y}\circ f_y$, defined by
\begin{align}\label{eq:5b016}
T'_y=(\phi_{x,y}\circ \pi_1 \circ T_y,\pi_2 \circ T_y)=(\phi_{x,y}\circ f_y,\pi_2 \circ T_y)
\end{align}
where $\pi_i$ are the projections of $B(0,\kappa_3r)\times \T_{\infty}^m$. Similarly, we define the torus action $\mu_y'$ which acts on the fibers of $\phi_{x,y} \circ f_y$.

By using the same notations as in the proof of Proposition \ref{P:trans}, for any $z\in f_x(\Omega_{x,y})$, we set $z'=\phi_{x,y}^{-1}(z)$ and $\bar z=\iota_u(z')$. Then, we define the translation $H_{x,y}^z$ of $\T_{\infty}^m$ such that for any $w \in \T_{\infty}^m$,
\begin{align*}
H^z_{x,y}(w)\coloneqq\overline{w+\pi_2 \circ \hat T_x(\bar z)}.
\end{align*}

In addition, we define the map $H_{x,y}$ on $f_x(\Omega_{x,y})\times \T_{\infty}^m \subset B(0,\kappa_3r)\times \T_{\infty}^m$ by 
\begin{align*}
H_{x,y}((z,w))\coloneqq (z,H^z_{x,y}(w)).
\end{align*}

Now we define a map of $\Omega_{x,y}$,
\begin{align}\label{eq:self}
\Phi_{x,y} \coloneqq T_x^{-1} \circ H_{x,y}\circ (\text{id} \times L_{x,y})\circ T'_y,
\end{align}
where $T_x$ and $T_y$ are from Proposition \ref{P:para} and $L_{x,y}$ is from \eqref{eq:auto}. The map $\Phi_{x,y}$ is illustrated by the following commutative diagram.

\[
  \begin{tikzcd}[column sep =huge] 
     B(0,\kappa_3r)\times \T_{\infty}^m \arrow{r}{\text{id} \times L_{x,y}} &B(0,\kappa_3r)\times \T_{\infty}^m  \arrow{r}{H_{x,y}} & B(0,\kappa_3r)\times \T_{\infty}^m \arrow{d}{T^{-1}_x} \\
      \Omega_{x,y} \arrow{rr}{\Phi_{x,y}}\arrow{u}{T'_y} & &\Omega_{x,y}
  \end{tikzcd}
\]

Notice that by the definition, $\Phi_{x,y}$ is a fiber-preserving map between the fibrations $\phi_{x,y} \circ f_y$ and $f_x$. Now we show that $\Phi_{x,y}$ is close to the identity map.

\begin{prop}\label{P:trans2}
For the map $\Phi_{x,y}$ constructed in \eqref{eq:self}, the following properties are satisfied.
\begin{enumerate}[label=(\roman*)]
\item For any $z \in \Omega_{x,y}$, $d(z,\Phi_{x,y}(z)) \le CK(r/2)$.

\item $\Phi_{x,y}$ is a $CK(r/2)$-almost-isometry.

\item $|\na^2 \Phi_{x,y}| \le Cr^{-1}K(r/2)$.

\item If $(M^n,g)$ further satisfies \eqref{cond:HOAF}, then $|\na^i \Phi_{x,y}|=O(r^{1-i}K(r/2))$ for any $i \ge 3$.
\end{enumerate}
\end{prop}

\begin{proof}
Notice that we only need to prove (i), since (ii), (iii) and (iv) follow from Proposition \ref{P:para} and Proposition \ref{P:trans}. For any $a \in \Omega_{x,y}$, by Lemma \ref{L:neigh} there exists a lift $w$ of $a$ to $T_yM$ such that $w' \in \hat U_y$. Moreover, we set $z=f_x(a)$.

From the definitions \eqref{eq:554} and \eqref{eq:5b016}, if we set $\hat T'_y=T'_y\circ \exp_y$, then 
\begin{align*}
\pi_2\circ \hat T'_y=i_y \circ \hat L_y.
\end{align*}

Therefore, it follows from \eqref{eq:553a} that the fact that $i_y$ is a $CK(r/2)$-almost isometry that
\begin{align}\label{eq:5b017}
d_T(i_y \circ \bar w'_c,\pi_2 \circ \hat T'_y(w')) \le CK(r/2),
\end{align}
where as before $w'_{c'}$ is the projection of $w'$ onto $\la c_1^y,c_2^y,\cdots,c_m^y \ra$ and the $\overline{\,\cdot\,}$ is the quotient map of $\T_{y}^{m}$, respectively.

From Lemma \ref{L:neigh}, there exists a point $z_1 \in \hat U_x$ which is a lift of $\ex_y(z')$ and hence we can find a $c \in \Gamma(x,\kappa_1r)$ such that $\tau_c(\bar z)=z_1$. Now we define $\bar\iota_u=\tau_c \circ \iota_u$, $w=\bar\iota_u(w')$ and consider the fibration $\bar f_x \coloneqq \bar\iota_u \circ \hat f_y \circ \bar \iota_u^{-1}$. It is clear that $w$ is on the fiber of $\bar f_x$ through $z_1$ and by the same proof of Proposition \ref{P:trans}, $d\bar f_x$ is almost a projection onto $H$, up to an error $CK(r/2)$. Since by our choice the distance between $w$ and $z_1$ is uniformly bounded, we conclude that 
\begin{align*}
w_H=(z_1)_H+O(K(r/2))
\end{align*}
and hence
\begin{align}\label{eq:5b019}
w_c=w-w_H=w-z_1+(z_1)_c+O(K(r/2)),
\end{align}
where $w_c$ is the projection of $w$ onto $\la c_1^x,c_2^x,\cdots,c_m^x \ra$.

Now we define a linear map $F$ from $T_y(M)$ to $T_x(M)$ such that on $\la c_1^y,c_2^y,\cdots,c_m^y \ra$,
\begin{align*}
F=I_{x,y},
\end{align*}
where $I_{x,y}$ is a linear map from $\la c_1^y,c_2^y,\cdots,c_m^y \ra$ to $\la c_1^x,c_2^x,\cdots,c_m^x \ra$ defined by
\begin{align*}
I_{x,y}(c_i^y)=\sum_{j=1}^m k_{ji}c_j^x.
\end{align*}
Moreover, $F=(d\bar\iota_u)_{z'}$ on $\la c_1^y,c_2^y,\cdots,c_m^y \ra^{\perp}$. We claim that 
\begin{align}\label{eq:5b020}
F=(d\bar\iota_u)_{z'}+O(K(r/2)).
\end{align}

Indeed, for any $1 \le i \le m$, we consider the geodesic from $z'$ to $\tau_{c_i^y}(z')$ and set the initial tangent vector to be $v_i$ and define $V_i=(d\bar\iota_u)_{z'}(v_i)$. Then by the same argument as in Proposition \ref{P:trans}, we have 
\begin{align}\label{eq:5b021}
v_i=c_i^y+O(K(r/2)).
\end{align}

Moreover, it is clear that
\begin{align*}
\bar\iota_u(\tau_{c_i^y}(z'))=\Pi_{j=1}^m \tau_{c_j^x}^{k_{ij}}(z_1)
\end{align*}
and hence
\begin{align}\label{eq:5b022}
V_i=\sum_{j=1}^m k_{ji}c_j^x+O(K(r/2)).
\end{align}

From \eqref{eq:5b021} and \eqref{eq:5b022}, the claim \eqref{eq:5b020} holds.

Now we consider a geodesic $\gamma(t)$ from $z'$ to $w'$, then we have
\begin{align*}
\dot \gamma(0)=w'-z'+O(K(r/2))=w'-w'_H+O(K(r/2))=w'_{c'}+O(K(r/2)).
\end{align*}

Therefore, it follows from the claim that
\begin{align}\label{eq:5b024}
w-z_1= (d\bar\iota_u)_{z'}(\dot \gamma(0))+O(K(r/2))=I_{x,y}(w'_{c'})+O(K(r/2)).
\end{align}

By \eqref{eq:5b019} and \eqref{eq:5b024}, we have
\begin{align}\label{eq:5b025}
w_c=I_{x,y}(w'_{c'})+z_1-(z_1)_H+O(K(r/2)).
\end{align}

On the other hand, since the distance from $w$ to $H$ is uniformly bounded, similar to \eqref{eq:5b017} we have
\begin{align}\label{eq:5b026}
d_T(i_x \circ \bar w_c,\pi_2 \circ \hat T_x(w)) \le CK(r/2).
\end{align}

From \eqref{eq:5b017}, \eqref{eq:5b025}, \eqref{eq:5b026} and the fact that $\hat T_x(z_1)=\hat T_x(\bar z)$, the proof is complete.
\end{proof}

Notice that for any $\theta <\kappa_4$ and $x$ with $r=r(x) \ge A_2$, $B(x,\theta r)$ may not be a saturated set with respect the fibration $f_x$. If we define $\tilde B(x,\theta r)=f_x^{-1}(B(0,\theta r))$, it is clear from Theorem \ref{T:localfiber} that
\begin{align*}
B\lc x,\theta r(1-CK(r/2))\rc \subset \tilde B ( x,\theta r) \subset B\lc x,\theta r(1+CK(r/2))\rc.
\end{align*}

Therefore, we will not distinguish between $\tilde B ( x,\theta r)$ and $B ( x,\theta r)$ in the following arguments.

Next, we can modify the fibration $f_x$ and its torus action so that the new fibration is compatible with $f_y$ and their torus actions differ by an automorphism, see also \cite[Lemma $3.14$]{Mi10}.

\begin{prop}\label{P:trans3}
Given two points $x,y$ with $r(x),r(y)\ge A_2$ and constants $\sigma<\theta<\kappa_4$, we assume that $B(x,\theta r(x))$ and $B(y, \theta r(y))$ have nonempty intersection.
\begin{enumerate}[label=(\roman*)]
\item There exists a fibration $\tilde f_x$ on $B(x,\theta r(x))$ with the same properties as $f_x$ such that
\begin{align*}
\tilde f_x= \phi_{x,y}\circ f_y
\end{align*}
on $B(x,\sigma r(x))\cap B(y,\sigma r(y))$. Moreover, the new fibration $\tilde f_x$ coincides with the old $f_x$ wherever $f_x= \phi_{x,y}\circ f_y$ on $B(x,\theta r(x))$.

\item There exists a map $\tilde T_x$ from $B(x,\sigma r(x))$ to $B(0,\sigma r(x))\times \T_{\infty}^m$ which is a parametrization of $\tilde f_x$ and satisfies all properties in Proposition \ref{P:para}. 

\item There exists a torus action $\tilde \mu_x$ on the fibers of $\tilde f_x$ such that $\tilde \mu_x$ agrees with $\mu_y$ up to an automorphism $L_{x,y}$ on $B(x,\sigma r(x))\cap B(y,\sigma r(y))$. Moreover, $\tilde \mu_x$ satisfies the same conclusions in Lemma \ref{L:ave1}.

\item The transition map between $\tilde f_x$ and $f_y$ has the image contained in the group $\mathbb T^m \rtimes G(A^{\infty})$, where $A^{\infty}$ and $G(A^{\infty})$ are defined in \eqref{torus0a} and \eqref{def:group} respectively.
\end{enumerate}
\end{prop}

\begin{proof}
To prove (i) and (ii), we choose a cutoff function $\phi(z)$ on $\R^{n-m}$ such that $\phi=1$ on $B(0,\sigma)$ and $\phi=0$ outside $B(0,\theta)$. For any $w \in \T_{\infty}^m$, we define $\lambda(w)=\phi(\frac{\pi_1(w)}{r})$. Moreover, we define $\tilde \Phi_{x,y} \coloneqq T'_{y} \circ \Phi_{x,y}^{-1} \circ T_x^{-1}$ and $\Psi_{x,y}=(1-\lambda)\text{id}+\lambda \tilde \Phi_{x,y}$, where $T'_{y}$ is from \eqref{eq:5b016} and the linear structure is from $B(0,\theta r) \times \T_{\infty}^m$. 

Now we define the map $\tilde T_x$ from $B(x,\theta r)$ to $B(0,\theta r) \times \T_{\infty}^m$ by
\begin{align*}
\tilde T_x \coloneqq  \Psi_{x,y}^{-1}\circ T_x.
\end{align*}

Moreover, we set
\begin{align*}
\tilde f_x \coloneqq  \pi_1 \circ \tilde T_x.
\end{align*}

From Proposition \ref{P:trans2}, it is easy to see that $\tilde f_x$ and $\tilde T_x$ satisfy all properties in (i) and (ii). 

As in \eqref{eq:555}, we define for any $a \in \T_{\infty}^m$ and $w \in \T(z)$,
\begin{align*}
\tilde \mu_x(a)(w)=\tilde T_x^{-1}\circ (\overline{a+\tilde T_x(w)}).
\end{align*}

It is clear that $\tilde\mu_{x}$ and $\mu_y$ differ by an automorphism. Indeed, from our definition for any $v \in B(x,\sigma r(x))\cap B(y,\sigma r(y))$, if we set $T'_y(v)=(z,b)$, then
\begin{align}\label{eq:transfinite}
\tilde T_x(v)=T_x \circ \Phi_{x,y}(v)=(z,\overline{L_{x,y}(b)+\pi_2 \circ \hat T_x(\bar z)})
\end{align}
where $L_{x,y}$ is defined in \eqref{eq:auto}.

Therefore,
\begin{align*}
\tilde T_x(\mu_y(a)(v))=(z,\overline{L_{x,y}(a)+L_{x,y}(b)+\pi_2 \circ \hat T_x(\bar z)})=\tilde T_x(\tilde \mu_x(L_{x,y}(a))(v)).
\end{align*}

Then we immediately have
\begin{align*}
\tilde\mu_{x}(L_{x,y}(a))(v)= \mu_y(a)(v).
\end{align*}

Finally, (iv) follows from \eqref{eq:transfinite} and Proposition \ref{P:finite}.
\end{proof}

The next lemma is similar to \cite[Lemma $3.15$]{Mi10}.
\begin{lem}\label{L:trans4}
Given three points $x,y,z$ with $r(x),r(y),r(z)\ge A_2$ and constants $\sigma<\theta<\kappa_4$, we assume that $B(x,\theta r(x))$, $B(y, \theta r(y))$ and $B(z,\theta r(z))$ have nonempty intersection. Then there exists a diffeomorphism $\tilde \phi_{x,z}$ with the same properties as $\phi_{x,z}$ such that
\begin{align*}
\tilde \phi_{x,z}= \phi_{x,y}\circ \phi_{y,z}
\end{align*}
on $f_z\lc B(x,\sigma r(x))\cap B(y,\sigma r(z))\cap B(z,\sigma r(z)) \rc$. Moreover, $\tilde \phi_{x,z}$ coincides with the old $\phi_{x,z}$ wherever $\phi_{x,z}= \phi_{x,y}\circ \phi_{y,z}$.
\end{lem}

\begin{proof}
By using the same cutoff function $\lambda$ as in the proof of Proposition \ref{P:trans3}. We define
\begin{align*}
\tilde \phi_{x,z}(v)= \lambda(v)\phi_{x,y}\circ \phi_{y,z}(v)+(1-\lambda(v)) \phi_{x,z}(v). 
\end{align*}
Then it follows Proposition \ref{P:trans} and Lemma \ref{L:trans1} that such $\tilde \phi_{x,z}$ is the required map.
\end{proof}

Now we are ready to construct a global fibration on the end, see also \cite[Theorem $3.16$]{Mi10}.

\begin{thm}\label{T:met}
Let $(M^n,g)$ be a complete Riemannian manifold with \eqref{cond:AF} and \eqref{cond:SHC}. There exist an integer $0\le m \le n-1$, a flat torus $\T_{\infty}^m$, a compact set $K \subset M^n$ such that $M^n \backslash K$ is endowed with a $m$-dimensional torus fibration $f$ over a smooth open manifold $Y$. Moreover, there exists an open cover $\Omega_i$ of $M^n\backslash K$ satisfying the following properties.

\begin{enumerate}[label=(\roman*)]
\item For any $i$, there exists a bundle diffeomorphim $T_i: \Omega_i \to U_i \times \T_{\infty}^m$, where $U_i =f(\Omega_i) \subset \R^{n-m}$, such that $\pi_1 \circ T_i=f$ and $T_i$ satisfies the estimates in Proposition \ref{P:para}.

\item For any $i$, there exists a $m$-dimensional torus action $\mu_i$ of on $\Omega_i$ which satisfies the estimates in Lemma \ref{L:ave1}. Moreover, on $\Omega_i \cap \Omega_j \ne \emptyset$, $\mu_i$ and $\mu_j$ differ by an automorphism.

\item The structure group of $f$ is contained in the group $\mathbb T^m \rtimes G(A^{\infty})$.
\end{enumerate}
\end{thm}

\begin{proof}
We only need to make the local fibrations and torus actions compatible. The strategy originates from the work of Cheeger-Gromov \cite{CG90}. We sketch the proof following the argument of Minerbe \cite[Theorem $3.16$]{Mi10}.

By taking a maximal set of point $x_i,i\in I$, such that $d(x_i,x_j) \ge \frac{\kappa_4}{8}\max\{r(x_i),r(x_j)\}$, $\{B(x_i,\frac{\kappa_4r(x_i)}{2})\}$ forms a uniformly locally finite covering of $M$. By dropping finitely many indices, there exists a fibration $f_i=f_{x_i}$ on $B(x_i,\kappa_4r(x_i))$ and a torus action $\mu_i=\mu_{x_i}$ by $\T_{\infty}^m$ on the fibers of $f_i$. We denote the minimal saturated set containing $B(x_i,\alpha r(x_i))$ by $\Omega_i(\alpha)$. Now we can divide $I$ into packs $S_1,S_2,\cdots,S_N$ such that if $x_i\ne x_j \in S_k$, then $\Omega_i(\kappa_4)\cap \Omega_j(\kappa_4)=\emptyset$.

Next we consider $2^N-1$ stages of operations such that each stage is indexed by $\mathcal{A}=\{a_1<a_2<\cdots<a_k\}$, where $a_i$ is an integer in $[1,N]$. Moreover, we order all stages so that
 \begin{align*}
\mathcal{A}=\{a_1<a_2<\cdots<a_k\} \prec \mathcal{B}=\{b_1<b_2<\cdots<b_l\}
\end{align*}
if one of the conditions holds:
\begin{enumerate}
\item $a_1<b_1$;

\item $a_1=b_1$ and $k>l$;

\item $k=l$ and there exists an $2 \le i_0 \le k$ such that $a_i=b_i$ for $i \le i_0$ and $a_{i_0}<b_{i_0}$.
\end{enumerate}

Now we denote the rank of $\mathcal{A}$ in this order by $|\mathcal{A}|$ and set $\alpha_t \coloneqq \kappa_4 \lc\frac{1}{2} \rc^{\frac{t}{2^N}}$. Each stage $\mathcal{A}=\{a_1<a_2<\cdots<a_k\}$ consists of steps which are indexed by all elements $\mathcal{I}=(i_1,i_2,\cdots,i_k)$ of $S_{a_1} \times S_{a_2} \cdots \times S_{a_k}$. At step $\mathcal{I}$, we do the modifications in the following way.
\begin{enumerate}
\item For any $2 \le p \le k$, by considering the pair $(f_{i_1},f_{i_p})$, we obtain the new fibration $\tilde f_{i_p}$ and torus action $\tilde \mu_i$ from Proposition \ref{P:trans3} on $\Omega_{i_p}(\alpha_{|\mathcal{A}|})$.

\item For any $2 \le p<q \le k$, by considering the triples $(\phi_{i_1,i_q},\phi_{i_p,i_1},\phi_{i_p,i_q})$, we obtain the new transition diffeomorphism $\tilde \phi_{i_p,i_q}$ from Lemma \ref{L:trans4} on $\Omega_{i_p}(\alpha_{|\mathcal{A}|})$.
\end{enumerate}

Then it is clear that for any $2 \le p<q \le k$, 
\begin{align*}
\tilde \phi_{i_p,i_q} \circ \tilde f_{i_p}=\tilde f_{i_q}
\end{align*}
on $\Omega_{i_1}(\alpha_{|\mathcal{A}|+1})\cap \Omega_{i_p}(\alpha_{|\mathcal{A}|+1}) \cap \Omega_{i_q}(\alpha_{|\mathcal{A}|+1})$.

Notice that at each stage, we can perform all modifications for each step at the same time. Then we pass to the next stage. After all $2^N-1$ stages, we obtain local fibration $f_i$ and torus action $\mu_i$ on $\Omega_i \coloneqq \Omega_i(\frac{\kappa_4}{2})$ and the transition diffeomorphisms $\phi_{i,j}$ such that $\phi_{i,j} \circ f_j=f_i$ on $\Omega_i \cap \Omega_j$.

Now the open manifold $Y$ is constructed by attaching all $f_i(\Omega_i)$ by the transition diffeomorphisms $\phi_{i,j}$.
\end{proof}

\begin{thm} \label{T:ave1}
In the same setting as Theorem \ref{T:met}, there exists a metric $\bar g$ on $M^n \backslash K$ satisfying the following properties.
\begin{enumerate}[label=(\roman*)]
\item For any $i$, $\bar g$ is invariant under the action of $\mu_i$.

\item $g=\bar g+O(r^{-1}K(r/2))$.

\item $|\na \bar g|=O(r^{-1}K(r/2))$.

\item The curvature of $\bar g$, denoted by $\overline{Rm}$, satisfies $\overline{Rm}=O(r^{-2}K(r/2))$.

\item If $(M^n,g)$ further satisfies \eqref{cond:HOAF}, then $|\na^i \bar g|=O(r^{-1-i}K(r/2))$ for any $i\ge 1$.
\end{enumerate}
\end{thm}

\begin{proof}
As in \eqref{eq:556}, we define on $\Omega_i$ an invariant metric $\bar g_i$ by
\begin{align*}
\bar g_i \coloneqq \frac{1}{\mathfrak{m}(\T^m_{\infty})} \int_{\T^m_{\infty}} (\mu_i(a))^*g\,d\mathfrak{m}(a).
\end{align*}
Since two torus actions differ by an automorphism, $\bar g_i=\bar g_j$ on $\Omega_i \cap \Omega_j$. So we define $\bar g$ such that $\bar g=\bar g_i$ on $\Omega_i$. Finally, (ii), (iii) (iv) and (v) follow from Proposition \ref{P:ave}.
\end{proof}

The open manifold $Y$ obtained in Theorem \ref{T:met} can be equipped with a metric $h$ which is the pushdown of $\bar g$. More precisely, for any $v \in T_yY$, $h(v,v)=\bar g(w,w)$ where $w \in T_xM$ such that $w$ is perpendicular to the fiber and $df_x(w)=v$. It is clear that $f: (M^n \backslash K,\bar g) \to (Y,h)$ becomes a Riemannian submersion. Now we fix a point $p_0 \in Y$ and set $\rho(y) \coloneqq d_h(p_0,y)$ for any $y \in Y$. It is clear that for $x \in M^n$,
\begin{align*}
\lim_{x \to \infty} \frac{r(x)}{\rho(f(x))}=1.
\end{align*}

\begin{prop} \label{P:ALE}

For the open manifold $Y$, the following properties are satisfied.
\begin{enumerate}[label=(\roman*)]
\item $Y$ has the Euclidean volume growth.

\item $|Rm_h|=O(\rho^{-2}K(\rho/4))$.

\item If $(M^n,g)$ further satisfies \eqref{cond:HOAF}, then $|\na^{h,i} Rm_h|=O(\rho^{-2-i}K(\rho/4))$ for any $i\ge 1$.
\end{enumerate}
\end{prop}

\begin{proof}
The fact that $Y$ has Euclidean volume growth follows immediately from Theorem \ref{T:met} and Theorem \ref{T:ave1}. We denote the Levi-Civita connection of $\bar g$ by $\bar \na$ and for any vector $X$ on $Y$ we define $\bar X$ to be its horizontal lift.

For any vector fields $X,Y,Z$ and $W$ on $Y$, we have by O'Neill's formula 
\begin{align}
&\la R_h(X,Y)Z,W \ra_h \notag \\
=&\la R_{\bar g}(\bar X,\bar Y)\bar Z,\bar W \ra_{\bar g}-\frac{1}{2} \la [\bar X, \bar Y]^v,[\bar Z,\bar W]^v \ra_{\bar g}-\frac{1}{4}\lc \la [\bar X, \bar Z]^v,[\bar Y,\bar W]^v \ra_{\bar g}-\la [\bar Y, \bar Z]^v,[\bar X,\bar W]^v \ra_{\bar g} \rc, \label{eq:5b026}
\end{align}
where the $v$ denotes the vertical part. Therefore, we only need to estimate the term like $[\bar X, \bar Y]^v$. Notice that from Theorem \ref{T:met} (i), locally we have invariant fields $V_i, 1\le i \le m $ which is almost an orthonormal basis. More precisely, 
\begin{align*}
\la V_i,V_j \ra_{\bar g}=\delta_{ij}+O(K(r/2)). 
\end{align*}
Then by direct computations,
\begin{align*}
\la [\bar X, \bar Y]^v,V_i \ra_{\bar g}=\la  \bar \na_{\bar X}\bar Y-\bar \na_{\bar Y}\bar X ,V_i \ra_{\bar g}=\la \bar X,\bar \na_{\bar Y}V_i  \ra_{\bar g}-\la \bar Y,\bar \na_{\bar X}V_i  \ra_{\bar g}.
\end{align*}

In the local coordinate chart, $V_i$ is given by $\partial_{n-m+i}$, therefore it is clear that
\begin{align*}
|\bar \na \bar V_i|=O(r^{-1}K(r/2))
\end{align*}
and hence
\begin{align*}
|[\bar X, \bar Y]^v|\le Cr^{-1}K(r/2)|\bar X|\,|\bar Y|,
\end{align*}
where for simplicity we use $|\cdot|$ to denote $|\cdot|_{\bar g}$. From \eqref{eq:5b026} and Theorem \ref{T:ave1} (iv) we immediately have
\begin{align*}
|Rm_h|=|Rm_{\bar g}|+O(r^{-2}K^2(r/2))=O(r^{-2}K(r/2))=O(\rho^{-2}K(\rho/4)).
\end{align*}

If $(M^n,g)$ further satisfies \eqref{cond:HOAF}, we do all computations in the local coordinate charts. In any chart $U \times \T^m_{\infty}$, it follows from Corollary \ref{C:para} that
\begin{align}\label{eq:5b029}
\bar g_{ij}-\delta_{ij}=O(K(r/2)) \quad \text{and} \quad |\partial^{\alpha} \bar g_{ij}|=O(r^{-|\alpha|}K(r/2)), \quad \forall |\alpha| \ge 1.
\end{align}

For any $1 \le i \le n-m$, the horizontal lift of $\partial_i$, denoted by $\bar \partial_i$, can be expressed as
\begin{align*}
\bar \partial_i=\partial_i+\sum_{j=n-m+1}^n a_{ij} \partial_j,
\end{align*}
where the coefficients $a_{ij}$ are determined by
\begin{align*}
\sum_{j=n-m+1}^n a_{ij} \bar g_{jk}=-\bar g_{ik}
\end{align*}
for any $n-m+1 \le k\le n$. Since $h_{ij}=\bar g(\bar \partial_i,\bar \partial_j)$, it is clear from \eqref{eq:5b029} that
\begin{align*}
h_{ij}=\delta_{ij}+O(K(r/2)) \quad \text{and} \quad |\partial^{\alpha} h_{ij}|=O(r^{-|\alpha|}K(r/2)), \quad \forall |\alpha| \ge 1.
\end{align*}

From this, we immediately have $|\na^{h,i} Rm_h|=O(\rho^{-2-i}K(\rho/4))$ for any $i \ge 1$.
\end{proof}

From Proposition \ref{P:ALE}, $Y$ is an ALE end such that there exists an ALE coordinate system at infinity, see Appendix \ref{app:A} for details. Now the following corollary is immediate.

\begin{cor} \label{cor:cone}
Let $(M^{n},g)$ be a TALE manifold. Then the tangent cone at infinity is isometric to a flat cone $C(S^{n-m-1}/\Gamma)$.
\end{cor}

Here, by slight abuse of notation, $\Gamma \subset O(n-m)$ is a finite group acting freely on $S^{n-m-1}$ if $n-m \ge 3$, $S^{n-m-1}/\Gamma=S^1$ if $n-m=2$ and $S^{n-m-1}/\Gamma$ is a single point if $n-m=1$.

\subsection{TALE manifolds with circle fibration}

In this section, we prove the following theorem.

\begin{thm}\label{T:circle}
Let $(M^n,g)$ be a complete Riemannian manifold with \eqref{cond:AF}, then the following conditions are equivalent:
\begin{enumerate} [label=(\roman*)]
\item $(M^n,g)$ is a TALE manifold with circle fibration.

\item The tangent cone at infinity is a flat cone $C(S^{n-2}/\Gamma)$, where $\Gamma \subset \text{O}(n-1)$ is a finite subgroup acting freely on $S^{n-2}$.

\item The section $S(\infty)$ of $C(S(\infty))$ is a smooth $n-2$ dimensional Riemannian manifold.
\end{enumerate}
\end{thm}

\begin{proof}
The fact that (i) implies (ii) follows immediately from Corollary \ref{cor:cone} and (ii) implies (iii) is obvious. Therefore, we only need to prove that (iii) implies (i).

It follows from Theorem \ref{T:infra} that there exists a compact set $K$ and a fibration $f:M\backslash K \longrightarrow C(S(\infty))/\bar B(p_{\infty},R_0)$. Since $C(S(\infty))$ is $n-1$ dimensional, the fiber of $f$ is a circle. For any $x\in M\backslash K$, there exists a loop $\gamma^x$ which represents the fiber of $f$ through $x$ and we denote its homotopic geodesic loop by $c^x$. In addition, there exists a small constant $\kappa >0$ such that minimal saturated set $\Omega(B(x,\kappa r))$ of $B(x,\kappa r)$ is homotopic to $S^1$.

We fix a geodesic ray $\alpha(t)$ starting from the point $p$ and consider the sliding $c(t)$ of $c^{\alpha(t_0)}$ along the $\alpha(t)$. Since the sliding preserves the generator of the local fundamental group, by switching the orientation of $c^x$ if necessary we conclude that $c(t)=c^{\alpha(t)}$.

Now we claim that $\|\mathbf r(c(t))\| \to 0$ as $t \to \infty$. Indeed, it follows from Lemma \ref{L:group2} and Lemma \ref{L:es1} that after rescaling, the group $\Gamma(\alpha(t),2\kappa t)$ as $t \to \infty$ converges to a local group $\Gamma_{\infty} \subset \text{Iso}(\R^n)$. Moreover, $B(0,\kappa)/\Gamma_{\infty}$ is isometric to an open set in $C(S(\infty))\backslash \{p_{\infty}\}$. In particular, if we set $\mathbf r(c(t)) \to A \in O(n)$, then $\tau_{\alpha(t)} \to A^{-1}$. Since $B(0,\kappa)/\Gamma_{\infty}$ is smooth with dimension $n-1$, it implies that $A^{-1}$ has a fixed $n-1$ dimensional subspace. Moreover, since locally  $\Omega(B(\alpha(t),\kappa t))$ is orientable, we can assume that $\mathbf r(c(t)) \in \text{SO}(n)$ and hence its limit $A$ must be the identity.

Therefore, it follows from the same argument of Theorem \ref{T:basis2} that for some constant $L>0$,
\begin{align} \label{eq:circ1}
\|\mathbf r(c(t))\| \le Ct^{-1}K(t/2) \quad \text{and} \quad |L(c(t))-L| \le \tilde K(t/2).
\end{align}

From \eqref{eq:circ1} and \cite[Proposisition $2.3.1$ (ii)]{BK81} that for any $c^m(t) \in \Gamma(\alpha(t),\kappa t)$,
\begin{align*} 
|\mathbf t(c^m(t))-m\mathbf t(c(t))| \le Cm^2t^{-1}K(t/2)(mt^{-1}+1) \le CmK(t/2).
\end{align*}

By the same proof of Lemma \ref{L:tr} (i) we have
\begin{align} \label{eq:circ3}
\|\mathbf r(c^m(t))\| \le Cmt^{-1}K(t/2) \le CK(t/2).
\end{align}

Next we claim that $\Gamma(\alpha(t),\kappa t)$ is generated by $c(t)$. Indeed, if we denote the shortest element of $\Gamma(\alpha(t),\kappa t)$ by $c'$, then $c'$ must generate $\Gamma(\alpha(t),\kappa t)$ since otherwise the dimension of $C(S(\infty))$ is smaller than $n-1$. If $c'\ne c(t), c^{-1}(t)$, then $c(t)=(c')^m$ for some integer $m \ne 1,-1$, which contradicts our choice of $c(t)$. Therefore, we have proved the claim.

From \eqref{eq:circ3}, \eqref{cond:SHC} holds on $\alpha$ and the conclusion follows from our main theorem and Remark \ref{rem:weak}.
\end{proof}

\begin{rem}
Minerbe has proved in \emph{\cite[Theorem $3.26$]{Mi10}} similar results under \eqref{cond:HC'} and a uniform volume assumption which guarantees the fundamental pseudo-group is generated by a single element, see \emph{\cite[Proposition $3.1$]{Mi10}}. Notice that by Proposition \ref{T:ave1}, the higher-order estimates of the averaged metric $\bar g$ are stronger than those obtained in \emph{\cite[Proposition $3.22$]{Mi10}}.
\end{rem}

\subsection{Topology of TALE $4$-manifolds}
\label{sub:topo}

Let $(M^n,g)$ be a TALE manifold, it follows immediately from Proposition \ref{P:ALE} that the end of $M$ is diffeomorphic to $X \times (0,+\infty)$ for some closed manifold $X$ which is the total space of a $\T^m$-fiber bundle over $S^{l-1}/\Gamma$ for an integer $0 \le m <n$, where we set $l=n-m$.

\begin{defn}
The asymptotic torus fibration of a TALE manifold $(M^n,g)$ is defined as the fiber bundle
\begin{align*}
X \overset{f}{\longrightarrow} S^{l-1}/\Gamma.
\end{align*}
with fiber $\T^m$ such that the structure group of $f$ is contained in $\mathbb T^m \rtimes G$ for some finite subgroup $G \subset \emph{GL}(m,\Z)$.
\end{defn}

If the dimension $n=4$, the asymptotic torus fibration and the topology of the end can be explicitly described. We only consider the case when $(M^4,g)$ is orientable, since its double cover is still a TALE manifold for which each end can be analyzed similarly. In particular, it implies that the total space $X$ of the asymptotic torus fibration is orientable. We classify $X$ and its corresponding fiber bundle as follows.

\noindent ($l=1$): In this case, the base is a point and $X$ is diffeomorphic to $\T_{\infty}^3$. \\

\noindent ($l=2$): In this case, the base is a circle and $X$ is diffeomorphic to a mapping torus of $\mathbb T_{\infty}^2$ defined by
\begin{align*} 
M_L= \frac{(I\times \mathbb T_{\infty}^2)}{(0,x)\sim (1,L(x))}
\end{align*} 
where $L$ is a finite-order element of $\text{GL}(2,\mathbb Z)$. It is well-known that the diffeomorphism type of $M_L$ is determined by the conjugacy class of $L$. Since $M_L$ is orientable, it follows from \cite[Proposition $1$]{Ta71} that there are $5$ finite-order elements in $\text{SL}(2,\mathbb Z)$ up to conjugacy in $\text{GL}(2,\mathbb Z)$. We list them as follows.
\begin{align*} 
L_1= \begin{pmatrix} 
1& 0 \\
   0 & 1
    \end{pmatrix}, \quad
L_2= \begin{pmatrix} 
-1& 0 \\
   0 & -1
    \end{pmatrix}, \quad
 L_3= \begin{pmatrix} 
0& -1 \\
   1 & -1
    \end{pmatrix}, \quad
L_4= \begin{pmatrix} 
0& -1 \\
   1 & 0
    \end{pmatrix}, \quad
L_5= \begin{pmatrix} 
1& -1 \\
   1 & 0
    \end{pmatrix}. 
\end{align*} 
Notice that $L_1,L_2,L_3,L_4,L_5$ are generators of the isometry groups $1,\Z_2,\Z_3,\Z_4,\Z_6$ of $\T_{\infty}^2$, respectively. Also, they generate the monodromy groups of the corresponding torus fibrations.\\

\noindent ($l=3$): In this case, the base of the asymptotic circle fibration is either $S^2$ (cyclic) or $\RP^2$ (dihedral). Since the total space $X$ is orientable, the fiber bundle is determined by its Euler class $e$ (see for example \cite{Mon87}) which we describe explicitly as follows.

(ALF-$A_k$): If the base is $S^2$, then $S^2 =B^{+}\cup B^{-}$ where
    \begin{align*}
B^+=\{(x_1,x_2,x_3) \in S^2 \,\mid x_3 \ge 0\} \quad \text{and} \quad B^-=\{(x_1,x_2,x_3) \in S^2 \,\mid x_3 \le 0\}.
    \end{align*}
On $B^+$ and $B^-$ the circle bundles are trivial. Then the integer $e$ is determined by identifying $(\cos t, \sin t,0, \theta)$ with $(\cos t,\sin t, 0,\theta+f(t))$, where $f(2\pi)-f(0)=-2\pi e$. When $e=0$, the fiber bundle is trivial. When $e=-1$, it is the well-known Hopf fibration $S^1  \rightarrow S^3 \rightarrow S^2$. In general, $X$ is diffeomorphic to $S^3/\Z_{|e|}$, where the action of the generator $\tau$ on $S^3=\{(z,w) \in \C^2 \,\mid |z|^2+|w|^2=1\}$ is defined by
\begin{align*}
\tau(z,w)=(e^{\frac{2\pi i}{|e|}}z,e^{\frac{2\pi i}{|e|}}w).
\end{align*}
For the cyclic case, we set $k=-e-1$ and call the TALE manifold the ALF-$A_k$ type.

(ALF-$D_k$): If the base is $\RP^2$, then $\RP^2 =\{(x_1,x_2,x_3) \in S^2 \,\mid x_3 \ge 0\}/ (x_1,x_2,0) \sim (-x_1,-x_2,0)$. The fiber bundle is trivial over the disc and we identify
$(\cos t, \sin t,0, \theta)$ with $(-\cos t,-\sin t, 0,f(t)-\theta)$, where  $f(\pi)-f(0)=-2\pi e$. When $e=0$, $X$ is diffeomorphic to $S^2 \times S^1/\pm$. When $e \ne 0$, $X$ is diffeomorphic to $S^3/D_{4|e|}$, where $D_{4|e|}$ is the binary dihedral group generated by $\tau$ and $\sigma$ defined by
\begin{align*}
\tau(z,w)=(e^{\frac{\pi i}{|e|}}z,e^{\frac{\pi i}{|e|}}w), \quad \sigma(z,w)=(\bar w,-\bar z).
\end{align*}
For the dihedral case, we set $k=-e+2$ and call the TALE manifold the ALF-$D_k$ type.

For later applications, we prove the following lemma about the topology of the end. Recall that two closed manifolds $N_1$ and $N_2$ are (topologically) $h$-cobordant if there exists a compact manifold $W$ such that $\partial W=N_1\coprod N_2$ and the inclusions $N_1 \to W$ and $N_2 \to W$ are homotopy equivalences.

\begin{lem} \label{L:topoend}
Let $M$ be an $n$-dimensional open manifold and suppose there are two compact sets $K_1$ and $K_2$ of $M$ such that for $i=1,2$, there exists a homeomorphism $\Phi_i: M\backslash \text{int}(K_i) \to N_i \times [0,\infty)$ and $\Phi_i(\partial K_i)=N_i \times \{0\}$, where $N_i$ is a closed manifold. Then $N_1$ and $N_2$ are $h$-cobordant.
\end{lem} 

\begin{proof}
We choose a large $a$ so that $\Phi_2 \circ \Phi_1^{-1}(N_1 \times \{a\})$ is well defined and the image is disjoint from $N_2 \times \{0\}$. With $a$ fixed, we choose a large $b$ so that $\Phi_1 \circ \Phi_2^{-1}(N_2 \times \{b\})$ is well defined and the image is disjoint from $N_1 \times \{a\}$. Next we define $\tilde N_1=\Phi_1^{-1}(N_1 \times \{a\})$, which is a boundary of a compact manifold $\tilde K_1$. Similarly, we set $\tilde N_2=\Phi_2^{-1}(N_2 \times \{b\})$, which is a boundary of a compact manifold $\tilde K_2$.

Now we consider the compact manifold $W=\tilde K_2 \backslash \text{int} (\tilde K_1)$. It is clear by the definition that $\partial W=\tilde N_1 \coprod \tilde N_2$. Moreover, under the homeomorphism $\Phi_1$, it is easy to construct a strong deformation retract $F((x,s),t)=(x,(1-t)s+ta)$ for any $(x,s) \in \Phi_1(W)$ from $W$ to $\tilde N_1$. Similarly, there exists a strong deformation retract from $W$ to $\tilde N_2$. Therefore, $N_1$ and $N_2$ are $h$-cobordant.
\end{proof}

\begin{rem}\label{R:topoend}
It is not necessary that $N_1$ and $N_2$ are homeomorphic, see \emph{\cite[Theorem $3$]{Mil61}}. However, for any two oriented closed $3$-manifolds $N_1$ and $N_2$, if they are $h$-cobordant, they must be diffeomorphic, see \emph{\cite[Theorem $1.4$]{Tu88}}. Notice that this result holds for any two closed $3$-manifolds due to Perelman's resolution of Thurston's geometrization conjecture \emph{\cite{Pe1,Pe2,Pe3}}.
\end{rem}

\section{Ricci-flat TALE manifolds}

\subsection{Improvement of the decay order}
\label{S:RF1}
Let $(M^n,g)$ be a Ricci-flat TALE manifold with $l$-dimensional tangent cone at infinity. It follows from Shi's local estimates \cite{Shi89}, see also \cite[Proposition $A.2$]{Mi10}, that
 \begin{align*}
|\na^k Rm| \le C(k,n)\frac{K(r/2)}{r^{2+k}},\quad \forall k\ge 1.
\end{align*}
Therefore, $(M^n,g)$ satisfies \eqref{cond:HOAF} by redefining $K$.  

Now we have the following theorem for the curvature decay.

\begin{thm}\label{T:met2}
If $l\ge 4$ or $l=3$ and $n=4$, then
\begin{align*} 
|Rm|=O(r^{-\frac{(l-2)(n-1)}{n-3}}).
\end{align*} 
\end{thm}

\begin{proof}
It is clear from Theorem \ref{T:met} that there exists a constant $C>0$ such that for any $t \ge s >0$,
\begin{align*} 
\frac{|B(p,t)|}{|B(p,s)|} \ge C \lc \frac{t}{s} \rc^l.
\end{align*} 
In addition,
\begin{align*} 
\int |Rm|^{\frac{n}{2}} r^{n-l} \,dV<\infty
\end{align*} 
since
\begin{align*} 
\int_1^{\infty} \frac{(K(t))^{\frac{n}{2}}}{t} \,dt \le C\int_1^{\infty} \frac{K(t)}{t} \,dt <\infty.
\end{align*}  
Notice that our assumption for $l$ is equivalent to the inequality $l>4\frac{n-2}{n-1}$. Therefore, the conclusion follows from \cite[Theorem $4.12$]{Mi092}.
\end{proof}

\begin{rem}
It is unclear whether any Ricci-flat TALE manifold has faster-than-quadratic curvature decay if $l=2,3$.
\end{rem}

To deal with the case $l=1$ (i.e. $C(S(\infty))=\R_+$), we consider the fibration $f: M^n\backslash K \longrightarrow (A,\infty)$ obtained in Theorem \ref{T:met}. From the fibration and Corollary \ref{C:para}, we have a coordinate chart $(A,\infty) \times \T_{\infty}^{n-1}$ such that
\begin{align} \label{eq:RF01a}
g_{ij}-\delta_{ij}=O(K(r/2)) \quad \text{and} \quad |\partial^{\alpha} g_{ij}|=O(r^{-|\alpha|}K(r/2)) , \quad \forall |\alpha|\ge 1.
\end{align}

We regard the flat metric $\tilde g=\delta_{ij}$ as the background metric and use $\tilde \na, \tilde \Delta$, etc. to denote its covariant derivative, Laplacian, etc. For any $x \in M \backslash K$,  we set $T_x=f^{-1}(f(x))$ and $dV_{T}$ to be the volume form induced by $\tilde g$ on $T_x$. Moreover, we denote the volume of $\T_{\infty}^{n-1}$ by $\mathfrak m_{\infty}$.

Following \cite{Mi09}, for any function $u$ on $M$ compactly supported in $M \backslash K$, we define
\begin{align} \label{eq:RF05b}
(\Pi_0 u)(x) \coloneqq \frac{1}{ \mathfrak m_{\infty} } \int_{T_x} u \,dV_{T} \quad \text{and} \quad \Pi_{\perp} u \coloneqq u-\Pi_0u.
\end{align}

\begin{lem} \label{lem:LRF01}
$\tilde  \Delta \circ\Pi_0=\Pi_0 \circ \tilde \Delta$ and $\tilde \Delta\circ \Pi_{\perp}=\Pi_{\perp} \circ \tilde \Delta$.
\end{lem}

\begin{proof}
Under the coordinate chart $(A,\infty) \times \T_{\infty}^{n-1}$, it follows from a direct calculation that
\begin{align*}
(\Pi_0 \circ \tilde \Delta u)(x)=&\frac{1}{\mathfrak m_{\infty}} \int_{T_x}\sum_{1 \le i \le n} \partial_i^2 u  \, dV_{T}  =  \partial_1^2  \lc \frac{1}{\mathfrak m_{\infty}} \int_{T_x} u   \, dV_{T} \rc= (\tilde  \Delta \circ \Pi_0 u)(x).
\end{align*}
Moreover, $\tilde \Delta\circ \Pi_{\perp}=\Pi_{\perp} \circ\bar \Delta$ follows immediately from \eqref{eq:RF05b}.
\end{proof}

Lemma \ref{lem:LRF01} indicates that to estimate the function $u$, we can estimate $\Pi_0u$ and $\Pi_{\perp}u$ separately. We denote the coordinate functon of $(A,\infty)$ by $t$. Notice that $t$ can be extended to a $\T_{\infty}^{n-1}$-invariant function on $M \backslash K$. To estimate the $\T^{n-1}_{\infty}$-invariant part, we need the following lemma.

\begin{lem}[Hardy's inequality] \label{lem:LRF02}
For any $R_0 \ge 0$ and smooth function $\phi(t)$ compactly supported on $(R_0,\infty)$ and $\delta \ne -1$,
\begin{align*}
\int_0^{\infty} \phi^2(t)(t-R_0)^{\delta}\,dt \le \frac{4}{(\delta+1)^2} \int_0^{\infty} (\phi')^2(t)(t-R_0)^{\delta+2}\,dt.
\end{align*}
\end{lem}
\begin{proof}
From the integration by parts and the H\"older inequality,
\begin{align*}
&\lc \int_0^{\infty} \phi^2(t)(t-R_0)^{\delta}\,dt \rc^2 =\frac{4}{(\delta+1)^2}\lc \int_0^{\infty} \phi(t)\phi'(t)(t-R_0)^{\delta+1}\,dt \rc^2 \\
\le & \frac{4}{(\delta+1)^2}\lc \int_0^{\infty} \phi^2(t)(t-R_0)^{\delta}\,dt \rc\lc \int_0^{\infty} (\phi')^2(t)(t-R_0)^{\delta+2}\,dt \rc.
\end{align*}
Therefore, the conclusion follows immediately.
\end{proof}

\begin{lem} \label{lem:LRF04}
There exist positive constants $A_0$ and $C$ satisfying the following property.

Suppose $u$ is a smooth function compactly supported in $t^{-1}((R_0,\infty))$ for $R_0 \ge A_0$. Then
\begin{align*}
\int u^2  \,d\tilde V \le C\int (\tilde \Delta u)^2 (t-R_0)^{4} \,d\tilde V.
\end{align*}
\end{lem}

\begin{proof}
Let $u=\Pi_0u+\Pi_{\perp} u=u_0+u_{\perp}$. Since $u_0$ depends only on $t$, $\tilde \Delta u_0=u_0''$. By applying Lemma \ref{lem:LRF02} twice, we have
\begin{align}
\int u_0^2  \,d\tilde V= \mathfrak{m}_{\infty} \int u_0^2  \,dt \le C  \mathfrak{m}_{\infty} \int (u_0'')^2 (t-R_0)^4 \,dt = C\int (\tilde \Delta u_0)^2 (t-R_0)^{4} \,d\tilde V. \label{eq:RF07d}
\end{align}

On the other hand, for any $x \in t^{-1}((R_0,\infty))$, we have
\begin{align*}
\frac{1}{\mathfrak m_{\infty}} \int_{T_x} u_{\perp}\,dV_{T}=0.
\end{align*}
From the Poincar\'e inequality, we obtain
\begin{align*}
\int_{T_x} u_{\perp}^2   \,dV_{T} \le C\int_{T} |\na_{T} u_{\perp}|^2\,dV_{T} \le C\int_{T} |\tilde \na u_{\perp}|^2\,dV_{T}.
\end{align*}
By the coarea formula, we compute
\begin{align*}
\int u^2_{\perp}  \,d\tilde V \le C\int |\tilde \na u_{\perp}|^2\,d\tilde V=C\int (\tilde \Delta u_{\perp})u_{\perp}\, d\tilde V \le C\lc \int u_{\perp}^2\, d\tilde V \rc^{\frac{1}{2}}\lc \int (\tilde \Delta u_{\perp})^2\, d\tilde V \rc^{\frac{1}{2}}.
\end{align*}
Therefore,
\begin{align}
\int u^2_{\perp}  \,d\tilde V \le C\int (\tilde \Delta u_{\perp})^2\, d\tilde V.
\label{eq:RF07g}
\end{align}

Combining \eqref{eq:RF07d} and \eqref{eq:RF07g} we have
\begin{align*}
\int u^2  \,d\tilde V \le& 2\int u_0^2+u_{\perp}^2 \,d\tilde V\le C\int \lc (\tilde \Delta u_0)^2+(\tilde \Delta u_{\perp})^2 \rc (t-R_0)^4 \,d\tilde V=C\int (\tilde \Delta u)^2 (t-R_0)^4\,d\tilde V, 
\end{align*}
where we have used the following fact from Lemma \ref{lem:LRF01},
\begin{align*}
\int (\tilde \Delta u_0)( \tilde \Delta u_{\perp}) (t-R_0)^4 \,d\tilde V=0.
\end{align*}
\end{proof}

\begin{rem}\label{rem:x61}
Lemma \ref{lem:LRF04} also holds for any smooth tensor $u$ compactly supported in $t^{-1}((R_0,\infty))$. Indeed, one only needs to consider each component of $u$ under the coordinate chart $(A,\infty) \times \T_{\infty}^{n-1}$.
\end{rem}

We can now improve the decay order if $l=1$.
\begin{thm}\label{T:ftq}
If $(M^{n},g)$ is a Ricci-flat TALE manifold such that the tangent cone at infinity is $\R_+$, then there exists a constant $\delta>0$ which depends only on $\T^{n-1}_{\infty}$ such that
\begin{align*}
|Rm|=O(e^{-\delta r}).
\end{align*}
\end{thm}
\begin{proof}
From the Ricci-flatness condition, if we set $u=Rm$, then $\Delta u=Q(u)$ for some $Q$ quadratic in $u$. Therefore, it is clear from \eqref{eq:RF01a} that
\begin{align}
|\tilde \Delta u |_{\tilde g}\le O(r^{-2}K(r))|u|_{\tilde g}+O(r^{-1}K(r/2))|\tilde \na u|_{\tilde g}+O(K(r/2))|\tilde \na^2 u|_{\tilde g}.
\label{eq:RF09b}
\end{align}
For any large $R_0 \ge A_0$ and $k$, we choose a cutoff function $\phi_k$ which is supported on $t^{-1}((R_0,R_0+k+1))$ and $\phi_k=1$ on $t^{-1}([R_0+1,R_0+k])$. Then it follows from Lemma \ref{lem:LRF04} and Remark \ref{rem:x61} that
\begin{align}\label{eq:RF10}
\int |u|_{\tilde g}^2\phi_k^2\,d\tilde V \le C\int |\tilde \Delta (u\phi_k)|_{\tilde g}^2 (t-R_0)^4 \,d\tilde V.
\end{align}
From \eqref{eq:RF09b} and \eqref{eq:RF10},
\begin{align}
\int |u|_{\tilde g}^2\phi_k^2\,d\tilde V \le& C\int \lc t^{-4}K^2(t/2)|u|_{\tilde g}^2+t^{-2}K^2(t/3))|\tilde \na u|_{\tilde g}^2+K^2(t/3))|\tilde \na^2 u|_{\tilde g}^2 \rc (t-R_0)^4 \,d\tilde V \notag \\
&+ C\int \lc |u|_{\tilde g}^2|\tilde \Delta \phi_k|+|\tilde \na u|_{\tilde g}^2|\tilde \na \phi_k|^2  \rc (t-R_0)^4 \,d\tilde V.
\label{eq:RF11}
\end{align}

Now, we claim
\begin{align}\label{eq:RF12}
\int_{t \ge R_0+1} |\tilde \na^2 u|_{\tilde g}^2 (t-R_0)^4+|\tilde \na u|_{\tilde g}^2 (t-R_0)^2 \,d\tilde V \le C \int_{t \ge R_0} |u|_{\tilde g}^2 \,d\tilde V.
\end{align}

From the elliptic equation of $u$ and \cite[Theorem $9.11$]{GT01} we have
\begin{align}\label{eq:RF12x1}
\int_{R_0+1 \le t \le 2R_0} |\tilde \na^2 u|_{\tilde g}^2 +|\tilde \na u|_{\tilde g}^2 \,d\tilde V \le C \int_{t \ge R_0} |u|_{\tilde g}^2 \,d\tilde V.
\end{align}

Moreover, if we set $t_i=2^iR_0$ for $i \ge 1$, then it follows from the scaling invariant version of \cite[Theorem $9.11$]{GT01} that
\begin{align}\label{eq:RF12x2a}
\int_{t_i \le t \le 2t_i} |\tilde \na^2 u|_{\tilde g}^2 t_i^4 +|\tilde \na u|_{\tilde g}^2 t_i^2 \,d\tilde V \le C \int_{2t_i/3 \le t \le 4t_i} |u|_{\tilde g}^2 \,d\tilde V.
\end{align}
Indeed, we can lift the coordinate chart $(A,\infty) \times \T_{\infty}^{n-1}$ to $(A,\infty) \times \R^{n-1}$ and apply \cite[Theorem $9.11$]{GT01} to the lift of $u$ on $(A,\infty) \times \R^{n-1}$. By taking the sum of \eqref{eq:RF12x2a} for $i \ge 1$, we have
\begin{align}\label{eq:RF12x2}
\int_{t \ge 2R_0} |\tilde \na^2 u|_{\tilde g}^2 (t-R_0)^4 +|\tilde \na u|_{\tilde g}^2 (t-R_0)^2 \,d\tilde V \le C \int_{t \ge R_0} |u|_{\tilde g}^2 \,d\tilde V.
\end{align}

Combining \eqref{eq:RF12x1} and \eqref{eq:RF12x2}, \eqref{eq:RF12} follows immediately. Since $K(t) \to 0$ as $t \to +\infty$, if $R_0$ is sufficiently large, we conclude from \eqref{eq:RF11} and \eqref{eq:RF12} that by taking $k \to \infty$,
\begin{align*}
\int_{t \ge R_0+1} |u|_{\tilde g}^2 \,d\tilde V \le C\int_{R_0 \le t \le R_0+1} |u|_{\tilde g}^2 \,d\tilde V
\end{align*}
and hence
\begin{align}\label{eq:RF13}
\int_{t \ge R_0+1} |u|_{\tilde g}^2 \,d\tilde V \le \frac{C}{C+1}\int_{t \ge R_0} |u|_{\tilde g}^2 \,d\tilde V.
\end{align}

From \eqref{eq:RF13}, a standard iteration argument implies that there exists some constant $\lam>0$ such that
\begin{align*}
\int_{t \ge R} |u|_{\tilde g}^2 \,d\tilde V \le Ce^{-\lam R}
\end{align*}
for any $R \ge R_0$. Therefore, it follows from the equation of $u$ and \cite[Theorem $9.20$]{GT01} that for some $\delta>0$ which depends only on $\T_{\infty}^{n-1}$,
\begin{align*}
|u|_{\tilde g}=O(e^{-\delta r}).
\end{align*}
From \eqref{eq:RF01a}, we conclude that
\begin{align*}
|Rm|_{g}=O(e^{-\delta r}).
\end{align*}
\end{proof}

\subsection{Hitchin-Thorpe inequality}
For any compact oriented Einstein $4$-manifold $(M^4,g)$, we have the following celebrated Hitchin-Thorpe inequality \cite{Tho69,Hit74}
\begin{align*}
2\chi(M) \ge 3|\tau(M)| 
\end{align*}
with equality if and only if $g$ is flat or finitely covered by a $K3$ surface. In this section, we consider an oriented Ricci-flat TALE $4$-manifold $(M^4,g)$ and prove a Hitchin-Thorpe type inequality. As before, we set $l$ to be the dimension of the tangent cone at infinity. 

(I) $l=4$.

In this case, Nakajima has proved the following theorem. Notice that for any finite subgroup $\Gamma \subset \text{SO}(4)$ acting freely on $S^3$, $\eta(S^3/\Gamma)$ is the eta invariant of standard $S^3/\Gamma$, which is calculated in \cite{APS75b}. If $\Gamma \subset \text{SU}(2)$, $\eta(S^3/\Gamma)$ is computed explicitly by Nakajima \cite{Na90}.

\begin{thm}[Theorem $4.2$ of \cite{Na90}] \label{T:HTALE}
Let $(M^4,g)$ be a Ricci-flat ALE manifold with end $S^3/\Gamma$. Then
\begin{align} \label{eq:HT1}
2( \chi(M)-\frac{1}{|\Gamma|} ) \ge 3|\tau(M)+\eta(S^3/\Gamma)| 
\end{align}
with equality if and only if $(M,g)$ or its opposite orientation space is a quotient of a \hy ALE $4$-manifold.
\end{thm}

From Theorem \ref{T:HTALE}, we have the following immediate corollary, see also \cite[Theorem $1.5$]{LV16}.

\begin{cor}\label{C:HTALE}
Let $(M^4,g)$ be a Ricci-flat ALE manifold such that $M$ is homeomorphic to $\mini$ for some finite subgroup $\Gamma \subset \emph{SU}(2)$. Then $(M,g)$ or its opposite orientation space is a \hy ALE $4$-manifold. In particular, $M$ is diffeomorphic to $\mini$.
\end{cor}

\begin{proof}
Since $(M,g)$ is an ALE manifold and $M$ is homeomorphic to $\mini$, $M$ has one end which is diffeomorphic to $S^3/\Gamma_1 \times \R_+$ for some finite group $\Gamma_1 \subset \text{SO}(4)$. Therefore, it follows immediately from Lemma \ref{L:topoend} and Remark \ref{R:topoend} that $S^3/\Gamma_1$ is diffeomorphic to $S^3/\Gamma$. From the work of De Rham \cite{DeRham51}, $S^3/\Gamma_1$ must be isometric to $S^3/\Gamma$. Then it is clear that $\Gamma_1$ is conjugate to $\Gamma$ in $\text{SO}(4)$. From Kronheimer's result \cite{Kro89a}, there exists a \hy metric on $\mini$ and hence we have equality in \eqref{eq:HT1}:
\begin{align*}
2( \chi(M)-\frac{1}{|\Gamma|} ) = 3|\tau(M)+\eta(S^3/\Gamma)|.
\end{align*}
Therefore, it follows from the equality case of Theorem \ref{T:HTALE} that $(M,g)$ or its opposite orientation space is a \hy ALE $4$-manifold. In particular, it follows from \cite[Theorem $1.2$]{Kro89b} that $M$ must be diffeomorphic to $\mini$.
\end{proof}

In particular, Corollary \ref{C:HTALE} states that for any smooth manifold $M$ homeomorphic to $\R^4$, if $M$ admits a Ricci-flat ALE metric, then $M$ must be isometric to $\R^4$ with flat metric. In particular, $M$ is diffeomorphic to $\R^4$ with standard differential structure. 

(II) $l=3$.

In this case, the tangent cone at infinity is $\R^3$ (cyclic type) or $\R^3/\Z_2$ (dihedral type). It follows from Theorem \ref{T:met2} that 
    \begin{align} \label{eq:RFa00}
|Rm|=O(r^{-3}).
    \end{align} 
Therefore, we have from Theorem \ref{T:ave1}
    \begin{align} \label{eq:RFa01}
g=\bar  g+O(r^{-2}) \quad \text{and} \quad |\na^i \bar g|=O(r^{-2-i}), \quad \forall i\ge 1.
    \end{align} 
Since $M\backslash K$ has a circle fibration over an ALE end $Y$, we define 
\begin{align*}
\rho(x)=\sqrt{x_1^2+x_2^2+x_3^2},
\end{align*}
where $(x_1,x_2,x_3)$ are coordinates constructed in Theorem \ref{T:ap01}. Notice that $\rho$ can be extended to an $S^1$-invariant function on $M\backslash K$. We define a large domain $D_t=\{x \in M \mid \rho(x) \le t\}$ and denote the second fundamental forms of $\partial D_t$ with respect to $g$ and $\bar g$ by $I(t)$ and $\bar I(t)$, respectively. Then it is clear that
    \begin{align} \label{eq:RFa03}
I(\rho) =O(\rho^{-1}) \quad \text{and} \quad \bar I(\rho)=O(\rho^{-1}).
    \end{align} 

Now we prove the following theorem, see also \cite[Corollary $1.2$]{DaiWei07}.

\begin{thm} \label{T:HTALE1}
Let $(M^4,g)$ be a Ricci-flat TALE manifold with circle fibration. Then
\begin{align} 
2\chi(M) &\ge 3\left| \tau(M)-\frac{e}{3}+\emph{sgn}\,e \right| \quad \emph{(cyclic type)} \label{eq:HT2a}\\
2\chi(M) &\ge 3\left|\tau(M)-\frac{e}{3}\right| \qquad \qquad \emph{(dihedral type)} \label{eq:HT2b}
\end{align}
with equality if and only if $(M,g)$ or its opposite orientation space is a quotient of a \hy $4$-manifold. Here $e$ is the Euler number of the asymptotic circle fibration.
\end{thm}

\begin{proof}
For large $t$, $D_t$ is diffeomorphic to $M$. So it follows from Gauss-Bonnet-Chern theorem that
   \begin{align*}
\chi(M)=\frac{1}{8\pi^2}\int_{D_t}|W_+|^2+|W_-|^2 \,dV+\int_{\partial D_t} Rm*I+I*I*I \,d\sigma
    \end{align*} 
Here we use $I$ to denote the second fundamental form of $\partial D_t$, $Rm*I$ is a bilinear form and $I*I*I$ is a trilinear form. Since the volume of $\partial D_t=O(t^2)$, it follows from \eqref{eq:RFa00} and \eqref{eq:RFa03} that by taking $t \to \infty$,
    \begin{align}\label{eq:asym1}
\chi(M)=\frac{1}{8\pi^2}\int |W_+|^2+|W_-|^2 \,dV.
    \end{align} 

From the signature theorem (see \cite{APS75a} \cite[Page $348$]{EGH80}),
    \begin{align}\label{eq:asym1a}
\tau(M)=\frac{1}{12\pi^2}\int_{D_t}|W_+|^2-|W_-|^2 \,dV+\int_{\partial D_t} Rm*I \,d\sigma-\eta(\left.g\right|_{\partial D_t}).
    \end{align} 
By taking the limit $t \to \infty$,
    \begin{align}\label{eq:asym2}
\tau(M)=\frac{1}{12\pi^2}\int |W_+|^2-|W_-|^2 \,dV- \lim_{t \to \infty} \eta(\left.g\right|_{\partial D_t}).
    \end{align} 
In particular, $\lim_{t \to \infty} \eta(\left.g\right|_{\partial D_t})$ exists. Now we claim
        \begin{align*}
\lim_{t \to \infty} \eta(\left.g\right|_{\partial D_t})=\lim_{t \to \infty}\eta(\left.\bar g\right|_{\partial D_t}).
    \end{align*}
Indeed, we extend $\bar g$ to be a complete metric on $M$ and choose a cutoff function $\phi$ on $M$ such that $\phi=1$ on $D_{9t/10}$ and $\phi=0$ outside $D_{t}$. Now we define a metric $g_t=\phi \bar g+(1-\phi)g$. Similar to \eqref{eq:asym1a} we have
\begin{align}\label{eq:asym4a}
\tau(M)=\frac{1}{12\pi^2}\int_{D_t} |\bar W_{+}|^2-|\bar W_{-}|^2 \,d\bar V+\int_{\partial D_t} \bar Rm*\bar I \,d\bar \sigma-\eta(\left.\bar g\right|_{\partial D_t})
\end{align} 
and
\begin{align}\label{eq:asym4b}
\tau( M)=\frac{1}{12\pi^2}\int_{D_t} |W_{+,t}|^2-|W_{-.t}|^2 \,dV_t+\int_{\partial D_t} Rm_t*I_t \,d\sigma_t-\eta(\left.g\right|_{\partial D_t}).
\end{align} 
Here we use bar and subscript $t$ to denote the elements of $\bar g$ and $g_t$, respectively. From \eqref{eq:RFa01}, it is easy to see that $|Rm_t|=O(\rho^{-3})$ and $I_t=O(\rho^{-1})$ with uniform decay independent of $t$. Therefore, it follows from \eqref{eq:asym4a} and \eqref{eq:asym4b} that
\begin{align*}
\eta(\left.\bar g\right|_{\partial D_t})=\eta(\left.g\right|_{\partial D_t})+O(t^{-2})
\end{align*} 
and hence the claim is proved by taking $t \to \infty$.

Notice that $\eta_{\text{ad}} \coloneqq \lim_{t \to \infty}\eta(\left.\bar g\right|_{\partial D_t})$ is the adiabatic limit of the asymptotic circle fibration. Combining \eqref{eq:asym1} and \eqref{eq:asym2}, we immediately have
\begin{align}\label{eq:asym4d}
2\chi(M)+3(\tau(M)+\eta_{\text{ad}})=\frac{1}{4\pi^2}\int |W_+|^2\,dV \ge 0.
\end{align}   

$\eta_{\text{ad}}$ can be calculated by considering the standard models of ALF gravitational instantons, for which $W_+\equiv 0$. Indeed, for the standard ALF-$A_k$ gravitational instanton, $\chi=k+1,\tau=-k$ if $k \ge 0$ and $\chi=\tau=0$ if $k=0$. Since $k=-e-1$, it is clear from \eqref{eq:asym4d} that for any $e \le 0$,
\begin{align}\label{eq:asym4e}
\eta_{\text{ad}}=-\frac{e}{3}+\text{sgn}\,e.
\end{align}
Notice that \eqref{eq:asym4e} also holds for all $e \ge 0$ since the sign of $\eta$-invariant is changed if we change the orientation.

Similarly, for the ALF-$D_k$ type, if $k \ge 2$, then $\chi=k+1$, $\tau=-k$. Since $k=-e+2$,
\begin{align}\label{eq:asym4f}
\eta_{\text{ad}}=\frac{k-2}{3}=-\frac{e}{3}.
\end{align}
By the same reason, \eqref{eq:asym4f} can be extended for all $e \ge 0$.

Therefore, inequalities \eqref{eq:HT2a} and \eqref{eq:HT2b} hold with equality if and only if $W_+$ or $W_-$ vanishes. In either case, it implies that the universal covering is a \hy $4$-manifold.
\end{proof}

\begin{rem}
Biquard and Minerbe \emph{\cite[Section $4$]{BiMi11}} proved the same Hitchin-Thorpe inequality for ALF gravitational instantons by using which they showed that ALF-$D_k$ type gravitational instanton exits only if $k \ge 0$.
\end{rem}

Let $\mcm$ and $\mdm$ be the minimal resolution of $\C^2/\Z_m$ and $\C^2/D_{4m}$ respectively, where $\Z_m$ is the cyclic group of order $m$ and $D_{4m}$ is the binary dihedral group of order $4m$.

\begin{cor}\label{C:HTALE1}
Let $(M^4,g)$ be a Ricci-flat TALE manifold with circle fibration such that $M$ is homeomorphic to $\mcm$ for $m \ne 2$. Then $(M,g)$ is isometric to a Multi-Taub-NUT metric. In particular, $M$ is diffeomorphic to $\mcm$.
\end{cor}

\begin{proof}
It follows from Lemma \ref{L:topoend} and Remark \ref{R:topoend} that the boundary of $M$ is diffeomorphic to $S^3/\Z_m$. Therefore, the asymptotic circle fibration is of cyclic type and its Euler number $e=\pm m$. On the other hand, it is clear that $\chi(M)=\chi(\mcm)=m$ and $\tau(m)=\tau(\mcm)=1-m$. Then it follows from \eqref{eq:HT2a} that
\begin{align} \label{eq:HTALF1}
2m \ge 3\left|1-m-\frac{e}{3}+\text{sgn}\,e\right| 
\end{align}
If $e=-m$, then the equality of \eqref{eq:HTALF1} holds and $(M,g)$ is \hy. Therefore, it follows from \cite{Mi11} that $(M,g)$ is isometric to a Multi-Taub-NUT metric. If $e=m$, then it is easy to solve \eqref{eq:HTALF1} that $m=1,2$ or $3$. If $m=1$ or $3$, then the equality of \eqref{eq:HTALF1} holds and by the same reason, $(M,g)$ with opposite orientation is isometric a Multi-Taub-NUT metric.
\end{proof}

In particular, Corollary \ref{C:HTALE1} states that for any smooth manifold $M$ homeomorphic to $\R^4$, if $M$ admits a Ricci-flat ALF metric, then $M$ must be isometric to the Taub-NUT metric.

Similarly, since $\chi(\mdm)=3+m$ and $\tau(\mdm)=-2-m$, we have the following result by using \eqref{eq:HT2b}.
\begin{cor}\label{C:HTALE2a}
Let $(M^4,g)$ be a Ricci-flat TALE manifold with circle fibration such that $M$ is homeomorphic to $\mdm$. Then $(M,g)$ is isometric to a Cherkis-Hitchin-Ivanov-Kapustin-Lindstr\"om-Ro\v cek metric. In particular, $M$ is diffeomorphic to $\mdm$.
\end{cor}

(III) $l=2$

In this case, the asymptotic torus fibration is
\begin{align*}
X \longrightarrow S^1.
\end{align*}
with fiber $\T_{\infty}^2$ such that the monodromy group is given by $1,\Z_2,\Z_3,\Z_4$ or $\Z_6$. Notice that the total space $X$ is a closed flat manifold whose eta invariant is given by $0,0,-\frac{2}{3},-1$ or $-\frac{4}{3}$, see, e.g., \cite{SZ12}.

By the same argument of Theorem \ref{T:HTALE1}, we have
\begin{thm} \label{T:HTALE2}
Let $(M^4,g)$ be a Ricci-flat TALE manifold with $\T^2$-fibration. Then
\begin{align*} 
2\chi(M) \ge 3|\tau(M)+\eta|
\end{align*}
with equality if and only if $(M,g)$ or its opposite orientation space is a quotient of a \hy $4$-manifold. Here $\eta=0,0,-\frac{2}{3},-1$ or $-\frac{4}{3}$ if the monodromy of the asymptotic $\T^2$-fibration is $1,\Z_2,\Z_3,\Z_4$ or $\Z_6$, respectively.
\end{thm}

Now we immediately have
\begin{cor} \label{C:HTALE2}
Let $(M^4,g)$ be a Ricci-flat TALE manifold with $\T^2$-fibration such that $M$ is homeomorphic to an ALG gravitational instanton. Then $(M,g)$ is an ALG gravitational instanton.
\end{cor}

(IV) $l=1$

In this case, $X$ is a flat torus $\T_{\infty}^3$ whose eta invariant is $0$. Similarly, we have

\begin{thm} \label{T:HTALE3}
Let $(M^4,g)$ be a Ricci-flat TALE manifold with $\T^3$-fibration. Then
\begin{align*} 
2\chi(M) \ge 3|\tau(M)|
\end{align*}
with equality if and only if $(M,g)$ or its opposite orientation space is a quotient of a \hy $4$-manifold. 
\end{thm}

It follows from \cite[Theorem $3.4$]{CC16} that any ALH gravitational instantons is diffeomorphic to $\widetilde{\R \times \T^3/\pm}$, the minimal resolution of $\R \times \T^3/\pm$. Then by Theorem \ref{T:HTALE3} we immediately have

\begin{cor} \label{C:HTALE3}
Let $(M^4,g)$ be a Ricci-flat TALE manifold with $\T^3$-fibration such that $M$ is homeomorphic to $\widetilde{\R \times \T^3/\pm}$. Then $(M,g)$ is an ALH gravitational instanton.
\end{cor}

As an application of the Hitchin-Thorpe inequality, we prove the following theorem.

\begin{thm}\label{main1}
If $(M^4,g)$ is a complete Ricci-flat Riemannian manifold with \eqref{cond:AF} such that $M^4$ is homeomorphic to $\R^4$ and the tangent cone at infinity is not $\R \times \R_+$, then $g$ is isometric to either the flat or the Taub-NUT metric. In particular, $M^4$ is diffeomorphic to $\R^4$. 
\end{thm}
\begin{proof}
Since $M$ is simply-connected at infinity, it follows from \cite[Theorem A (ii)]{PT01} and our assumption that $C(S(\infty))$ is isometric to $\R^4$ or $\R^3$. If $C(S(\infty))=\R^4$, then $(M,g)$ has Euclidean volume growth and must be an ALE manifold. Therefore, it follows from Corollary \ref{C:HTALE} that $(M,g)$ is isometric to $(\R^4,g_E)$.

If $C(S(\infty))=\R^3$, then it follows from Theorem \ref{T:circle} that $(M,g)$ is a TALE manifold with circle fibration. It is clear that the Euler number of the asymptotic circle fibration over $S^2$ must be $\pm 1$. Thus, it follows from Corollary \ref{C:HTALE1} that $(M,g)$ is isometric to the Taub-NUT metric.
\end{proof}

\begin{rem}
Notice that the assumption that $M$ is homeomorphic to $\R^4$ can be weaken to $M$ is simply-connected at infinity, $\chi(M)=1$ and $\tau(M)=0$.
\end{rem}

\section{Further discussion}
In this section, we propose some questions.

\begin{quest}\label{Q:1}
What is the optimal order of curvature decay for Ricci-flat TALE manifolds? 
\end{quest}

Notice that in the $4$-dimensional case, if the tangent cone at infinity is either $3$ or $4$ dimensional, the result in Theorem \ref{T:003} is optimal.

\begin{quest}\label{Q:2}
Is there any example of simply-connected Ricci-flat TALE $4$-manifolds of type ALG or ALH which is non-\hy?
\end{quest}

A celebrated conjecture in \cite{BKN89} states that there is no simply-connected non-\hy Ricci-flat ALE $4$-manifold. For the ALF case, we do have non-\hy examples like Euclidean Schwarzschild metric and Euclidean Kerr–Newman metric, see \cite[Page $384$]{EGH80}.

\appendixpage
\addappheadtotoc
\appendix
  
\section{ALE coordinates at infinity}  
\label{app:A}

Let $(M^n,g,p)\,,n\ge 3$ be a complete Riemannian manifold with \eqref{cond:AF} and Euclidean volume growth. We define a function 
    \begin{align*}
\bar K(t) \coloneqq \max\left\{t^{-1}, t^{-1} \int_{1/2}^t K_0(s)\,ds,K(t/2),K_0(t/2)\right \}
    \end{align*} 
for $t \ge 1$ and $\bar K(t)=\bar K(1)$ for $0 \le t \le 1$. Here $K_0$ is defined as
\begin{align*}
K_0(t) \coloneqq \int_t^\infty \frac{K(s)}{s}\,ds.
\end{align*}

In this appendix, we prove the following theorem by using the same arguments of \cite{BKN89}. For simplicity, we assume $M^n$ has only one end.

\begin{thm}\label{T:ap01}
For any $\alpha \in (0,1)$, there exist a compact subset $K \subset M$, two constants $A>0,C>1$, a finite subgroup $\Gamma \subset O(n)$ acting freely on $\R^n \backslash B(0,A)$ and a $C^{\infty}$-diffeomorphism $\Phi: M \backslash K \to (\mathbb{R}^n \backslash B(0,A))/\Gamma$ such that under this identification, for any $z,w \in \R^n \backslash B(0,A)$,
\begin{enumerate}[label=(\roman*)]
\item $|g_{ij}(z) -\delta_{ij}| \le C\bar K(C^{-1}|z|)$;

\item $|\partial_k g_{ij}(z)| \le C |z|^{-1}\bar K(C^{-1}|z|)$;

\item $\displaystyle \frac{|\partial_k g_{ij}(z)-\partial_k g_{ij}(w)|}{|z-w|^{\alpha}} \le C \max \left\{|z|^{-1-\alpha}\bar K(C^{-1}|z|),|w|^{-1-\alpha}\bar K(C^{-1}|w|) \right\}$;

\item If $(M^n,g)$ further satisfies \eqref{cond:HOAF}, then $|\partial^{|m|}g_{ij}(z)| \le C_m|z|^{-m} \bar K(C^{-1}|z|)$ for any partial derivatives of order $m \ge 2$.
\end{enumerate}
\end{thm}

\begin{proof}
The theorem follows essentially from Step 1-4 of the main theorem of \cite{BKN89} and we sketch the proof for completeness. For simplicity, we use the notation $g-g_E \in C^{1,\alpha}_{\bar K}$ if (i),(ii) and (iii) above hold. Moreover, for any $k \ge 2$, $f \in C^{k,\alpha}_{r^{k-1}\bar K}$ if there exists a constant $C>1$ such that
\begin{align*}
& \sup_{z \in \R^n \backslash B(0,C)} \sigma^{-1}(z) \lc \sum_{j=0}^k |z|^j |D^j f(z)|   \rc \\
+&\sup_{\substack{z,w \in \R^n \backslash B(0,C) \\ z \ne w}} \lc \min \left\{ |z|^{k+\alpha} \sigma^{-1}(z),|w|^{k+\alpha} \sigma^{-1}(w) \right\} \frac{|D^kf(z)-D^k f(w)|}{|z-w|^{\alpha}} \rc <\infty,
\end{align*} 
where $\sigma(z):=|z|^{k-1} \bar K(C^{-1}z)$.

Since $(M^n,g)$ satisfies \eqref{cond:AF}, it follows from \cite{Ka89} that the tangent cone at infinity is a flat cone $C(S^{n-1}/ \Gamma)$ for some finite subgroup $\Gamma \subset O(n)$ acting freely on $S^{n-1}$. Moreover, for any large $R$, there exists a smooth hypersurface $S_R$ of $M$ such that
\begin{align*}
R^{-1}d(S_R,\partial B(p,R)) &\le \ep(R), \\
|R h_R-I| &\le \ep(R),
\end{align*} 
where $h_R$ is the second fundamental form of $S_R$ and $\epsilon(R) \to 0$ if $R \to \infty$.

We claim that if $R$ is sufficiently large, the map $\Psi_R:(x,t) \in S_R \times [R ,\infty) \to \ex((t-R)v_R(x))$ is a diffeomorphism, where $v_R$ is the unit outer normal vector of $S_R$. Indeed, we consider the following ordinary differential equation:
\begin{align*}
&J''(t)+t^{-2}K(t/2)J(t)=0,\quad \text{for} \quad t \in [R,\infty),\\
&J(R)=1,\quad J'(R)=R^{-1}(1-\ep(R)).
\end{align*}
Then it is clear that $J'(t) \le R^{-1}(1-\ep(R))$ and
\begin{align*}
J'(t)=&R^{-1}(1-\ep(R))-\int_R^t s^{-2}K(s/2)J(s)\,ds \\
\ge& R^{-1}(1-\ep(R))-\int_R^t s^{-2}K(s/2) \lc1+\frac{s}{R}(1-\ep(R)) \rc\,ds \\
\ge& R^{-1}(1-\ep(R))-2R^{-1}\int_R^t s^{-1}K(s/2) \,ds >0
\end{align*}
if $R$ is sufficiently large and hence $J(t)>0$. Then it follows from \cite{War66} that $S_R$ has no focal point along any geodesic $\sigma(t)=\ex((t-R)v_R(x))$ as long as $d(\sigma(t),S_R)=t-R$. In addition, it follows from \cite[2.2]{Ka89} that $S_R$ has no focal point. Therefore, the map $\Psi$ is a diffeomorphism for large $R$.

By rescaling, we assume that $R=1$ and set $S=S_1, h=h_1$ and $\Psi=\Psi_1$. Moreover, we write
\begin{align*}
\Psi^* g=dt^2+t^2g_t
\end{align*}
on $S \times [1,\infty)$. Fix a point $x \in S$ and $X \in T_xS$, we define for $t \ge 1$,
\begin{align*}
X(t) =d\Psi_{(x,t)} X.
\end{align*}
It is clear that $X(t)$ is the Jacobi field along $\sigma(t)=\Psi(x,t)$ such that $X(1)=X$ and $X'(1)=h(X)$. If we take a parallel orthonormal frame $\{E_1(t),\cdots, E_n(t)\}$ along $\sigma$, then it follows from the Jacobi equation that
 \begin{align}\label{ap:02}
X^i(t)=X^i(1)+(X^i)'(1)(t-1)-\int_1^t \int_1^s X^j(u)K_j^i(u)\,duds
 \end{align} 
where $X^i(t)=\la X(t),E_i(t)\ra$ and $K_j^i(t)=\la R(E_j(t),\sigma'(t))\sigma'(t),E_i(t)\ra$.

From \eqref{ap:02}, we have
 \begin{align}
&t^{-1}X^i(t)-(X^i)'(1)+\int_1^{\infty}X^j(u)K_j^i(u)\,du \notag \\
=&t^{-1}\lc X^i(1)-(X^i)'(1)+\int_1^{\infty}X^j(u)K_j^i(u)\,du \rc+t^{-1}\int_1^t \int_s^{\infty} X^j(u)K_j^i(u)\,duds\label{ap:03}.
 \end{align} 

From \cite{War66}, $|X(t)| \le Ct$ and hence 
 \begin{align*}
\left |t^{-1}\int_1^t \int_s^{\infty} X^j(u)K_j^i(u)\,duds \right| \le Ct^{-1}\int_1^t \int_s^{\infty}u^{-1}K(u/2)\,duds=Ct^{-1}\int_1^t K_0(s/2)ds.
 \end{align*} 

Therefore, it follows from \eqref{ap:03} that
 \begin{align}\label{ap:03a}
\left |t^{-1}X^i(t)-(X^i)'(1)+\int_1^{\infty}X^j(u)K_j^i(u)\,du \right| \le C\bar K(t).
 \end{align} 

Since $(S,g_t)$ converges smoothly to the space form $(S^{n-1}/\Gamma,g_{std})$, then by adding a diffeomorphism of $S$ we obtain that
 \begin{align}\label{ap:03x}
|g_t-g_{std}|=O(\bar K(t)).
 \end{align} 

Now we consider the following two cases.

\textbf{Case 1}: $\Gamma=1$.

We set $\Psi_t(x,r)=\Psi(x,tr)$ for $x\in S$ and $r \in (1/2,3/2)$ and $\sigma(t)=\Psi(x,t)$. It follows from Lemma \ref{L:es2} that there exists a small number $\delta>0$ and a harmonic coordinate system $\mathbb H_t$ on $B(\sigma(t),2\delta t)$ such that
 \begin{align*}
|d\mathbb H_t-\omega_t| \le CK(t/2) \quad \text{on} \quad B(\sigma(t),2\delta t),
 \end{align*} 
where $\omega_t$ is the dual frame of $\{E_1(t),\cdots, E_n(t)\}$. Therefore,
 \begin{align}\label{ap:04}
|d(t^{-1}\mathbb H_t\circ \Psi_t)-t^{-1}\Psi_t^* \omega_t| \le CK(t/2) \quad \text{on} \quad B((x,1),2\delta).
 \end{align} 

From \eqref{ap:03a}, we have for any $r \in (1/2,3/2)$ and $t<s$,
 \begin{align}\label{ap:05}
|t^{-1} \Psi_t^*\omega_t(x,r)-s^{-1} \Psi_s^*\omega_s(x,r)|=O(\bar K(t)).
 \end{align} 

For $y \in B(x,\delta) \cap S$, we denote the parallel transport of $\omega_t(\Psi(y,t))$ along $\Psi(y,\cdot)$ by $\bar \omega_t$, then it follows from \cite[6.2.1]{BK81} that
 \begin{align}\label{ap:06}
|\bar \omega_t(\Psi(y,s))-\omega_s(\Psi(y,s))| \le C \int_t^s \frac{K(u/2)}{u}\,du \le CK_0(t/2).
 \end{align} 

Similar to \eqref{ap:05} we have
 \begin{align}\label{ap:07}
|t^{-1} \Psi_t^*\omega_t(y,r)-s^{-1} \Psi_s^*\bar \omega_t(y,r)|=O(\bar K(t)).
 \end{align}

From \eqref{ap:05}, \eqref{ap:06} and \eqref{ap:07},
 \begin{align}\label{ap:08}
|t^{-1} \Psi_t^*\omega_t(y,r)-s^{-1} \Psi_s^*\omega_s(y,r)|=O(\bar K(t)).
 \end{align}

Combining \eqref{ap:04} and \eqref{ap:08} we have
 \begin{align*}
|d(t^{-1}\mathbb H_t\circ \Psi_t)-d(s^{-1}\mathbb H_s\circ \Psi_s)|= O(\bar K(t))
 \end{align*} 
and hence
\begin{align*}
|t^{-1}\mathbb H_t\circ \Psi_t-s^{-1}\mathbb H_s\circ \Psi_s | =O(\bar K(t)).
 \end{align*} 

Therefore, 
 \begin{align*}
|t^{-1}\mathbb H_t\circ \Psi_t-\mathbb H_{\infty}| =O(\bar K(t))
 \end{align*}
on $B((x,1),\delta)$. If we denote the standard Euclidean metric by $g_E$, then it follows from $\mathbb H_{\infty}^* g_E=g_E$ and $\mathbb H_{\infty}((x,1))=0$ that $\mathbb H_{\infty} \in \text{Iso}(\R^n)$. By redefining $\mathbb H_t$ by $\mathbb H^{-1}_{\infty} \circ \mathbb H_t$, we can assume that
 \begin{align}\label{ap:08a}
|t^{-1}\mathbb H_t\circ \Psi_t-I| =O(\bar K(t)).
 \end{align}

Now we take geodesics $\sigma_1,\cdots,\sigma_N$ with $\sigma_a(t)=\Psi(x_a,t)$ for $x_a \in S$ such that the collection $\{B(\sigma_a(\beta^j),\delta \beta^j)\}_{1\le a\le N,j\ge 1}$ cover $M \backslash K$ for some compact set $K$ and $\beta>1$. We denote $B(\sigma_a(\beta^j),\delta \beta^j)$ by $B_{a,j}$ and the corresponding harmonic coordinates by $\mathbb H_{a,j}$. Moreover, we can assume that $B_{a,j} \cap B_{b,k} \ne \emptyset$ if $k-j$ is uniformly bounded.

For $j \le k$ and $B_{a,j} \cap B_{b,k} \ne \emptyset$, it follows from \eqref{ap:08a} that $(\bar K(\beta^j))^{-1} \beta^{-j}| \mathbb H_{a,j}(\Psi(x,r\beta^j))-\mathbb H_{b,k}(\Psi(x,r\beta^k))|$ is uniformly bounded. Therefore, from the apriori estimates for harmonic functions, we have on $\mathbb H_{b,k}(B_{a,j} \cap B_{b,k})$,
 \begin{align*}
|\mathbb H_{a,j} \circ \mathbb H^{-1}_{b,k}-I|+\beta^j |\partial(\mathbb H_{a,j} \circ \mathbb H^{-1}_{b,k})|+\beta^{j(1+\alpha)} \|\partial(\mathbb H_{a,j} \circ \mathbb H^{-1}_{b,k})\|_{C^{\alpha}} \le C\bar K(\beta^j)\beta^j.
 \end{align*}

It is easy to check that by taking a partition of unity $\{\rho_{a,j}\}$ associated to the covering $\{B_{a,j}\}$ the map defined by
 \begin{align}\label{ap:08ax}
\Phi(x) \coloneqq \sum \rho_{a,j}(x) \mathbb H_{a,j}(x) 
 \end{align}
 satisfies
 \begin{align*}
\Phi(x)-I \in C^{2,\alpha}_{r\bar K}.
 \end{align*}
Therefore, the metric under this map
 \begin{align*}
(\Phi^{-1})^*g-g_E \in C^{1,\alpha}_{\bar K}.
 \end{align*}

\textbf{Case 2}: $\Gamma \ne 1$.

As above, there exists a diffeomorphism $\Psi$ from $(\R^n \backslash B(0,A))/\Gamma$ to the end $E \coloneqq M\backslash K$. Now we set the covering map from $\R^n \backslash B(0,A)$ to $(\R^n \backslash B(0,A))/\Gamma$ by $P$ and consider the metric $\tilde g \coloneqq (\Psi\circ P)^*g$. In particular, $\tilde g$ is invariant under the action of $\Gamma$. From Case 1, by enlarging $A$ if necessary, there exists a map $\Phi_1$ from $\R^n \backslash B(0,A)$ to $\R^n$ such that
 \begin{align}\label{ap:09}
\Phi_1-I \in C^{2,\alpha}_{r\bar K}.
 \end{align}

We define a new group action of $\Gamma$ on $\R^n \backslash B(0,A)$ by
 \begin{align*}
\tilde a=\Phi_1 \circ a \circ \Phi_1^{-1} \quad \text{for any } a \in \Gamma.
 \end{align*}

From \eqref{ap:09}, it is clear that for any $a \in \Gamma$,
 \begin{align*}
\tilde a-a \in C^{2,\alpha}_{r\bar K}.
 \end{align*}
 
Now we construct a map $\tilde \Psi$ from $\R^n \backslash B(0,A)$ to $\R^n \backslash B(0,A)$ defined as
 \begin{align*}
\tilde \Psi(x) \coloneqq \frac{1}{|\Gamma|}\sum_{a\in \Gamma} a^{-1} \circ \tilde a(x).
 \end{align*}

 From the definition of $\tilde \Psi$, it is easy to check that $\tilde \Psi -I \in C^{2,\alpha}_{r\bar K}$ and for any $a \in \Gamma$,
  \begin{align*}
a \circ \tilde \Psi=\tilde \Psi \circ \tilde a.
 \end{align*}
 
We define the map $\Phi \coloneqq \tilde \Psi\circ \Phi_1 $. Then it is clear that $\Phi-I \in C^{2,\alpha}_{r\bar K}$ and the group action of $a \in \Gamma$ is exactly the orthogonal transformation.

In sum, the proof of (i),(ii) and (iii) is complete and (iv) follows from the higher-order estimates of the harmonic coordinates. 
\end{proof}

\begin{rem}\label{rem:A01}
If the dimension $n=2$, then the tangent cone at infinity is a flat cone $C(S^1)$ where the length of the circle is $2\pi \beta$ for some constant $\beta \in (0,1]$. If $\beta=\frac{p}{q}$ for coprime integers $p$ and $q$, then the conclusions of Theorem \ref{T:ap01} also hold in this case.

Indeed, we consider the $p$-fold covering $\widetilde \C \coloneqq \{(r,\theta)\, \mid r>0, 0\le \theta \le 2p\pi\}/(r,0) \sim (r,2p\pi)$ of $\C \backslash \{0\}$, where the covering map $\pi$ is defined by
\begin{align*} 
\pi((r,\theta))=re^{i\theta}.
\end{align*}

By lifting the metric $g$ to a metric $g$ on $\widetilde \C\backslash B(0,A)$, it follows from the same arguments of Case $1$ of Theorem \ref{T:ap01} that there exists a map $\Phi: \widetilde \C \backslash B(0,A) \to \widetilde \C \backslash B(0,A)$ such that
\begin{align*} 
\Phi-I \in C^{2,\alpha}_{r\bar K}.
\end{align*}

Notice that in \eqref{ap:08ax}, we can take the coverings sufficiently refined so that the sum is taken in one fundamental domain. Since the Deck transformation group of the covering map is $\Z_q$, it follows from the same arguments of Case $2$ of Theorem \ref{T:ap01} that there exist a map $\Psi: C(S^1) \backslash B(0,A) \to M$ such that 
\begin{align*} 
g-(dt^2+t^2d\theta^2) \in C^{1,\alpha}_{\bar K}.
\end{align*}

On the other hand, if $\beta$ is irrational, then we only obtain the $C^0$-approximation
\begin{align*} 
g=dt^2+t^2d\theta^2+O(\bar K(t))
\end{align*}
for $t >A$ and $\theta \in [0,2\beta \pi)$, see \eqref{ap:03x}.
\end{rem}

\begin{rem}\label{rem:A02}
For the special case $K(t)=t^{-\ep}$, it follows from \emph{\cite{BKN89}} that $(M^n,g)$ has coordinates at infinity of order $\tau$ where $\tau=\ep$ if $n \ge 4$ or $\ep \ne 1$ and $\tau=\ep'$ for any $\ep'<1$ if $n=3$ and $\ep=1$.
\end{rem}

\section{Asymptotically flat gravitational instantons}
\label{app2}

In this appendix, we prove that any \hy $4$-manifold with \eqref{cond:AF} is a TALE manifold by slightly generalizing the work of Chen-Chen \cite{CC15}.

A key observation for \hy $4$-manifold is that for any geodesic loop $\gamma$, $\mathbf r(\gamma) \in \text{SU}(2)$ and hence for any vector $v$,
\begin{align} \label{eq:gra1}
|(\mathbf r(\gamma)-I)v|=|\mathbf r(\gamma)-I||v|.
\end{align}
It indicates that to estimate the rotational part, we only need to estimate its effect on one direction. By using this obeservation, Chen and Chen have proved in \cite{CC15} the following result.

\begin{thm}[Theorem $3.4$ of \cite{CC15}]\label{T:shol1}
For any complete \hy $4$-manifold $(M^4,g)$ with \eqref{cond:AF}, there exists a constant $C_H>1$ such that for any geodesic loop $\gamma$ based at $q$ with $r=r(q) \ge 2$ and $L(\gamma) \le C_H^{-1}r$,
\begin{align*}
|\mathbf r(\gamma) -I| \le \frac{C_H}{r}L(\gamma).
\end{align*}
\end{thm}

\begin{proof}
We consider the following Jacobi equation
\begin{align*} 
J''(t)=4t^{-2}K(t/2)J(t), \quad J(1)=0\quad\text{and}\quad J'(1)=1.
\end{align*}
As in the proof of Lemma \ref{L:vola}, we have for any $t \ge 1$
\begin{align}\label{eq:gra2}
t-1 \le J(t) \le J'(\infty)(t-1) \le Ct.
\end{align}    

Now we claim that there exists a constant $C_H>1$ such that for any geodesic loop $\gamma$ based at $q$ with $r=r(q) \ge 2$ and $L(\gamma) \le C_H^{-1}r$,
\begin{align}\label{eq:gra3}
|\mathbf r(\gamma) -I| \le \frac{J'(r)}{J(r)}L(\gamma).
\end{align}
From \eqref{eq:gra2}, \eqref{eq:gra3} implies the conclusion.

We argue by contradiction and $C_H$ will be determined in the proof. We assume that $q$ is a point such that $r \ge 2$ and there exists a geodesic loop $\gamma$ based at $q$ with $L(\gamma) \le C_H^{-1}r$, but 
\begin{align*} 
|\mathbf r(\gamma) -I| >\frac{J'(r)}{J(r)}L(\gamma).
\end{align*}

With $q$ and $\gamma$ fixed, we consider a function $l(x)$ which is defined as the length of the sliding $\gamma_x$ of $\gamma$ to $x$. Notice that on $B(q,\text{inj}(q))$, $l(x)$ is well-defined and smooth. It follows from Theorem \ref{T:hol} (i) that on $B(q,\text{inj}(q))$
\begin{align*} 
\na l(x)= (\mathbf r(x)-I)W(x),
\end{align*}
where $\mathbf r(x)=\mathbf r(\gamma_x)$ and $W(x)$ is the initial tangent vector of $\gamma_x$.

Now we can define an arc-length parametrized curve $\{\alpha(s): s \in [0,s_0)\}$ which is the gradient flow of $l$ starting from $q$. More precisely,
\begin{align*}
\alpha'(s)=-\frac{\na l(\alpha(s))}{|\na l(\alpha(s))|}=-\frac{\mathbf r(s)W(s)-W(s)}{|\mathbf r(s)W(s)-W(s)|},
\end{align*}
where $\mathbf r(s)$ is the rotational part of the sliding of $\gamma$ along $\alpha$ to $\alpha(s)$ and $W(s)$ is the corresponding initial tangent vector. Notice that even though $l$ is only well defined locally, its gradient flow $\alpha(s)$ is well defined. Indeed, on $B(q,\inj(q))$, $\na l$ is well defined, so we can locally define the gradient flow $\{\alpha(s): s \in[0,\ep]\}$ starting from $q$. Then based at $\alpha(\ep)$, we can locally extend the gradient flow as long as the length of the geodesic loop is small. In particular, the sliding of $\gamma$ and $l$ are defined along $\alpha$.

Now we reverse $\alpha$ by defining $\tilde{\alpha}(t)=\alpha(r-t)$ for any $t \in (t_0,r]$ where $t_0=r-s_0$ is the infimum of all $t$ such that $\tilde{\alpha}(t)$ is well defined. Therefore, we have
\begin{align*}
\tilde{\alpha}'(t)=\frac{\mathbf r(t)W(t)-W(t)}{|\mathbf r(t)W(t)-W(t)|}
\end{align*}
and if we set $l(t)=l(\tilde \alpha(t))$, then it follows from Theorem \ref{T:hol} (i) and \eqref{eq:gra1} that
\begin{align}\label{eq:gra6}
l'(t)=|\mathbf r(t)W(t)-W(t)|=|\mathbf r(t)-I|.
\end{align}

Now we claim that for any $\max\{t_0,1\} <t \le r$, $l(t) < t/2$. Otherwise, we set $t_1$ to be the largest number in $(\max\{t_0,1\},r]$ such that $l(t) = t/2$. Then it follows from \eqref{eq:gra6} and Theorem \ref{T:hol} (ii) that for any $t \in[t_1,r]$,
\begin{align*}
l''(t) \le \max_{\gamma_t} |Rm|l(t) \le 4t^{-2}K(t/2)l(t).
\end{align*}
Therefore, $(l'J-J'l)' =l''J-J'' l \le 0$. By our assumption $l'(r)J(r)-J'(r)l(r)>0$, we conclude that for any $t \in [t_1,r]$,
\begin{align} \label{E306}
l'(t)J(t)-J'(t)l(t)>0.
\end{align}

By integration,
\begin{align*}
l(t_1) <\frac{l(r)}{J(r)}J(t_1) \le \frac{C_H^{-1}r}{J(r)}J(t_1)=C_H^{-1}\frac{r}{J(r)}\frac{J(t_1)}{t_1}t_1 \le t_1/2
\end{align*}
if $C_H$ is sufficiently large. Therefore, the claim holds.

Now we prove that $t_0 \ge 1$. Indeed, if $t_0<1$, then by the same reason we get from \eqref{E306} that
\begin{align*} 
l'(1)J(1)-J'(1)l(1)>0.
\end{align*}
which contradicts the definition of $J$. In sum, we have proved that for any $t \in (t_0,r]$, $l(t) < t/2$. However, since $r(\tilde \alpha(t)) \ge 1$, by a compactness argument, the curve $\tilde \alpha(t)$ can be extended backwards past $t_0$. This contradicts the definition of $t_0$. Therefore, the proof of \eqref{eq:gra3} is complete.
\end{proof}

Theorem \ref{T:shol1} shows that \eqref{cond:HC'} is satisfied for \hy $4$-manifold $(M^4,g)$ with \eqref{cond:AF}. 

\begin{thm}\label{T:gra1}
Any complete \hy $4$-manifold $(M^4,g)$ with \eqref{cond:AF} is a TALE manifold.
\end{thm}

\begin{proof}
We only need to prove that \eqref{cond:SHC} holds.

From Theorem \ref{T:nil}, there exists a fibration from the end of $M$ to $C(S(\infty))$ and the fiber $F$ is a nilmanifold. If dim $F=1$, then we obtain a circle bundle and the conclusion follows from Theorem \ref{T:circle}. Therefore, we assume that dim $F=2$ or $3$. In particular, the fundamental group $\pi_1(F)$ contains more than one generator.

From \eqref{cond:HC'}, there exists a small constant $\kappa>0$ such that for any $x$ with $r=r(x)$ sufficiently large, $T_1$ is a $\theta$-translational subset, where $T_1$ is defined in \eqref{eq:Tr} and $\theta<\frac{1}{100}$. In particular, for any $x$ we obtain a standard short basis $\{c_1,\cdots,c_m\}$ of $\Gamma(x,\kappa r)$, where $m \le 4$. We claim that $m \ne 1$. Indeed, if $m=1$, it implies that $\Gamma(x,\kappa r)$ is generated by a single element. In particular, $\pi_1(F)$ is generated by one element, where $F$ is the fiber through $x$. This is a contradiction.

By the construction of the standard short basis, $[c_1,c_j]=0$ for any $2\le j\le m$. If we set $\mathbf m(c_1)=\mathbf r(c_1)x+\mathbf t(c_1)=Ax+a$ and $\mathbf m(c_j)=\mathbf r(c_j)x+\mathbf t(c_j)=Bx+b$, then it follows from \cite[Theorem $2.4.1$ (ii),(iii)]{BK81} that
\begin{align}\label{eq:gra01}
|[A,B]-I| \le Cr^{-2}K(r/2)|a||b|
\end{align}
and
\begin{align*}
|\mathbf t([\mathbf m(c_1),\mathbf m(c_i)])|=|(A-I)b-([A,B]-I)b+A(I-B)A^{-1}a| \le Cr^{-2}K(r/2)|a||b|(|a|+|b|).
\end{align*}

Therefore, we have
\begin{align}\label{eq:gra03}
|(A-I)b-A(B-I)A^{-1}a| \le Cr^{-2}K(r/2)|a||b|(|a|+|b|)\le Cr^{-1}K(r/2)|a||b|.
\end{align}

Now we fix a nonincreasing positive function $\ep(r)$ such that if $r \to \infty$, $\ep(r) \to 0$ and $K(r/2)/\ep(r) \to 0$. We claim that for $r$ sufficiently large,
\begin{align}\label{eq:gra04}
|A-I| \le r^{-1}\ep(r)|a|.
\end{align}

Otherwise, we assume that there exists a sequence $x_i$ with $r_i=r(x_i) \to \infty$ such that
\begin{align}\label{eq:gra05}
|A_i-I| > r_i^{-1}\ep(r_i)|a_i|
\end{align}
where we add the subscript $i$ to denote the corresponding elements at $x_i$. Therefore, it is clear from \eqref{eq:gra03} that
\begin{align}\label{eq:gra06}
\lim_{i \to \infty} \left| \frac{(A_i-I)b_i}{|A_i-I||b_i|}-\frac{(B_i-I)a_i}{|B_i-I||a_i|} \right|=0.
\end{align}

Since $A_i-I$ and $B_i-I$ are small, we set $A_i=\exp(\tilde A_i)$ and $B_i=\exp(\tilde B_i)$, where $\tilde A_i,\tilde B_i \in \mathfrak{su}(2)$. By taking a subsequence if necessary, we assume that
\begin{align*}
\frac{\tilde A_i}{|\tilde A_i|} \to \tilde A_{\infty}, \quad \frac{\tilde B_i}{|\tilde B_i|} \to \tilde B_{\infty}, \quad \frac{a_i}{|a_i|} \to a_{\infty} \quad \text{and} \quad \frac{b_i}{|b_i|} \to b_{\infty}.
\end{align*}
Therefore, it follows from \eqref{eq:gra06} that
\begin{align}\label{eq:gra07}
\tilde A_{\infty}b_{\infty}=\tilde B_{\infty}a_{\infty}.
\end{align}

From our choice of the standard short basis, it is clear that $a_{\infty}$ and $b_{\infty}$ are not collinear. Therefore, it follows from \eqref{eq:gra07} that $\tilde A_{\infty}$ and $\tilde B_{\infty}$ are not collinear.

Since the Lie bracket in $\mathfrak{su}(2)$ can be regarded as the cross-product in $\R^3$, it follows from \eqref{eq:gra01} that
\begin{align*}
|\tilde A_i| |\tilde B_i| \le C |[\tilde A_i,\tilde B_i]| \le Cr_i^{-2}K(r_i/2)|a_i||b_i|,
\end{align*}
which contradicts \eqref{eq:gra05} if $i$ is sufficiently large. Therefore, the claim \eqref{eq:gra04} holds and by \eqref{eq:gra03} we have
\begin{align*}
|B-I| \le 2r^{-1}\ep(r)|b|.
\end{align*}

In sum, we have proved that for any $1 \le j \le m$,
\begin{align*}
|\mathbf r_{c_j}-I| \le C|c_j|r^{-1}\ep(r).
\end{align*}
Since $\{c_1,c_2,\cdots,c_m\}$ is a standard short basis, it follows from the same argument of Lemma \ref{L:tr} that for any $a \in \Gamma(x,\kappa r)$
\begin{align*}
|\mathbf r_{a}-I| \le C|a|r^{-1}\ep(r).
\end{align*}

Therefore, \eqref{cond:SHC} holds.
\end{proof}

It follows from Theorem \ref{T:003} that any \hy $4$-manifold of ALE, ALF or ALH type has faster-than-quadratic curvature decay and hence can be completely classified. It is not clear if there are \hy $4$-manifolds of ALG type which do not have faster-than-quadratic curvature decay.

\vskip10pt

Xiuxiong Chen, Stony Brook University, Stony Brook, NY 11794, USA; School of Mathematics, University of Science and Technology of China, Hefei, Anhui, China; xiu@math.sunysb.edu\\

Yu Li, Simons Center for Geometry and Physics, Stony Brook University, Stony Brook, NY 11794, USA; yu.li.4@stonybrook.edu.\\


\begin{thebibliography}{99}
\bibitem{Ab85} U. Abresch, \emph{Lower curvature bounds, Toponogov's theorem, and bounded topology}, Ann. Sci. Ecole Norm. Sup. (4) 18 (1985), no. 4, 651–670.


\bibitem{AH88} M. Atiyah, N. J. Hitchin, \emph{The geometry and dynamics of magnetic monopoles}, M. B. Porter Lectures, Princeton University Press, Princeton,
NJ, 1988. viii+134 pp.

\bibitem{APS75a} M. Atiyah, V . Patodi,  I. Singer, \emph{Spectral asymmetry and riemannian geometry I}, Math. Proc. Cambridge Philos. Soc, 77 (1975), 43-69.

\bibitem{APS75b} M. Atiyah, V . Patodi,  I. Singer, \emph{Spectral asymmetry and riemannian geometry II}, Math. Proc. Cambridge Philos. Soc, 78 (1975), 405-432.


\bibitem{BBI01} D. Burago, Y. Burago, S. Ivanov, \emph{A course in metric geometry}, Graduate Studies in Mathematics, vol. 33, American Mathematical Society, Providence, RI, 2001.


\bibitem{BiMi11} O. Biquard, V. Minerbe, \emph{A Kummer Construction for Gravitational Instantons}, Communications in Mathematical Physics, December 2011, Volume 308, Issue 3, pp 773-794.

\bibitem{BK81} P. Buser, H. Karcher, \emph{Gromov's almost flat manifolds}, Asterisque 81 (1981), 1-148. 

\bibitem{BKN89} S. Bando, A. Kasue, H. Nakajima, \emph{On a construction of coordinates at infinity on manifolds with fast curvature decay and maximal volume growth}, Inventiones mathematicae, June 1989, Volume 97, Issue 2, pp 313-349.


\bibitem{Bru57} F. Fa\`a di Bruno, \emph{Note sur une nouvelle formule de calcul diff\'erentiel}, Quarterly J. Pure Appl. Math. 1 (1857) 359-360.

\bibitem{CC15} G. Chen, X. X. Chen, \emph{Gravitational instantons with faster than quadratic curvature decay (I)}, preprint, http://arxiv.org/abs/1505.01790.

\bibitem{CC17} G. Chen, X. X. Chen, \emph{Gravitational instantons with faster than quadratic curvature decay (II)}, J. reine angew. Math., 2019(756), 259-284. DOI 10.1515/crelle-2017-0026.

\bibitem{CC16} G. Chen, X. X. Chen, \emph{Gravitational instantons with faster than quadratic curvature decay (III)}, Math. Ann. (2020). https://doi.org/10.1007/s00208-020-01984-9.


\bibitem{CH05} S. Cherkis, N. Hitchin, \emph{Gravitational instantons of type $D_k$}, Commun. Math. Phys. 260(2), 299–317 (2005).

\bibitem{CFG92} J. Cheeger, K. Fukaya, M. Gromov, \emph{Nilpotent Structures and Invariant Metrics on Collapsed Manifolds}, Journal of the American Mathematical Society Vol. 5, No. 2 (Apr., 1992), pp. 327-372.


\bibitem{CG90} J. Cheeger, M. Gromov, \emph{Collapsing Riemannian manifolds while keeping their curvature bounded. part II}, J. Differential Geom. 32 (1990), 269-298.


\bibitem{CK99} S. Cherkis, A. Kapustin, \emph{Singular monopoles and gravitational instantons}, Commun. Math. Phys. 203(3), 713–728 (1999).




\bibitem{DaiWei07} X. Dai, G. Wei, \emph{Hitchin-Thorpe inequality for noncompact Einstein $4$-manifolds}, Advances in Mathematics, Volume 214, Issue 2, 1 October 2007, Pages 551-570.

\bibitem{Dan93} A. S. Dancer, \emph{Dihedral singularities and gravitational instantons}, J. Geom. Phys. 12(2), 77–91 (1993).

\bibitem{DeRham51} G. de Rham, \emph{Complexes \`a automorphismes et hom\'eomorphie diff\'erentiable}, Ann. Inst. Fourier Grenoble 2 (1950), 51–67 (1951).

\bibitem{EGH80} T. Eguchi, P. B. Gilkey, A. J. Hanson, \emph{Gravitation, gauge theories and differential geometry}, Physics Reports, 66 (1980) 213-393.


\bibitem{Fu87} K. Fukaya, \emph{Collapsing Riemannian manifolds to ones of lower dimensions}, J. Differential Geom. Volume 25, Number 1 (1987), 139-156.

\bibitem{Fu88} K. Fukaya, \emph{A boundary of the set of the Riemannian manifolds with bounded curvatures and diameters}, J. Differential Geom. Volume 28, Number 1 (1988), 1-21.

\bibitem{Fu89} K. Fukaya, \emph{Collapsing Riemannian manifolds to ones of lower dimension II}, J. Math. Soc. Japan 41 (1989), 333-356.

\bibitem{GH78} G. W. Gibbons, S. W. Hawking, \emph{Classification of gravitational instanton symmetries}, Comm. Math. Phys. 66
(1979), 291-310.

\bibitem{Gr06} M. Gromov, \emph{Metric Structures for Riemannian and Non-Riemannian Spaces}, Modern Birkh\"auser Classics, Birkh\"auser.

\bibitem{GT01} D. Gilbarg and N.S. Trudinger, \emph{Elliptic partial differential equations of second order}, Springer.




\bibitem{Hei12} H. J. Hein, \emph{Gravitational instantons from rational elliptic surfaces}, J. Amer. Math. Soc. 25 (2012), no.2, 355-393.

\bibitem{Hit74} N. Hitchin, \emph{Compact four-dimensional Einstein manifolds}, J. Differential Geometry 9 (1974), 435-441.

\bibitem{Hi84} N. Hitchin, \emph{Twistor construction of Einstein metrics}, Global Riemannian geometry (Durham, 1983), 115-125, Ellis Horwood Ser. Math. Appl., Horwood, Chichester, 1984.


\bibitem{Jo83} J. Jost, \emph{harmonic mappings between riemannian manifolds}, Proc. Centre for Math. Analysis, Australian Nat. Univ., 4 (1983).

\bibitem{Ka88} A. Kasue, \emph{A compactification of a manifold with asymptotically nonnegative curvature}, Annales scientifiques de l'\'Ecole Normale Sup\'erieure, S\'erie 4 (1988) Volume: 21, Issue: 4, page 593-622.

\bibitem{Ka89} A. Kasue, \emph{A Convergence Theorem for Riemannian Manifolds and Some Applications}, Nagoya Math. J. 114, p. 21-51 (1989).


\bibitem{Kirby89} R. C. Kirby, \emph{The topology of 4-manifolds}, Lecture Notes in Mathematics. 1374. Berlin: Springer-Verlag. ISBN 3-540-51148-2, (1989).

\bibitem{Kro89a} P. B. Kronheimer, \emph{The construction of ALE spaces as \hy quotients}, J. Differential Geom. 29 (1989), no.3, 665-683. MR0992334.

\bibitem{Kro89b} P. B. Kronheimer, \emph{A Torelli-type theorem for gravitational instantons}, J. Differential Geom. 29 (1989), no.3, 685-697. MR0992335.



\bibitem{Li12} P. Li, \emph{Geometric Analysis}, Cambridge Studies in Advanced Mathematics (Book 134), Cambridge University Press.




\bibitem{LS00} J. Lott, Z. Shen, \emph{Manifolds with quadratic curvature decay and slow volume growth}, Ann. Sci. Ecole Norm. Sup. (4) 33:2 (2000), 275–290.

\bibitem{LV16} M. T. Lock, J. A. Viaclovsky, \emph{Quotient singularities, eta invariants, and self-dual metrics}, Geom. Topol. 20(2016), 1773–1806.


\bibitem{Mil61} J. Milnor, \emph{Two complexes which are homeomorphic but combinatorially distinct}, Ann. of Math. (2) 74 (1961).


\bibitem{Mi09} V. Minerbe, \emph{A mass for ALF manifold}, Comm. Math. Phys. 289 (2009), no. 3, 925-955.

\bibitem{Mi092} V. Minerbe, \emph{Weighted Sobolev Inequalities and Ricci Flat Manifolds}, Geometric and Functional Analysis, February 2009, Volume 18, Issue 5, pp 1696-1749.

\bibitem{Mi10} V. Minerbe, \emph{On the asymptotic geometry of gravitational instantons}, Annales scientifiques de l'École Normale Supérieure, Série 4 : Volume 43 (2010) no. 6 , p. 883-924.

\bibitem{Mi11} V. Minerbe, \emph{Rigidity for Multi-Taub-NUT metrics}, J. reine angew. Math. 656 (2011), 47-58.

\bibitem{MNO05} Y. Mashiko, K. Nagano, K. Otsuka, \emph{The asymptotic cones of manifolds of roughly non-negative radial curvature}, J. Math. Soc. Japan 57(2005), no.1, 55–68.

\bibitem{Mon87} J. M. Montesinos, \emph{Classical Tessellations and Three-Manifolds}, in: Universitext, Springer, Berlin, 1987.



\bibitem{Na90} H. Nakajima, \emph{Self-duality of ALE Ricci-flat $4$-manifolds and positive mass theorem. In Recent topics in differential and analytic geometry}, volume 18 of Adv. Stud. Pure Math., pages 385-396. Academic Press, Boston, MA, 1990.


\bibitem{Pa81} D. N. Page, \emph{A periodic but nonstationary gravitational instanton}, Phys. Lett. B 100 (1981), no.4, 313-315.

\bibitem{Pe1} G. Perelman, \emph{The entropy formula for the Ricci flow and its geometric applications}, arXiv:math/0211159.

\bibitem{Pe2} G. Perelman, \emph{Ricci flow with surgery on three-manifolds}, arXiv:math/0303109.

\bibitem{Pe3} G. Perelman, \emph{Finite extinction time for the solutions to the Ricci flow on certain three-manifolds}, arXiv:math/0307245.

\bibitem{Pe06} P. Petersen, \emph{Riemannian Geometry}, Graduate Texts in Mathematics (Book 171), Springer; 2nd edition.

\bibitem{PT01} A. Petrunin, W. Tuschmann, \emph{Asymptotical flatness and cone structure at infinity}. Math. Ann. 321(4), 775–788 (2001).

\bibitem{Rong96} X. Rong, \emph{On the Fundamental Groups of Manifolds of Positive Sectional Curvature}, Ann. of Math.143, p. 397-411 (1996).



\bibitem{Shi89} W. X. Shi, \emph{Deforming the metric on complete Riemannian manifolds}. J. Diff. Geom., 30:223-301, 1989.


\bibitem{SZ12} A. Szczepanski, \emph{Eta invariants for flat manifolds}, Annals of Global Analysis and Geometry 41 (2012),125-138.

\bibitem{Ta71} K. I. Tahara, \emph{On the finite subgroups of $GL(3,\mathbb Z)$}, Nagoya Math. J., 41:169-209, 1971.

\bibitem{Tho69} J. A. Thorpe, \emph{Some remarks on the Gauss-Bonnet integral}, J. Math. Mech. 18(1969), 779-786.


\bibitem{Tu88} V.G. Turaev, \emph{Towards the topological classification of geometric 3-manifolds}, in Topology and Geometry-Rohlin Seminar, Lecture Notes in Math. 1346, Springer, Berlin-Heidelberg-New York, 291-323 (1988).


\bibitem{War66} F. Warner, \emph{Extension of the Rauch comparison theorem to submanifolds}, Transactions of the American Mathematical Society, (1966), pp. 341–356.


\end{thebibliography}
\end{document}